\documentclass[reqno]{amsart}
\usepackage{hyperref}
\usepackage{amssymb}
\usepackage{cases}
\usepackage{graphicx}
\usepackage{fancyhdr}
\usepackage{float}
\usepackage{caption}
\usepackage{graphicx,lipsum}
\usepackage[utf8]{inputenc}
\usepackage[T1]{fontenc}
\usepackage[english]{babel}
\usepackage{array,multirow,makecell}
\usepackage{hyperref}
\hypersetup{colorlinks,%
            citecolor=red,%
            filecolor=yellow,%
            linkcolor=blue,%
            urlcolor=black}
\usepackage{array,multirow,makecell}
\usepackage[toc,page]{appendix}
\usepackage{ulem} 
\usepackage{comment} 

\usepackage{titletoc}
\setcellgapes{1pt}
\makegapedcells
\newcolumntype{R}[1]{>{\raggedleft\arraybackslash }b{#1}}
\newcolumntype{L}[1]{>{\raggedright\arraybackslash }b{#1}}
\newcolumntype{C}[1]{>{\centering\arraybackslash }b{#1}}
\usepackage{color}
\definecolor{dkgreen}{rgb}{0,0.6,0}
\definecolor{gray}{rgb}{0.5,0.5,0.5}
\definecolor{mauve}{rgb}{0.58,0,0.82}
\numberwithin{equation}{section}
\newtheorem{theorem}{Theorem}[section]
\newtheorem{lemma}[theorem]{Lemma}
\newtheorem{proposition}[theorem]{Proposition}
\newtheorem{definition}[theorem]{Definition}

\newtheorem{remark}[theorem]{Remark}

\usepackage[toc,page]{appendix}
\begin{document}
\title
[\hfil Analysis of a spatio-temporal advection-diffusion APC model]{Analysis of a spatio-temporal advection-diffusion model for human behaviors during a catastrophic event}
\author[K. Khalil et al.]
{K. Khalil$^{1,*}$, V. Lanza$^1$, D. Manceau$^1$, M.A. Aziz-Alaoui$^1$, D. Provitolo$^2$}
  
\address{K. Khalil,\newline
LMAH, University of Le Havre Normandie, FR-CNRS-3335, ISCN, Le Havre 76600, France.}
\email{kamal.khalil.00@gmail.com}

\address{V. Lanza,\newline
LMAH, University of Le Havre Normandie, FR-CNRS-3335, ISCN, Le Havre 76600, France.}
\email{valentina.lanza@univ-lehavre.fr }

\address{D. Manceau,\newline
LMAH, University of Le Havre Normandie, FR-CNRS-3335, ISCN, Le Havre 76600, France.}
\email{david.manceau@univ-lehavre.fr}

\address{M-A. Alaoui,\newline
LMAH, University of Le Havre Normandie, FR-CNRS-3335, ISCN, Le Havre 76600, France.}
\email{aziz.alaoui@univ-lehavre.fr}

\address{D. Provitolo\newline
Université Côte d'Azur, CNRS, Observatoire de la Côte d'Azur, IRD, Géoazur, UMR 7329, Valbonne, France.}
\email{Damienne.provitolo@geoazur.unice.fr}

\thanks{$^*$Corresponding author: K. Khalil; \texttt{kamal.khalil@univ-lehavre.fr}}
{\renewcommand{\thefootnote}{} \footnote{
$^{1}$LMAH, University of Le Havre Normandie, FR-CNRS-3335, ISCN, Le Havre 76600, France.\\
$^{2}$Université Côte d'Azur, CNRS, Observatoire de la Côte d'Azur, IRD, Géoazur, UMR 7329, Valbonne, France.}}

\begin{abstract}
In this work, using the theory of first-order macroscopic crowd models, we introduce a compartmental advection-diffusion model, describing the spatio-temporal dynamics of a population in different human behaviors (alert, panic and control) during a catastrophic event. For this model, we prove the local existence, uniqueness and regularity of a solution, as well as the positivity and $L^1$--boundedness of this solution. Then, in order to study the spatio-temporal propagation of these behavioral reactions within a population during a catastrophic event, we present several numerical simulations for different evacuation scenarios.

\end{abstract}
\maketitle
\tableofcontents

\section{Introduction}
In the last decades the world  has known some radical changes at almost all levels such as technological developments, climatic changes and human evolution. Due to these factors, populations (in both, developed or undeveloped countries)  are aggressively facing many natural disasters (tsunamis, earthquakes), technological events and terrorist attacks. In particular situations of sudden, unexpected and without alert disasters require high security measures and strategies in order to predict and manage the movement and the behavior of a crowd. During a catastrophe, people may experience many different behaviors, but there is still few information about the dynamics and the succession of such behaviors during the event, see \cite{Crocq,Pi_Re,Re_Ru}. Thus, for the development of an efficient disaster management strategy, it becomes necessary not only to take into account the disaster features but also the different psychological human behaviors during the disaster event. 

Recently, several pedestrians crowd models have been developed. Their main objective is to predict the movements of a crowd in different environmental situations. Mathematical models of human crowds are mainly divided into two categories, namely, microscopic models and macroscopic ones, see the recent survey \cite{Bellomo3} and papers \cite{Masmoudi,Bellomo,Bellomo2} for more details. In the microscopic approach, individuals are treated separately as particles and the evolution is determined using Newton's second law and by considering physical and social forces that describe the interaction among individuals as well as their interactions with the physical surrounding (for more details we refer to the works of Helbing described in \cite{MauryBook}). The macroscopic approach, that we adopt in this paper, considers a crowd as a whole quantity, without recognizing individual differences, and it is therefore more suitable to the study of the movement of an extremely large number of pedestrians. 
In particular, first-order macroscopic models introduced by Hughes \cite{Hughes2002} (see also \cite{Coscia})  are based on a mass conservation equation and a density-velocity closure equation with suitable boundary conditions. Furthermore, several models are devoted to study the dynamics of multiple pedestrian species in the context of macroscopic first-order systems (see \cite{Burger,Burger2,JMCoron,Pietschmann,DiFrancesco,Gomes,Hughes2002,Rosini,Sim_Lan_Hug} and references therein). In \cite{Hughes2002} Hughes studied crowds with large density of multiple pedestrian classes with different walking characteristics and destinations (identified by  the index $i$). The system reads as
\begin{equation}\label{eq:mainSystem Intro2}
            \left\{
                 \begin{aligned}
\partial_t \rho_i + \nabla \cdot  q_{i}(\rho) &=0 &\quad &\text{ in } [0,T)\times\Omega, \\
  q_{i}\cdot n & =q^{0}_i \cdot n &\quad  &\text{ in } [0,T)\times\partial \Omega ,\\
                    \rho_{i} (0)&=\rho_i^{0}&\quad  &\text{ in } \Omega, &
                   \end{aligned}\right.
\end{equation}
for $i=1,\ldots,N \, (N\geq 2)$, where $ \Omega \subset \mathbb{R}^{2} $ is a bounded domain with smooth boundary $ \partial \Omega $ and $ q_{i}(\rho)= \rho_i v(\rho)\nu_i $. The velocity is defined as $v(\rho)=A-B \tilde{\rho}$, thus it is linear with respect to the total population $ \tilde{\rho}$ and is the same for all populations. On the contrary, each population can have a different direction of the movement  $\nu_i$. Finally, for each population $i$, $q_i^0$ is the outflow from the boundary in the direction of the normal vector $n$ and $ \rho^{0}_{i} $ is the initial data. Moreover, in \cite{Burger} authors studied a nonlinear drift-diffusion model with in-outflow boundary conditions for the transport of particles. Notice that these systems are consisting of conservative equations  (with no reactions terms). A non-conservative system is proposed in \cite{Marino} but the authors consider only one population species and Neumann homogeneous boundary conditions, namely 
\begin{equation}\label{eq:mainSystem Intro4}
            \left\{
                 \begin{aligned}
\partial_t \rho + \nabla \cdot  q(\rho) &=\alpha(t,x)f(\rho)-\beta(t,x)\rho &\quad &\text{ in } [0,T)\times\Omega, \\
  \nabla \rho\cdot n & =0 &\quad  &\text{ in } [0,T)\times\partial \Omega ,\\
                    \rho (0)&=\rho^{0}&\quad  &\text{ in } \Omega , &
                   \end{aligned}\right.
\end{equation}
where $ q(\rho)= -\nabla \rho +f(\rho) \nabla V(\rho) $ where $f(\rho)=\rho(1-\rho)$ and  $V:\mathbb{R}^{n}\longrightarrow \mathbb{R}$ is a potential. 
 In all these papers, either there is no mention about the behaviors of the pedestrians or all the pedestrians have the same emotional state (mainly panic).

In the recent years the RCP (Reflex-Panic-Control) \cite{ijbc} and the APC (Alert-Panic-Control)  \cite{LanzaFede} models have been proposed in order to describe the evolution in time of human behaviors during a catastrophe. They both consist in systems of nonlinear ODEs and have been devised following the structure of the  compartmental models in mathematical epidemiology.  In \cite{Lanza_et_al.} the spatial dynamics has been integrated in the APC model, by considering the space as a discrete variable. In \cite{Cantin_Aziz} the first system of reaction-diffusion equations describing a population with several behaviors has been proposed.


The aim of the present paper is to introduce a spatio-temporal macroscopic first-order non-conservative pedestrians model describing the evolution of a population in a sudden, unexpected and without warning signs disaster. For this purpose, starting from the nonlinear ODE APC model proposed in \cite{Lanza_et_al.,LanzaFede}, we introduce a non-conservative first-order macroscopic model to describe the spatio-temporal dynamics of a population exhibiting different behavioral states. Our model reads as
\begin{equation}\label{eq:ourmodel}
            \left\{
                 \begin{aligned}
\partial_t \rho_i + \nabla \cdot  q_{i}(\rho) &=f_i(\rho_1,\dots,\rho_5) &\quad &\text{ in } [0,T)\times\Omega, \\
  q_{i}\cdot n & =q^{0}_i \cdot n &\quad  &\text{ in } [0,T)\times\partial \Omega ,\\
                    \rho_{i} (0)&=\rho_i^{0}&\quad  &\text{ in } \Omega, &
                   \end{aligned}\right.
\end{equation}
where, for each $i=1,\ldots,5,$ $\rho_i$  is the density of population representing a specific human behavior, $
q_i:=-d_i\nabla \rho_i+\rho_i \vec{v}_i(\rho)$ is the corresponding flux, $f_i$ is a given nonlinear coupling reaction term and $\rho_i^0$ is the initial population density. \\
Therefore, we establish a mathematical analysis of the model \eqref{eq:ourmodel}. By virtue of the abstract boundary evolution equations and the theory of semigroups of bounded linear operators (see \cite{Amann2,Desh,Greiner}), we prove the local existence, uniqueness and regularity of a solution of this system. Moreover, using the positively invariant regions approach due to Martin et al., see \cite{Hirsch-Smith, Martin-Smith, Martin}, we provide sufficient conditions on the parameters of our model ensuring the positivity. We also provide different numerical simulations for several scenarios of evacuation of a population in an emergency situation. 

The organization of this paper is as follows.  In Section \ref{sec:2}, we briefly present the equations of the APC model, we recall the structure of a first-order macroscopic pedestrians model and we introduce our advection-diffusion pedestrians APC model \eqref{eq:ourmodel}. Section \ref{sec:3} is devoted to the mathematical analysis of the model. Finally, Section \ref{sec:4} presents numerical results about different scenarios of evacuation.

\section[A spatio-temporal advection-diffusion APC model]{A spatio-temporal advection-diffusion model for human behaviors during a catastrophic event}\label{sec:2}

\subsection{The temporal model of the human behaviors of a population during a catastrophic event}\label{Section0}

In this section, we briefly present a model describing the evolution of a population during a sudden, rapid and unpredictable catastrophic event. The model describes the evolution of different human behaviors during a disaster, see \cite{Lanza_et_al., LanzaFede}. Depending on the emotional charges and their regulation, the different human reactions of a population in an emergency situation have been subdivided into three main categories, namely, alert, panic and control behaviors. More particularly, pedestrians in control are defined as individuals capable of managing enormous emotional charges in a dangerous situation, whereas pedestrians in a panic state are those who are incapable of regulating the emotional charges during a catastrophic event. Finally, alert pedestrians are those who have not yet understood the situation and are assimilating the event they are going to live. Each pedestrian experiences an alert behavior as soon as he is impacted by the catastrophe. Thus, the APC  model considers the time evolution of the following five variables:
\begin{itemize}
\item the density of individuals in an alert state $\rho_1(t)$,
\item the density of individuals that exhibit panic behaviors $\rho_2(t)$, 
\item the density of individuals in a state of control $\rho_3(t)$,
\item the density of individuals in a daily behavior before the catastrophe $\rho_4(t)$,
\item the density of individuals in a behavior of everyday life after the disaster $\rho_5(t)$, 
\item the density of individuals who die during the disaster $\rho_6(t)$. 
\end{itemize}
The corresponding model is given by the following nonlinear ODE system that matches with the classical compartmental SIR models ($t\geq 0$):
\begin{equation}\label{eq:main1System APC2 ODEs}
            \left\{
                 \begin{array}{l l l l }
                  \rho_1' = &  -(b_1+b_2+\delta_1) \rho_1 + \gamma(t) q+b_3 \rho_3 +b_4  \rho_2-\mathcal{F}(\rho_1, \rho_3) - \mathcal{G}(\rho_1, \rho_2),\\
                  \rho_2' = &-(b_4 + c_1 + \delta_2)\rho_2+ b_2 \rho_1 +c_2 \rho_3 +\mathcal{G}(\rho_1, \rho_2)- \mathcal{H}(\rho_2,\rho_3),\\
                  \rho_3' = & -(b_3 +c_2+\delta_3)\rho_3  +b_1 \rho_1 +c_1 \rho_5  -\phi(t)\rho_3 + \mathcal{F}(\rho_1,\rho_3)+\mathcal{H}(\rho_2,\rho_3), \\
                  \rho_4' = &  - \gamma(t) q ,\\
                  \rho_5' =& \phi(t) \rho_3, \\
                  \rho_6' =& \delta_1 \rho_1 + \delta_2 \rho_2 +\delta_3 \rho_3,\\
                 \end{array}\right.
      \end{equation} 
with the initial condition $(\rho_1,\rho_2,\rho_3,\rho_4,\rho_5,\rho_6)(0)=(0,0,0,1,0,0)$, since the population is supposed to be in a daily behavior before the onset of the disaster.
 \begin{figure}
\begin{center}
\begin{tabular}{c}
\includegraphics[width=8cm,height=6cm]{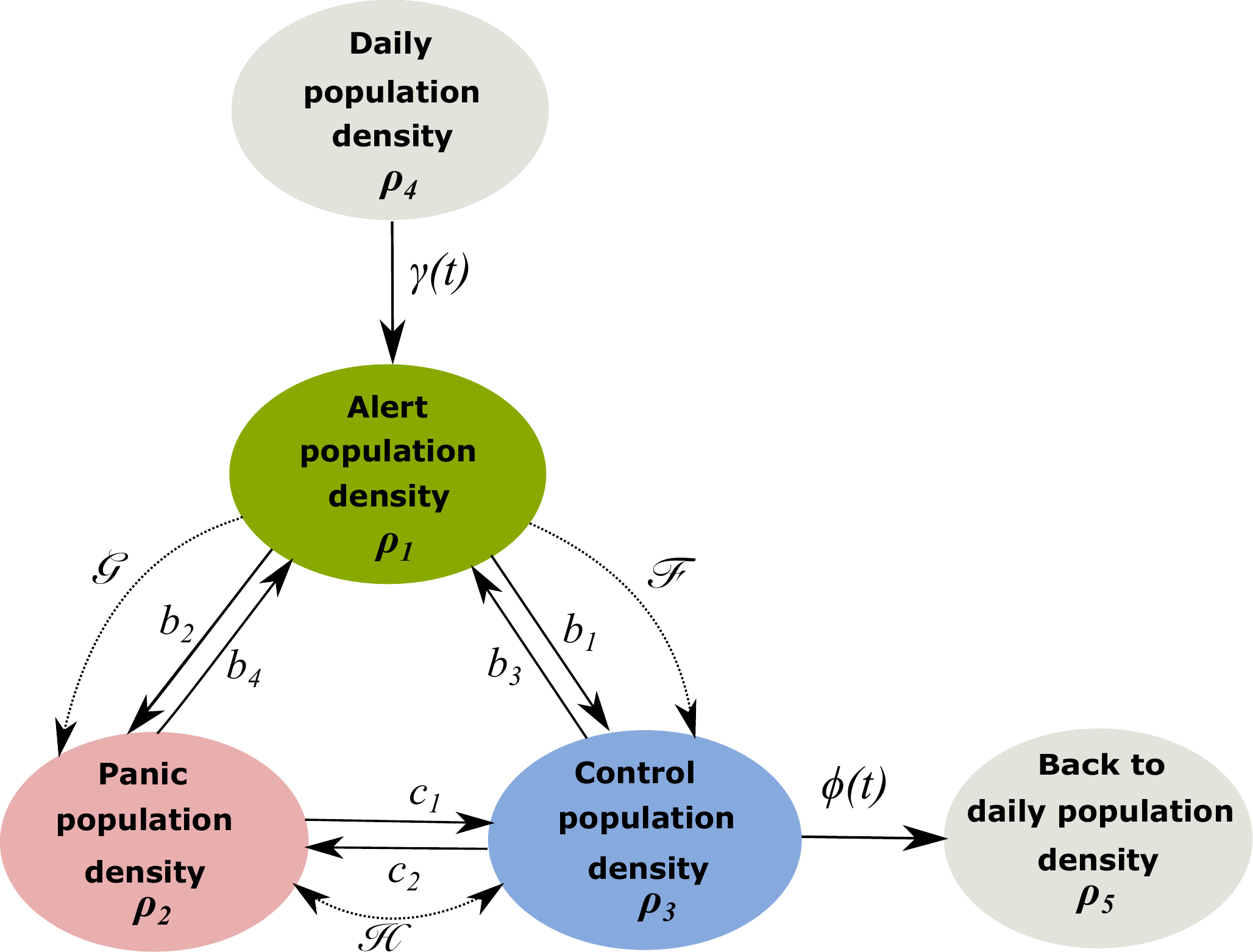}\\
\end{tabular}
\caption{The transfer diagram of the APC model. The arrows indicate the transitions among the compartments.}
\label{APC-Kamal00}
\end{center}
\end{figure}   
The detailed description of all the parameters is given in Table~\ref{tab:parameters-systeme-ACP} in the Appendix. In particular, the transitions among the compartments are of two types since they model two fundamental phenomena (see Figure \ref{APC-Kamal00}):
 \begin{itemize}
 \item \textbf{The intrinsic transitions:} They represent the behavioral transitions that depend on the individual properties (past experiences, level of risk culture etc.) They are modeled by linear terms in system \eqref{eq:main1System APC2 ODEs}. The parameters of these transitions are $b_i>0$ for $i=1,\ldots,4$ and $c_j>0$ for $j=1,2$.
  \item \textbf{The imitation phenomenon:} Individuals have a tendency to imitate the behaviors of people around.  Here we follow the dominant behavior principle, that is, in the case of two populations in interaction, the most adopted behavior is the most imitated one. Thus, imitation between two behaviors depend on the ratio of the two populations. Only alert behaviors are not imitable. In system \eqref{eq:main1System APC2 ODEs} the behavioral transitions due to imitation are represented by nonlinear terms defined as:
\begin{align}
\mathcal{F}(\rho_1, \rho_3)&:=\alpha_{13} \xi\left(\frac{\rho_3}{\rho_1+\varepsilon}\right) \rho_1 \rho_3,\label{F} \\
\mathcal{G}(\rho_1, \rho_2)&:=\alpha_{12} \xi\left(\frac{\rho_2}{\rho_1+\varepsilon}\right) \rho_1 \rho_2,\label{G}\\
\mathcal{H}(\rho_2,\rho_3)&:=\left( \alpha_{23}  \xi\left(\frac{\rho_3}{\rho_2+\varepsilon}\right) -\alpha_{32} \xi\left(\frac{\rho_2}{\rho_3+\varepsilon}\right) \right) \rho_2 \rho_3.\label{H} 
 \end{align}
 The parameter $0<\varepsilon\ll 1$ is considered here to avoid  singularities, and the following function
\begin{equation}\label{xiw}
\xi(w):=\dfrac{w^{2}}{1+w^{2}}, \; w \in \mathbb{R}.
\end{equation}
takes into account the dominant behavior principle, that is the fact that the rate of imitation depends on the ratio of the corresponding populations (see Figure \ref{Xi function}). For example, if we consider the imitation phenomenon from alert to panic, we remark that if $\frac{\rho_2}{\rho_1+\varepsilon} \ll 1$ is small, then $\xi\left(\frac{\rho_2}{\rho_1+\varepsilon}\right)$ is almost equal to zero, so the imitation is weak. Conversely, if $\frac{\rho_2}{\rho_1+\varepsilon}\gg 1$ is large, it means that we have a majority of individuals in panic. In this case, $\xi\left(\frac{\rho_2}{\rho_1+\varepsilon}\right)$ goes to $1$ and alerted individuals would imitate the panic ones. The same situation holds for the other imitation transitions. 
 \end{itemize}
  \begin{figure}
\begin{center}
\begin{tabular}{c}
\includegraphics[width=6cm,height=4cm]{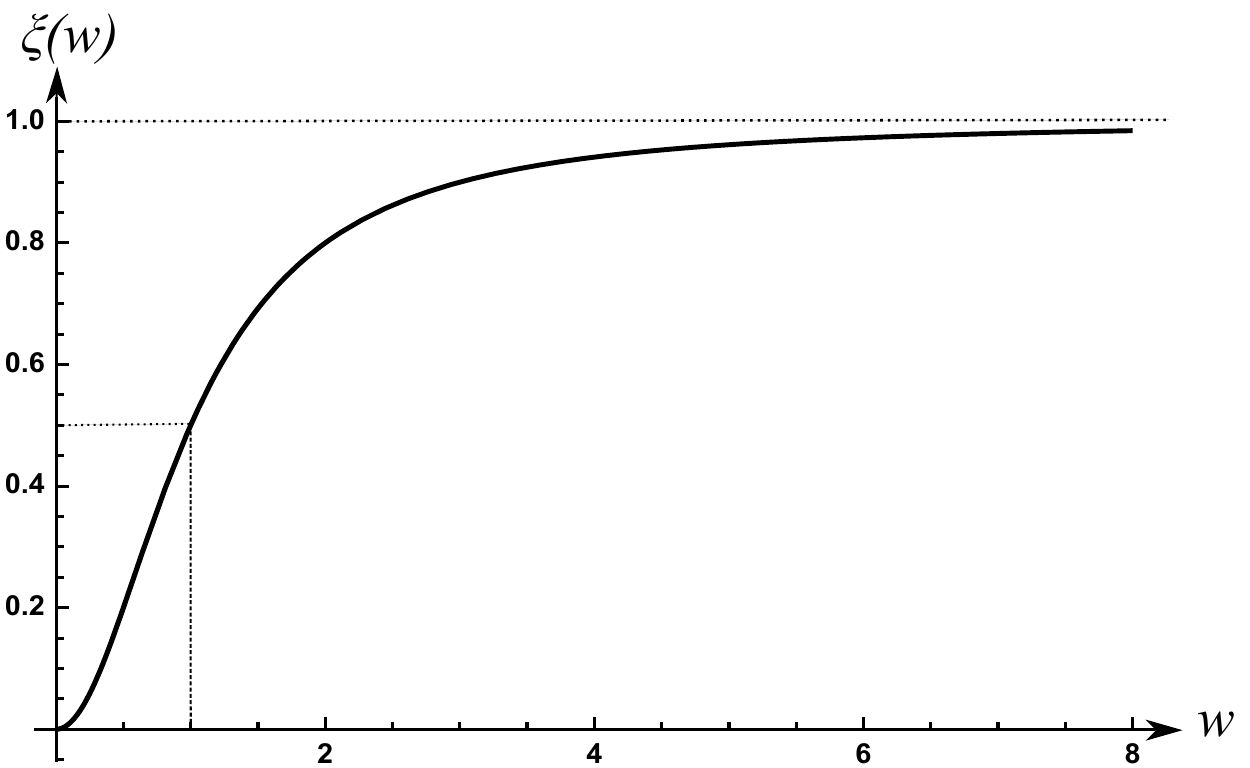}\\
\end{tabular}
\caption{The function $\xi$ involved in the imitation terms: the imitation starts very slowly, then it accelerates before slowing down and saturating.}
\label{Xi function}
\end{center}
\end{figure}

\begin{figure}
\begin{center}
\begin{minipage}[c]{.46\linewidth}
\includegraphics[width=6cm,height=4cm]{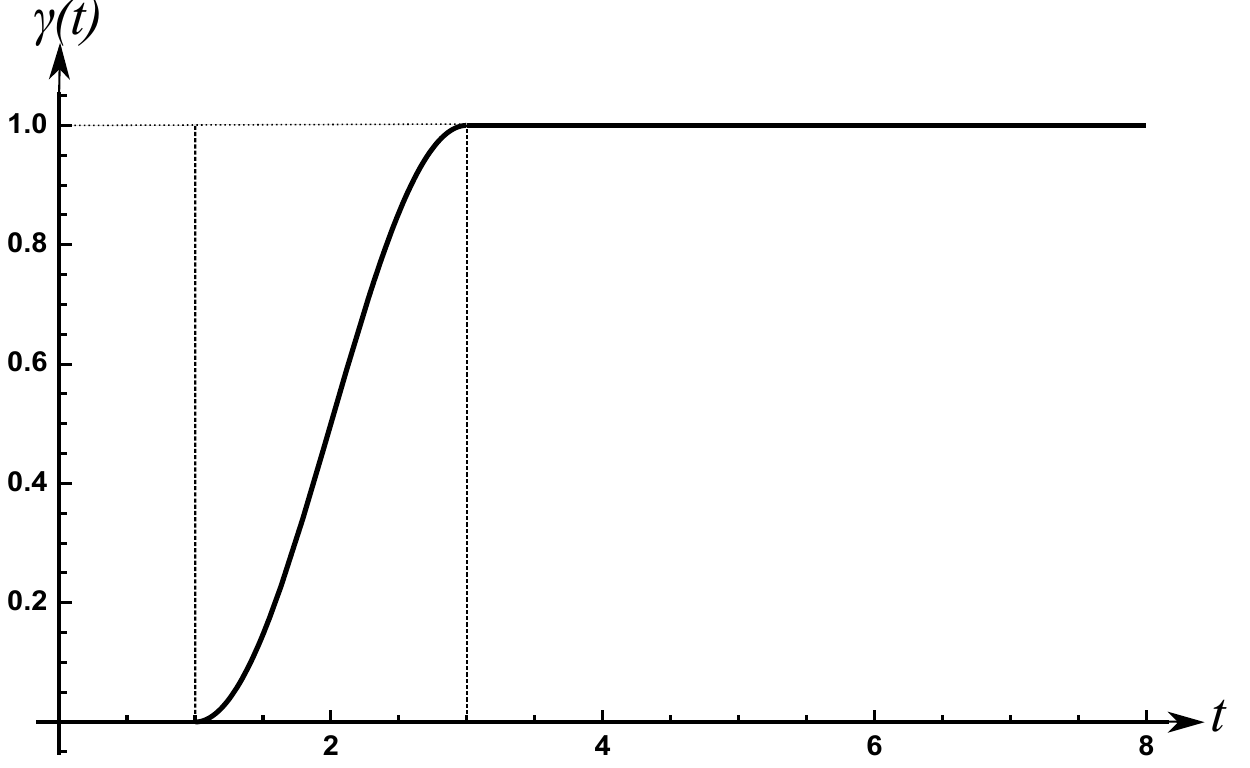}
\end{minipage} \hfill
\begin{minipage}[c]{.46\linewidth}
\includegraphics[width=6cm,height=4cm]{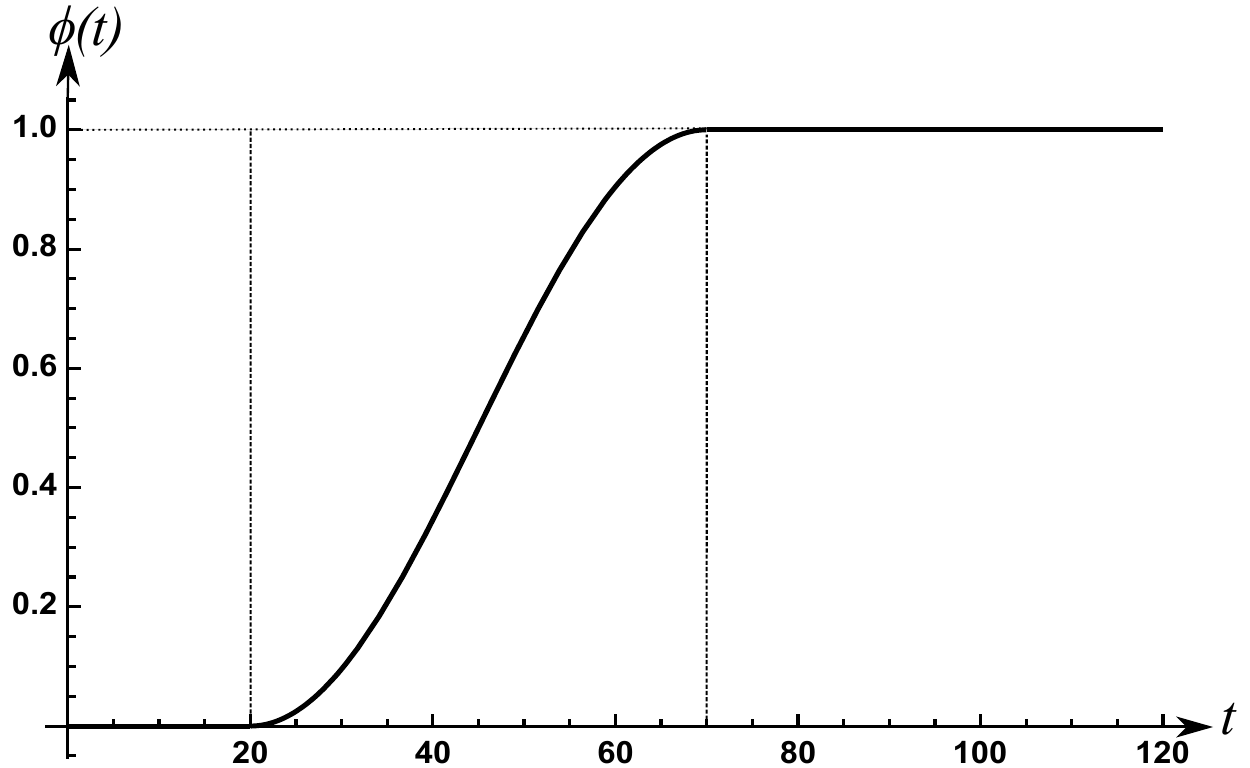}
\end{minipage}
\caption{Example of the functions $\gamma$ and $\phi$, which describe the transition from the daily to the alert behaviors, and from the control to everyday life behaviors: $\gamma(t)=\zeta(t,1,3)$ and $\phi(t)=\zeta(t,20,70)$ respectively.}
\label{Gamma and Phi functions}
\end{center}
\end{figure} 

 Finally, function  $\gamma$ describes the transition from the daily to the alert behavior at the beginning of the catastrophic event, while function $\phi $ represents the transition from a control behavior to an everyday life behavior at the end of the event. It is assumed that they are time-dependent functions that depend on the nature of the disaster. In \cite{Lanza_et_al.} the authors consider the functions $ \phi, \ \gamma: [0,\infty)\longrightarrow [0,1] $  defined as 
\begin{equation}\label{phi} 
\phi(t):=\zeta(t,\tau_0,\tau_1)\quad\text{for}\quad \tau_0 < \tau_1, 
\end{equation}
and 
\begin{equation}\label{gamma} 
 \gamma(t):=\zeta(t,\sigma_0,\sigma_1)\quad\text{for}\quad \sigma_0 < \sigma_1, 
\end{equation}
where 
\begin{equation}\label{z}
 \zeta(t,z_0,z_1) :=\left\{
 \begin{aligned}
 &0 & \quad \text{ if } t<z_0 , \\
 & \dfrac{1}{2}-\dfrac{1}{2}\cos\left(\frac{t-z_0}{z_1-z_0}\pi\right) & \quad \text{ if } z_0\leq  t \leq z_1,   \\
 & 1  & \quad \text{ elswhere}.    \\
  \end{aligned}
    \right.
\end{equation}
Here $ \tau_0 $ is the time at which the daily population starts to be impacted by the event, and $ \tau_1 $ is the time at which the total daily population becomes alert. Additionally, $ \sigma_0 $ represents the time at which the first individuals in a control state can go back to a pseudo-daily behavior, while~$ \sigma_1 $ is the time where this transition is highest. See Figure \ref{Gamma and Phi functions} for an example of these two functions.\\

Summing up the equations of \eqref{eq:main1System APC2 ODEs}, we notice that 
\begin{equation}
\sum_{i=1}^6 \rho_i(t)=1 \text{ (the total population density)}, \; \forall t\geq 0.  \label{Total density}
\end{equation} 
For this reason we only need to solve the system \eqref{eq:main1System APC2 ODEs} without considering the victims density~$\rho_6$, since the last equation is a linear combination of the others.

\subsection{First-order macroscopic pedestrians models} \label{Section1}
In this section we give the mathematical framework about first-order macroscopic pedestrians models \cite{Coscia,Hughes2002}. We recall that the continuity principle indicates that the variation of the density $\rho$ of a certain quantity, in $\Omega \subset \mathbb{R}^2$, is given by the balance of the flow $q$ 
 of this quantity across the boundary $\partial \Omega$ and the amount of quantity produced or removed inside the domain. Mathematically, this can be expressed as follows:
\begin{equation}
 \partial_t \rho+ \nabla \cdot q=S(t,x,\rho), \quad  t\geq 0, \ x\in \Omega, \label{NC Continuity Equation}
\end{equation}
where $S$ is the source term. To be more precise, the flux $q$ can be advective $(q_{adv})$, proportional to a velocity $\vec{v}(\rho)$ where $\rho$ is the transported scalar quantity, which means that $q_{adv}=\rho\vec{v}(\rho)$; but it may also be diffusive ($q_{diff}$), that is, it consists of a diffusion term $q_{diff}= -d \nabla \rho $ corresponding to the transportation of the scalar quantity according to its gradient, where $d$ is the constant diffusion rate. Thus, we have 
$$
q=q_{adv}+q_{diff}:=-d \nabla \rho+
\rho\vec{v}(\rho) .
$$ 
Moreover, the source term $S$ can be divided into a pure and a reaction source terms: 
 $$S=S_p+S_r.$$
The pure source term denoted by $S_p$ represents the self-creation/destruction rate inside the domain (using population dynamics terminology, it corresponds to the birth and death terms for example). This term will be denoted by $S_p (t,x,\rho)=g(t,x,\rho) $ where $g$ is a given linear function with respect to $\rho$. The reaction term $S_r$ describes the creation/destruction processes as a reaction to this quantity (corresponding to the reaction and interaction terms). This term will be denoted by $S_r(t,x,\rho)=f(t,x,\rho)$, where $f$ is a given nonlinear function with respect to $\rho$. Moreover, the associated boundary conditions are expressed as follows 
$$ 
q\cdot n =q^0 \cdot n,
$$
where $n$ is the outward unit normal vector to $\partial\Omega$.

These boundary conditions mean that the flux crossing the boundary part $\partial \Omega $ in the direction of the normal vector $n$ is given by an observed flux $q_0$ in the same direction $n$. According to that, the complete  first-order equation \eqref{NC Continuity Equation} in its non-conservative form is given by:
\begin{equation}\label{eq:mainSystem FirstOr}
            \left\{
                 \begin{array}{l l l l}
                  \partial_t \rho &=d\Delta \rho - \nabla \cdot (\rho\vec{v}(\rho) ) +g(t,x,\rho)+f(t,x,\rho), &\quad t\geq 0, \text{ in } \Omega ,\\
                    q\cdot n &=q^0 \cdot n,  &\quad  t\geq 0, \text{ on } \partial \Omega ,\\
                   \rho (0)&=\rho_0,      &\quad \text{ in } \Omega ,            \end{array}\right.
      \end{equation}  
where $\rho_0$ is the initial condition. For more details we refer to \cite{Batc} and references therein. 
 
\subsection{The spatio-temporal model corresponding to \eqref{eq:main1System APC2 ODEs}}\label{ Section21 Spatio-temporal model}
In this section, we present our advection-diffusion APC (Alert-Panic-Control) model using the first-order continuity equations \eqref{eq:mainSystem FirstOr} presented in Section \ref{Section1}. Notice that system \eqref{eq:mainSystem FirstOr} can be generalized to the case where several populations are in interaction (as is the case for the system \eqref{eq:mainSystem Intro2}, where each population can for example have different directions of movement). Moreover, our new model takes into account the transitions and the imitations characteristics which are not considered in the conservative system \eqref{eq:mainSystem Intro2}.\\

Let $\Omega\subset\mathbb{R}^2$ be a non-empty bounded domain with Lipschitz boundary. Consider the local population densities $\rho_{i} : [0,+\infty) \times \Omega \longrightarrow \mathbb{R}$ for $i=1,\ldots,5$ where
\begin{itemize}
\item $\rho_1(t,x)$ is the local density of individuals in the alert situation,
\item $\rho_2(t,x)$ is the local density of individuals in the panic situation, 
\item $\rho_3(t,x)$ is the local density of individuals in the control situation, 
\item $\rho_4(t,x)$ is the local density of individuals in the daily behavior before the disaster,
\item $\rho_5(t,x)$ is the local density of individuals corresponding to the daily situation after the disaster, 
\end{itemize}
and let $\rho$ be given by
\[
 \rho:=(\rho_1,\rho_2 ,\rho_3,\rho_4,\rho_5 )^{*}.
\] 
As for the model \eqref{eq:mainSystem FirstOr} we will consider an advective flux and a diffusive flux. The advection phenomenon models the movement of a population in a chosen direction, typically to escape the $\Omega$ domain. It is therefore natural to incorporate advection terms in our case.

We set 
$$
q_{i,adv}:=\rho_i\vec{v}_i(\rho), 
$$
the advective flux, where
$\rho\vec{v}_i(\rho)$, $i=1,\ldots,5$ is the corresponding velocity for each population of density $\rho_i$, $i=1,\ldots,5$. 
The alert population corresponds to the set of behaviors such as information seeking and hazard identification, and therefore here it is assumed that it cannot undergo the advection phenomenon.
Thus, it is assumed that only the populations in a situation of panic or control are concerned by the advection phenomenon, hence
\[
\vec{v}_i(\rho) =0,\quad \text{for }i=1,4,5.
\]
For $i=2,3$, we assume that the velocities $\vec{v}_i$ satisfy
\[
\vec{v}_i(\rho) =V_i(\rho)\,\vec{\nu},\quad \text{for }i=2,3,
\]
where the vector $\vec{\nu}:\overline{\Omega}\to \mathbb{R}^2$  that represents the direction of the movement, and satisfies certain conditions that will specified later. Moreover, $V_2,V_3$ are the scalar speed-density functions. Several different type of speed-density functions are used in the literature (see \cite{Buchmueller,Coscia,CristianiPiccoliTosin,Seyfried2005}). Here, we  choose a linear dependence:
$$ 
V_2 (\rho ) = V_{2,\max}\left( 1- \tilde{\rho}\right)\quad\text{and}\quad V_3 (\rho)=V_{3,\max} \left( 1- \tilde{\rho}\right),   
$$
where $V_{2,\max},V_{3,\max}$ are two positive constants and 
\[
\tilde{\rho} :=\sum_{i=1}^{5}\rho_i. 
\]
Similar assumptions on the panic and the control maximum speeds are used in \cite{CristianiPiccoliTosin,Liu-Zh-Hu}.
\begin{remark}
It should be pointed out that the examples of speed-density relationships discussed in this article (and in the references \cite{Buchmueller,Coscia,CristianiPiccoliTosin,Seyfried2005} cited above) are provided only by empirical data (experiments with certain restricted pedestrian flows) which are valid under a steady condition. To develop more sophisticated and complex density-speed relationships, and to get more information even just on the maximum speeds $V_{i, max}$, several additional factors should be taken into account \cite{Hermant}, such as: psychological and social forces, individual differences, the social and the age composition of the crowd, the spatial configuration, etc.  
\end{remark}

Moreover, each population diffuses in the spatial domain $\Omega$ according to the density gradient
$$
q_{i,diff}:=-d_i\nabla \rho_i,
$$
 with constant diffusion coefficients $d_i$ for $i=1,\ldots,5$. It is assumed that the whole crowd diffuses with different diffusion coefficients depending on the type of behavior.



With these assumptions and notations, the associated fluxes are given by
\[
q_i:=-d_i\nabla \rho_i+\rho_i \vec{v}_i(\rho), \quad \text{for }i=1,\ldots,5.
\]

The source (pure and reaction) terms correspond to the intrinsic transitions and the imitation terms described in Section \ref{Section0}.

Moreover, we divide the boundary $\partial\Omega$ of $\Omega$ in two parts, each of them corresponding to different boundary conditions:
\[
\partial \Omega := \Gamma_1 \cup \Gamma_2\quad\text{with}\quad \Gamma_1\cap\Gamma_2=\emptyset.
\]
Here $\Gamma_{1}$ corresponds to the part of boundary that could not be crossed by the population, while~$\Gamma_2$ corresponds to an exit. We define the observed fluxes $q^0_i$ on the boundary $\partial\Omega$ by
\[
q^0_i(\rho_i):=\begin{cases}
0&\text{on }\Gamma_1\\
-\rho_i v_{i,\text{out}}\,\vec{\nu}&\text{on }\Gamma_2
\end{cases}
\] where $v_{i,\text{out}}\geq0$ is the constant speed at the boundary $\Gamma_2$ and $\vec{\nu} $ is the direction of the movement. This means that each population cannot cross along $\Gamma_1$ and cross the escape $\Gamma_2$ with speed $v_{i,\text{out}}$. We assume that the function $\vec{\nu}$ satisfies:
\begin{equation}
\vec{\nu}(x)=\begin{cases}
(0,0)^*,& x\in \Gamma_1\\
(\nu_{x_1}(x),\nu_{x_2}(x))^{*},& x\in \Omega\\
n(x),& x\in\Gamma_2 \label{Desired direction}
\end{cases}
\end{equation}
where $n$ is the unit normal vector at the boundary part $\Gamma_2$. This choice of vector $ \vec{\nu}$ means that, at the part of the boundary $\Gamma_1$ where pedestrians cannot cross, the direction of movement vanishes, but at the target exit $\Gamma_2$, the pedestrians cross this part of the boundary in a direction parallel to the normal vector $n$, while the desired direction of motion inside the domain $\Omega$ is given in a suitable way that depends on the regularity of $\Omega$, and it satisfies the following assumption:
$$ \vec{\nu}_{| \Omega} \in W^{1,\infty}(\Omega,\mathbb{R}^2),$$
such that $\nabla \cdot \vec{\nu}(x) \leq 0$ for all $ x\in \Omega$. For example, we can take $\vec{\nu}_{| \Omega}$ to be normalized vectors between any point $ x\in \Omega $ and a centered (fixed) target point that lies outside $\overline{\Omega}$, see \eqref{Direction of the mouvment nu}.
Thus, the observed fluxes on the boundary $\partial\Omega$ are given by
\[
q^0_i(\rho_i)=\rho_i v_{i,\text{out}}\,\vec{\nu}.
\]
We assume that at the beginning $t=0$, the whole population is in a daily behavior, so we consider the following initial conditions: $\rho(t=0,0)=\rho_0$ where $\rho_0$ is given by
\begin{equation}\label{defrho0}
\rho_0:=(0,0,0,\theta,0)^*\quad\text{on }\Omega,
\end{equation}
 and $\theta:\Omega\to [0,\infty)$ is such that, the integral exists, and that
\[
\int_{\Omega}\theta(x) dx=1.
\]
Thus, from \eqref{eq:main1System APC2 ODEs} and \eqref{eq:mainSystem FirstOr}, we obtain the following system:
\begin{equation}\label{eq:main1System APC21}
            \left\{
    \begin{array}{l l l }
     \partial_t \rho_1 =& d_1 \Delta \rho_1   - (b_1+b_2+\delta_1) \rho_1 + \gamma(t) \rho_4+b_3 \rho_3 +b_4  \rho_2  & \\ & -\mathcal{F}(\rho_1, \rho_3) - \mathcal{G}(\rho_1, \rho_2) \; & \text{ in } \Omega ,\  t \geq 0,\\
     \partial_t \rho_2 = &d_2 \Delta \rho_2  -(b_4 + c_1 + \delta_2) \rho_2+ b_2 \rho_1 +c_2 \rho_3 & \\ & - \nabla \cdot  (\rho_2 \vec{v}_2(\rho))  +\mathcal{G}(\rho_1, \rho_2)  - \mathcal{H}(\rho_2,\rho_3)  \; & \text{ in } \Omega ,\ t\geq 0,\\
    \partial_t \rho_3 =&  d_3 \Delta \rho_3 -(b_3 +c_2+\delta_3)\rho_3  +b_1 \rho_1 +c_1 \rho_2   \\ 
         &-\phi(t)\rho_3 -\nabla \cdot (\rho_3 \vec{v}_{3}(\rho)) + \mathcal{F}(\rho_1,\rho_3)+ \mathcal{H}(\rho_2,\rho_3) \; & \text{ in } \Omega , \ t  \geq 0,\\
        \partial_t \rho_4 =& d_4 \Delta \rho_4  - \gamma(t) \rho_4  \; & \text{ in } \Omega ,\ t \geq 0\\
                      \partial_t \rho_5 =& d_5 \Delta \rho_5 +\phi(t) \rho_3 \; & \text{ in } \Omega , \ t \geq 0. \\
                 \end{array}\right.
      \end{equation} 
			with the boundary conditions
\begin{equation}\label{boundary conditions}
d_i\nabla\rho_i\cdot n =\left(\rho_i\vec{v}_i(\rho)-
\rho_i v_{i,\text{out}}\vec{\nu}\right)\cdot n \quad\text{on }\partial\Omega,\ t\geq 0,\quad i=1,\ldots5,
\end{equation}
or more explicitly, using the definition of $\vec{\nu}$: 
 \begin{equation*} \label{suitable boundary conditions 2}
 \left\{
    \begin{aligned}
    d_i \nabla \rho_{i}\cdot n &=  0   & \quad & \text{ on } \Gamma_1 \; \text{ for } i=1,\cdots, 5,\\
    d_i \nabla \rho_{i}\cdot n &= \rho_i V_{i}(\rho)- \rho_i v_{i,out}   & \quad &\text{ on } \Gamma_{2} \; \text{ for } i=1,\cdots, 5,\\
      \end{aligned}
    \right.
\end{equation*} 
and the initial condition
\begin{equation} \label{initial condition}
\rho(t=0,\cdot)=\rho_0\quad\text{in }\Omega.
\end{equation}

\section{Local existence, positivity and  $L^1$--boundedness  of the spatio-temporal model \eqref{eq:main1System APC21}-\eqref{initial condition}}\label{sec:3} 
In this section, we prove the well-posedness of the spatio-temporal APC model \eqref{eq:main1System APC21}-\eqref{initial condition} introduced previously in Section \ref{sec:2}. Then, we establish the positivity of the solutions and the $L^1$-boundedness of the population densities. 
\subsection{The abstract formulation and the associated boundary value Cauchy problem}   
To study the existence and uniqueness of solutions to the system spatio-temporal APC model \eqref{eq:main1System APC21}-\eqref{initial condition}  we use the abstract formulation and semigroup theory \cite{Nagel,Pazy}. In order to do that, for $ p>2 $, we define the Banach space $ X:= L^p(\Omega)^{5} $, the product of the Lebesgue spaces of order $p$,  equipped with the following norm $$ \| \varphi:=(\varphi_1,\cdots,\varphi_5)^{*}\|:=\sum_{i=1}^{5} \| \varphi_i \|,$$ where $\|\cdot \|$ is the usual norm in $L^p(\Omega)$, and $^*$ designate the vector transpose. It is clear that $X$ is a Banach lattice, that is, $$ | \varphi_i(x)| \leq  | \psi_i(x) |  \text{  for } a.e. \, x\in \Omega \text{ for all } i=1,\cdots,5   \text{  implies that  }  \|\varphi \|\leq \|\psi \|.$$ Moreover, we define the linear closed operator $  (\mathcal{A},D(\mathcal{A} )) $ on $X$ by
\begin{equation}
 \left\{
    \begin{aligned}
     \mathcal{A} 
     &=
 diag( d_{1}  \Delta,\cdots,d_{5}  \Delta)\\   
     D(\mathcal{A} )& = W^{2,p}(\Omega)^{5}.
    \end{aligned} \label{EqApp1Chapter1}
  \right. 
\end{equation} 

Let the Banach space $Z:=D(\mathcal{A})$ equipped with its usual norm. The nonlinear function $\mathcal{K}:[0,\infty)\times X_{\alpha}\longrightarrow X$ is defined by 

\begin{align}
\mathcal{K}(t,\varphi )& = \begin{pmatrix}
\mathcal{K}_{1}(t,\varphi_1, \nabla \varphi)\\
\mathcal{K}_{2}(t,\varphi_2, \nabla \varphi)\\
\mathcal{K}_{3}(t,\varphi_3, \nabla \varphi)\\
\mathcal{K}_{4}(t,\varphi_4, \nabla \varphi)\\
\mathcal{K}_{5}(t,\varphi_5, \nabla \varphi)\end{pmatrix}\nonumber \\
&=\begin{pmatrix}
 - (b_1+b_2+\delta_1)\varphi_1+ \gamma(t) \varphi_4+b_3 \varphi_3 +b_4  \varphi_2   -\mathcal{F}(\varphi_1, \varphi_3) - \mathcal{G}(\varphi_1, \varphi_2)\\
 -(b_4 + c_1 + \delta_2)\varphi_2+  b_2 \varphi_1 +c_2 \varphi_3  - \nabla \cdot  (\varphi_2 V_2(\varphi)\nu)  +\mathcal{G}(\varphi_1, \varphi_2)  - \mathcal{H}(\varphi_2,\varphi_3) \\
 -(b_3 +c_2+\delta_3)\varphi_3+b_1 \varphi_1 +c_1 \varphi_2 -\phi(t)\varphi_3 - \nabla \cdot (\varphi_3 V_{3}(\varphi)\nu) + \mathcal{F}(\varphi_1,\varphi_3)+ \mathcal{H}(\varphi_2,\varphi_3) \\
 - \gamma(t) \varphi_4   \\
\phi(t) \varphi_3 
  \end{pmatrix}, \label{The nonlinear term}
\end{align}

and $X_{\alpha}:=\lbrace \varphi \in W^{2\alpha,p}(\Omega)^{5}: d_i \partial_{n}  \varphi_{i|\partial \Omega} =0  \rbrace$ for some (fixed) $\alpha \in (1/p+1/2,1)$ equipped with the norm $ \|\cdot\|_{0,\alpha}:= \|\cdot\|+\|\nabla \cdot\|+[\cdot]_{\zeta} $ where $$  [\varphi]_{\zeta}:= \left( \int_{\Omega \times \Omega} \dfrac{|\varphi(x)-\varphi(y) |^p}{|x-y |^{2+p\zeta}} dx dy\right)^{1/p}, \quad \zeta=2\alpha -1.$$ Hence, $\| \varphi\|_{\alpha}:=\sum_{i=1}^{5} \| \varphi_i \|_{0,\alpha} $ defines a norm on $X_{\alpha}$ which make it a Banach space. Then from the Sobolev embedding, we have
$$ X_{\alpha}\hookrightarrow  C^1(\overline{\Omega})^{5}. $$

Moreover, let us define the boundary space $\partial X:=W^{1-1/p,p}( \partial \Omega )^{5}$ which is equipped with the norm $$\|\varphi \|_{\partial X}:= \sum_{i=1}^{5} |\varphi_i |_{p}$$ where $$ | \varphi |_{p}=\left( \int_{\partial\Omega} |\varphi(x)|^p d\sigma(x) + \int_{\partial\Omega \times \partial\Omega} \dfrac{|\varphi(x)-\varphi(y) |^p}{|x-y |^{p}} d\sigma(x) d\sigma(y)\right)^{1/p}  ,$$
where $\sigma $ is a (surface) measure in $\partial \Omega$. For more details about the fractional Sobolev spaces on $\Omega$ and its boundary $\partial\Omega$ respectively, we refer to the monographs of Adams  \cite{Adams} and Brezis \cite{Brezis}. Since $ 1-2/p >0 $, we obtain the continuous embedding $ \partial X \hookrightarrow C(\partial \Omega )^{5}$.
Moreover, we define the boundary operator $\mathcal{L}: Z \longrightarrow \partial X $  by
\begin{equation}
 \mathcal{L}\varphi= \left(d_1\partial_{n}   \varphi_1, \cdots, d_5 \partial_{n} \varphi_5  \right)^* \quad \text{ on } \partial \Omega.
 \end{equation}
The nonlinear boundary term $ \mathcal{M} : X_{\alpha} \longrightarrow\partial X$ is given by 

\begin{equation}
 \mathcal{M}\varphi = \left\{
    \begin{aligned}
     & (0,0,0,0,0)^* \; & \text{ on } \Gamma_1 , \\
 & 
(-v_{1,out} \varphi_1,  
 -v_{2,out} \varphi_2+  \varphi_2 V_2(\varphi ),
 -v_{3,out} \varphi_3+ \varphi_3 V_3(\varphi ),
-v_{4,out} \varphi_4,
-v_{5,out} \varphi_5)^*
  \; & \text{ on } \Gamma_2. \\
  \end{aligned} 
  \right. 
\end{equation}
To check that the operator $\mathcal{M}$ is well defined, we use the trace theorem, see \cite[Chapter 9]{Brezis}. Indeed, let $\varphi \in X_{\alpha}$, then, in particular, for each $i=1,\cdots,5$,  $\varphi_i \in W^{1,p}(\Omega)$. So that, using the trace theorem, we have, $$\varphi_{i |\partial \Omega} \in W^{1,1-1/p}(\partial \Omega).$$ Moreover, since $$ \varphi_j V_j(\varphi )=V_{j,max}\varphi_j (1-\sum_{i=1}^{5} \varphi_i )\quad  \text{ for } j=1,2, $$  
and using the embedding $ W^{2\alpha,p}(\Omega)\ \hookrightarrow  C^1(\overline{\Omega})$, we obtain that 
$$\varphi_j V_j(\varphi ) \in C^1(\overline{\Omega}) \subset  W^{1,p}(\Omega) \quad  \text{ for } j=1,2. $$
Hence, by trace theorem, we obtain that 
$$ (\varphi_j V_j(\varphi ))_{|\partial \Omega} \in W^{1,1-1/p}(\partial \Omega) \text{ for } j=1,2.  $$
Thus, $\mathcal{M}\varphi \in \partial
X.$
\\

Otherwise, the continuous embedding $ Z\hookrightarrow X_{\alpha}$ holds which yields that the linear operator $ \mathcal{A} : Z \longrightarrow X$ is bounded and $\mathcal{L}: Z \longrightarrow \partial X$ is bounded and surjective (see \cite{Amann,Amann2}). 
The initial conditions are given by the following vector
\begin{equation}
\rho_0=(0,0,0,\theta,0)^*.
\end{equation}
where $\theta $ is defined above. Now, under the above considerations, we can write our model \eqref{eq:main1System APC21}-\eqref{initial condition} as the following abstract boundary evolution system:

\begin{equation}\label{eq:main1System APC21 boundary}
            \left\{
                 \begin{aligned}
               u_t (t)  = & \mathcal{A}   u(t) + \mathcal{K}(t,u (t)),& \quad
                t\geq 0, \\
                   \mathcal{L} u (t)= & \mathcal{M}(u(t)),& \quad
                t \geq 0,\\
                   u(0) =&u_0, & \\
                 \end{aligned}
                 \right.
\end{equation} 
where \begin{equation*}
u(t):= (
  \rho_1(t,\cdot),
 \rho_2(t,\cdot),
\rho_3(t,\cdot),
\rho_4(t,\cdot),
\rho_5(t,\cdot))^*, \quad t\geq 0,
\end{equation*}
and $$u_0=\rho_0 .$$
Let $T>0$, $ u_0 \in Z$ and for each $v\in X_\alpha$, $ \mathcal{K}(\cdot,v) \in C([0,T],X) $. By a (classical) solution of equation \eqref{eq:main1System APC21 boundary} we mean any function $u\in C([0,T],Z)\cap C^{1}([0,T],X_{\alpha})$ such that $u$ pointwisely satisfies \eqref{eq:main1System APC21 boundary} in $[0,T]$.\\
 
In order to study the boundary evolution equation \eqref{eq:main1System APC21}, we proceed as in H. Amann \cite{Amann2}, G. Greiner \cite{Greiner} and later W. Desh et al. \cite{Desh} (see \cite[Section 12]{Amann2} and also \cite[Section 4]{Desh}), namely, the nonlinear boundary evolution problem \eqref{eq:main1System APC21} admits a (classical) solution, if and only if, the following semilinear Cauchy problem admits a (classical) solution
\begin{equation}\label{eq:main1System APC211}
            \left\{
                 \begin{aligned}
                u_t (t)  = & \mathcal{A}_{\beta-1}   u(t) + \tilde{\mathcal{K}}(t,u (t)),& \quad
                t\geq 0, \\
                   u(0) =&u_0, & \\
                 \end{aligned}
                 \right.
\end{equation}
in a extrapolated Banach space of $X$ denoted here by $ X_{\beta-1} $ for some $ 0<\beta<1 $ (see Section \ref{Preliminary results} for the definition of such spaces) where $\tilde{\mathcal{K}}:=\mathcal{K}+ (\lambda-\mathcal{A}_{\beta-1})  \mathcal{D} \mathcal{M}$ with $\mathcal{D}$ is the Dirichlet map associated to the operator $(\lambda -\mathcal{A})$, which means that $v = \mathcal{D} w$ is the unique solution of the abstract boundary value problem
\begin{equation}
            \left\{
                 \begin{aligned}
                  (\lambda -\mathcal{A})v =0 & \\
                   \mathcal{L}v =w &
                 \end{aligned}
                 \right.
\end{equation}
for each $w\in \partial X$ for some $\lambda \in \varrho(\mathcal{A})$.

For $T>0$, $ u_0 \in X_\beta$ and for each $v\in X_\alpha$, $ \tilde{\mathcal{K}}(\cdot,v) \in C([0,T],X_{\beta-1}) $, we recall that a (classical) solution to a semilinear evolution equation of the form \eqref{eq:main1System APC211}, is the unique function $u\in C([0,T],X_\beta)\cap C^1([0,T],X_{\beta-1})$ that satisfies \eqref{eq:main1System APC211} pointwisely in $[0,T]$.

This approach of studying the boundary evolution equation \eqref{eq:main1System APC21 boundary} by equivalently studying the Cauchy problem \eqref{eq:main1System APC211} was first introduced separately in \cite{Amann2,Greiner}, and was later perfected in \cite{Desh} for a control point of view (see also the references therein). The conditions under which this approach is used are cited in the references \cite{Amann2,Desh,Greiner} and are listed below:\\
\textbf{(C1)} There exists a new norm $ \mid \cdot \mid_{m} $ which is finer than the norm of $X$, such that the space $Z:= (D(\mathcal{A}), \mid \cdot \mid_{m }) $ is complete. That is, $ Z $ is continuously embedded in $ X $ and  $ \mathcal{A} \in L(Z, X). $ \\
\textbf{(C2)} The restriction operator $ \mathcal{A}_0= \mathcal{A}_{| ker(\mathcal{L})} $ generates a strongly continuous analytic semigroup. \\
\textbf{(C3)} The operator $ \mathcal{L}: Z\longrightarrow \partial X $ is bounded and surjective. \\
\textbf{(C4)} $Z $ is continuously embedded in $ X_{\alpha} $. That is, $ Z\hookrightarrow X_{\alpha} $ for some $ 0 < \alpha <1 . $ \\ 
\textbf{(C5)} The functions $ \mathcal{K}: [0,+\infty)\times X_{\alpha} \longrightarrow X $ and $\mathcal{M}: [0,+\infty)\times X_{\alpha} \longrightarrow \partial X $ are locally integrable in the first variable and continuous with respect to the second one. \\
We mention that the conditions \textbf{(C1)}-\textbf{(C5)} are all satisfied in this work. Before presenting our main results, we would like to mention a few necessary preliminary results.
\subsection{Preliminary results}\label{Preliminary results}
In this section, we give our preliminary theoretical results to study the local existence, uniqueness and positivity of the solutions of our equation \eqref{eq:main1System APC21 boundary}. For the sake of order, the proofs of the results in this section are given in Appendix \ref{Annexe:A}. In what follows, we introduce $ \mathcal{A}_0:=\mathcal{A}_{| \ker (\mathcal{L})} $ and $X$ is the Banach lattice introduced above.
\begin{definition}
\textbf{(i)} A vector $\varphi=(\varphi_1,\cdots,\varphi_5)^{*} \in X$ is said to be positive, i.e., $\varphi(x) \geq 0$, if and only if, $ \varphi_i (x) \geq 0 $ for $ a.e. \, x\in \Omega $ for all $ i=1,\cdots,5 $. So that, $ X^{+}$ denotes the positive cone of $X$.\\
\textbf{(ii)} A bounded operator $\mathcal{T}$, in the Banach lattice $X$, is said to be positive if and only if, for every $ \varphi  \in X $, $ \varphi(x) \geq 0 $ implies $ \mathcal{T} \varphi(x) \geq 0 $ for  $ a.e. \ x\in \Omega $.\\ 
\textbf{(iii)} A semigroup $(\mathcal{T}(t))_{t\geq 0}$, in the Banach lattice $X$, is said to be positive if and only if, for every $ \varphi  \in X $, $ \varphi(x) \geq 0 $ implies $ \mathcal{T}(t) \varphi(x) \geq 0 $ for all $ t\geq 0$ for  $ a.e. \ x\in \Omega $. 
\end{definition}
Hence, the following generation result holds for the operator $\mathcal{A}_0$ in $X$.
\begin{proposition}\label{Proposition:3.2}
The following assertions hold:\\
\textbf{(i)} The closed operator generates a contraction holomorphic $C_0$-semigroup $(\mathcal{T}(t))_{t\geq 0}$ on $X$.\\
\textbf{(ii)} The semigroup $ (\mathcal{T}(t))_{t\geq 0} $ generated by $ \mathcal{A}_0 $ is compact and positive.\\
Moreover, the semigroup $(\mathcal{T}(t))_{t\geq 0} $ is given by the following matrix-valued operators 
\begin{equation*}
\mathcal{T}(t)=diag(\mathcal{T}_{1}(t),\cdots ,\mathcal{T}_{5}(t) )^*, \quad t\geq 0,
\end{equation*}
where, for each $i=1,\cdots,5 $, $ (\mathcal{T}_{i}(t) )_{t\geq 0}$ is the semigroup generated by the operator $ \mathcal{A}_i=d_i \Delta $ in $L^p(\Omega)$.
\end{proposition}

Now, we present the inter- and extrapolation spaces associated to the generator $\mathcal{A}_0$. We define  on $ X$ the norm $\|x\|_{-1}=\|R(\lambda,\mathcal{A}_0) x\| $, for $ x\in X . $  Then the completion of $ ( X, \|\cdot\|_{-1}) $ is called the extrapolation space of $ X $ associated to $ \mathcal{A}_0 $ and will be denoted by $ X_{-1}. $ This means that $ \mathcal{A}_0 $ has a unique extension $ \mathcal{A}_{-1}: D(\mathcal{A}_{-1})=X\longrightarrow   X_{-1} .$ Since for every $ t\geq 0, $ $ (\mathcal{T}(t))_{t\geq 0}$ commutes with the operator resolvent $ R(\lambda,\mathcal{A}_0),$ the extension of $ (\mathcal{T}(t))_{t\geq 0}$ to $ X_{-1} $ exists and defines an analytic semigroup $ (\mathcal{T}_{-1}(t))_{t\geq 0}$ which is generated by $ \mathcal{A}_{-1} $. Let $\alpha \in (0,1)$, we define the following interpolated extrapolation spaces by: $$  X_{\alpha-1}= \overline{X} ^{\|\cdot \|_{\alpha-1}},  \quad \mbox{where}\quad \|x \|_{\alpha-1}:= \sup_{\omega >0} \| (\omega^{\alpha}R(\omega,\mathcal{A}_{-1}-\lambda)x \| .$$ 
Then, we have the following continuous embeddings:    
$$
D(\mathcal{A}_0) \hookrightarrow  X_{\alpha}  \hookrightarrow  X_{\beta} \hookrightarrow  X $$  $$ X \hookrightarrow  X_{\alpha -1}   \hookrightarrow  X_{\beta -1} \hookrightarrow  X_{-1}, $$ for all  $ 0 < \beta < \alpha < 1$, where $D(\mathcal{A}_0)$ is equipped with the graph norm that makes it a Banach space. 
\begin{remark}
It follows from \cite[Sections 4.3.3 and 4.6.1]{Triebel}, that the spaces $X_\alpha$ for $0<\alpha<1$ introduced here coincide with real interpolation spaces (of order $\alpha$) between $D(\mathcal{A}_0)$ and $X$. Moreover, the embedding $ Z \hookrightarrow X_{\alpha} $ also holds.
\end{remark}
Now, we give the following result for the extrapolated semigroup of $ (\mathcal{T}(t))_{t\geq 0}$ in $X_{\beta-1}$.
\begin{proposition}\label{Proposition:3.5}
 For each $0 \leq \delta \leq 1$, $(\mathcal{T}_{\delta-1}(t))_{t\geq 0} $ is the unique extension semigroup of $ (\mathcal{T}(t))_{t\geq 0} $ with the associated generator $ \mathcal{A}_{\delta-1} $ satisfying $ D(\mathcal{A}_{\delta-1})=X_{\delta} $. Moreover, the semigroup $(\mathcal{T}_{\delta-1}(t))_{t\geq 0} $ inherits all the properties of $  (\mathcal{T}(t))_{t\geq 0} $. That is, $(\mathcal{T}_{\delta-1}(t))_{t\geq 0} $ is strongly continuous, analytic, compact and positive. 
\end{proposition}
\begin{definition} \label{Positive cones}
We define the positive cone of $X_{\alpha}$ by $$ X_{\alpha}^+=X^{+}\cap X_{\alpha} ,$$
However, $ X_{\beta-1}^+ $ is a closed convex cone in $ X_{\beta-1}$, satisfying: 
$$ X_{\beta-1}^+ \cap X =X^{+} .$$
\end{definition}

\begin{proposition}\label{Lemma Local Lipschitz}
The function $\tilde{\mathcal{K}}: [0,+\infty)\times X_{\alpha}\longrightarrow X_{\beta-1}$ is Lipschitz continuous in bounded sets. That is, for all $R>0$ there exists $L_R \geq 0$ such that 
\begin{equation}
\| \tilde{\mathcal{K}}(t,\rho)-\tilde{\mathcal{K}}(s,\upsilon) \|_{\beta-1} \leq L_R (\mid t-s \mid+ \|\rho-\upsilon \|_{\alpha})  \quad \text{ for all } \rho, \upsilon \in B(0,R) \, \text{for all } t,s\geq 0. \label{Lipschitz condition}
\end{equation} 
\end{proposition}
\begin{remark}\label{Remark Local Xbeta}
Notice that if $ 1/p+1/2<\beta <\alpha <1 $, then $ X_{\beta} \hookrightarrow C^1(\overline{\Omega})^5 $. So, \eqref{Lipschitz condition} holds also in $X_{\beta}$.
\end{remark}
The following regularity Lemma is also necessary.
\begin{lemma}\label{Lemma regularity}
 Let $ 0 <\beta  <1 $ and $ \mathcal{B}: [0,T]\longrightarrow X_{\beta-1} $ such that there exist $ 0< \eta \leq 1 $ and $l\geq 0$ satisfying
\begin{equation}
\| \mathcal{B}(t)-\mathcal{B}(s) \|_{\beta-1} \leq l |t-s|^{\eta}, \quad t,s \in [0,T]. \label{Lemma Holder conti B}
\end{equation}
Then, $$ v(t) =\int_{0}^{t} \mathcal{T}_{\beta-1}(t-s) \mathcal{B}(s) ds \in D(\mathcal{A}_{\beta-1})=X_{\beta} \quad \text{ for } 0  \leq t \leq  T.$$
Moreover, $v\in C^{1}((0,T],X_{\beta})$.
\end{lemma}
\begin{definition}\cite{Amann2,Desh}
Let $ u_0  \in  X_{\alpha}$ and $T>0$. By a solution to equation \eqref{eq:main1System APC211}, we mean a function $u\in C([0,T],X_{\alpha})\cap C^{1}([0,T],X_{\beta-1})$, such that $u(t) \in X_{\beta}$ for  $0 \leq t \leq  T$ and such that \eqref{eq:main1System APC211} is pointwisely satisfied. In particular, this solution must satisfy the following integral formula:
\begin{equation}
u(t)=\mathcal{T}(t)u_{0}+ \int_{0}^{t} \mathcal{T}_{\beta-1}(t-s)\left( \mathcal{K}(s,u(s))+(\omega-\mathcal{A}_{\beta-1}) \mathcal{D} \mathcal{M}(u(s))\right)  ds, \quad t\in [0,T].
\end{equation}
\end{definition}

\subsection{Local existence and regularity}\label{Local existence and regularity} In this Section, we prove the local existence, uniqueness and regularity of 
solutions to equation \eqref{eq:main1System APC211} which yields the local well-posedness for the model  \eqref{eq:main1System APC21}-\eqref{initial condition}. 
The following result is an extension, using extrapolation theory, of results due to Pazy \cite[Section 6.3]{Pazy}, Lunardi \cite[Section 7.1]{Lunardi} and Cholewa \& Dlotko \cite[Chapter 2. Section 3]{T.Dlotko}. A similar result, with weaker spatial regularity, is due to Amann \cite[Theorem 12.1]{Amann}. To obtain our result, we use Proposition \ref{Lemma Local Lipschitz} and Lemma \ref{Lemma regularity}.
\begin{theorem}[Local existence and regularity]\label{Theorem maximal existence}
For each $u_0 \in X_\alpha $ there exist a maximal time $ T(u_0 ) >0 $ and a unique maximal solution $u(\cdot):=u(\cdot ,u_{0}) \in C([0,T(u_0 )),X_{\alpha})\cap C^{1}((0,T(u_0 )),X_{\beta})$ of equation \eqref{eq:main1System APC211} such that
\begin{equation} \label{integral formulation of solution}
u(t)=\mathcal{T}(t)u_{0}+ \int_{0}^{t} \mathcal{T}_{\beta-1}(t-s)\underbrace{\left( \mathcal{K}(s,u(s))+(\omega-\mathcal{A}_{\beta-1}) \mathcal{D} \mathcal{M}(u(s))\right)}_{= \tilde{\mathcal{K}}(s,u(s))}  ds, \quad t\in [0,T(u_0 )).
\end{equation}
Moreover the solution $u$ satisfies the following blow-up property:
\begin{equation}
T(u_0 )=+\infty \quad \text{or} \quad  \limsup_{t\rightarrow T(u_0 )^{-} } \| u(t )\|_{\alpha} = +\infty . \label{ETFFormula}
\end{equation} 
\end{theorem}
\begin{proof}
Let $u_0 \in X_\alpha $. So, using Proposition \ref{Lemma Local Lipschitz}, it yields from \cite[Section 7.1]{Lunardi} (by taking $X_{\beta-1}$ instead of $X$ and $ (\mathcal{T}_{\beta-1}(t))_{\geq 0}$ instead of $(\mathcal{T}(t))_{\geq 0}$), that there exist $T>0$ (chosen small enough) and a unique solution $u\in C([0,T],X_{\alpha})\cap C^{1}([0,T],X_{\beta-1})$  of equation \eqref{eq:main1System APC211} satisfying the variation of constants formula \eqref{integral formulation of solution}. Note that, in our case, $ \mathcal{A}_0 $ is densely defined in $X$, so that the continuity at $t=0$ holds. Now, to conclude, we use the integral formula of our solution \eqref{integral formulation of solution} and Lemma \ref{Lemma regularity} to prove that $u\in  C^{1}((0,T],X_{\beta})$. \\
First, note that $u_0 \in X_{\alpha}\hookrightarrow X_{\beta } $ in addition with the analyticity of the semigroup $ (\mathcal{T}(t))_{t\geq 0}  $ implies that $\mathcal{T}(t)u_0 \in C^{1}([0,T],X_{\beta}) $. Then, it remains to prove that 
$$t\mapsto  v(t)= \int_{0}^{t} \mathcal{T}_{\beta-1}(t-s)\tilde{\mathcal{K}} (s,u(s))  ds \in  C^{1}((0,T],X_{\beta}) .$$
Remark that, $u$ is H\"older continuous in $X_{\beta-1}$ (since it is $C^1$), that is, there exist $\tilde{l}\geq 0$ and $0<\vartheta \leq 1$, such that
\begin{equation}
\| u(t)-u(s) \|_{\beta-1} \leq \tilde{l} |t-s|^{\vartheta},\quad t,s \in [0,T].\label{Holder continuity for solution}
\end{equation}
Moreover, since $u_0 \in X_{\alpha} \hookrightarrow X_{\beta}$, it yields, using Remark \ref{Remark Local Xbeta}, that $$u\in  C([0,T],X_{\beta})\cap C^{1}([0,T],X_{\beta-1}).$$
Hence, $u$ is bounded in $X_{\beta}$, since it is continuous. Furthermore, using the reiteration theorem, we obtain that $X_{\beta}=(X_{\alpha},X_{\beta-1})_{\tilde{\theta}}$, with $ 0<\tilde{\theta}<1 $. That is, there exists $ c(\alpha,\beta) \geq 0 $, such that
$$ \| u(t)-u(s) \|_{\beta} \leq c(\alpha,\beta) \| u(t)-u(s)\|_{\alpha}^{1-\tilde{\theta}}\| u(t)-u(s)\|_{\beta-1}^{\tilde{\theta}}, \quad t,s\in [0,T].$$ 
Therefore, we have
$$ \| u(t)-u(s) \|_{\beta} \leq \tilde{c}(\alpha,\beta) | t-s|^{\tilde{\theta}\vartheta }, \quad t,s\in [0,T],$$ 
for some constant $\tilde{c}(\alpha,\beta) \geq 0$. Note that $u$ is bounded in $X_\beta$. Hence, by \eqref{Holder continuity for solution} and Remark \ref{Remark Local Xbeta} (using Proposition \ref{Lemma Local Lipschitz} for $X_\beta$ instead of $X_\alpha$), we obtain that
\begin{align*}
\| \tilde{\mathcal{K}}(t,u(t))-\tilde{\mathcal{K}}(s,u(s)) \|_{\beta-1} 
&\leq L_R (\mid t-s \mid+ \|u(t)-u(s)\|_{\beta})  \\
& \leq \tilde{L_R} (\mid t-s \mid+ |t-s |^{\tilde{\theta}\vartheta} ), \quad t,s\in [0,T].
\end{align*}
This proves that $\tilde{\mathcal{K}}(\cdot,u(\cdot))$ is H\"older continuous in $X_{\beta-1}$. Then, we conclude using Lemma \ref{Lemma regularity}, by taking $\mathcal{B}(\cdot)=\tilde{\mathcal{K}}(\cdot,u(\cdot))$ and we obtain that $ u\in C^{1}((0,T],X_{\beta}) $.\\
Henceforth, we can argue similarly as in \cite[Proposition 7.1.8]{Lunardi} to prove that the solution $u$ can be extended continuously to a maximal interval $[0,T(u_0))$, where $T(u_0)>0$ is the maximal time, such that the property \eqref{ETFFormula} is also satisfied.
\end{proof}
\begin{remark}
We mention that, in \cite[Theorem 7.1.2]{Lunardi}, the result of existence of a solution (without regularity) of equation \eqref{eq:main1System APC211} uses the fractional power space $ D(\mathcal{A}_0^\alpha) $ as an intermediate space $ X_{\alpha} $. However, this fact does not affect our existence result since the proof can be given in a similar way for any intermediate Banach space between $D(\mathcal{A}_0)$ and $X$, see the proof of  \cite[Theorem 7.1.2]{Lunardi}.
\end{remark}

\subsection{Positivity}
This section aims to prove the positivity of the maximal solution of our model \eqref{eq:main1System APC21}-\eqref{initial condition} obtained so far in Section \ref{Local existence and regularity}. That is, it remains to establish that the positive cone $ X^{+} $ is positively invariant for equation \eqref{eq:main1System APC211}. For that purpose, we use the subtangential condition due to Martin \& Smith \cite{Martin-Smith}, we refer also to the monograph of Martin \cite{Martin} and to Hirsch \& Smith \cite{Hirsch-Smith}. The proof of the present section is inspired from that in \cite[Section 2]{Martin-Smith}. \\

For $ \varphi \in X_{\alpha} $, we define 
\begin{align*}
[(\lambda-\mathcal{A}_{\beta-1})  \mathcal{D} \mathcal{M}]\varphi(x)=([(\lambda-\mathcal{A}_{\beta-1})  \mathcal{D} \mathcal{M}]\varphi_{1}(x),\cdots,[(\lambda-\mathcal{A}_{\beta-1})  \mathcal{D} \mathcal{M}]\varphi_{5}(x))^{*},
\end{align*}
and, then
\begin{align*}
&\tilde{\mathcal{K}}(t,\varphi)(x):=\tilde{\mathcal{K}}(t,\varphi(x))\\=&\biggl(\mathcal{K}_1 (t,\varphi_1(x),\nabla \varphi(x))+[(\lambda-\mathcal{A}_{\beta-1})  \mathcal{D} \mathcal{M}]_{1}\varphi_{1}(x),\cdots,\mathcal{K}_5 (t,\varphi_5(x),\nabla \varphi(x))+[(\lambda-\mathcal{A}_{\beta-1})  \mathcal{D} \mathcal{M}]_{5}\varphi_{5}(x)\biggr)^*,
\end{align*}
for $ a.e. $ $ x\in \Omega$. Then, we have the following positivity result.
\begin{theorem}[Positivity]\label{Positivity}
For each $ u_0 \in X_{\alpha}^{+} $ equation \eqref{eq:main1System APC211}  has a unique maximal solution $u(\cdot ,u_{0}) \in C([0,T(u_0 )),X_{\alpha})\cap C^{1}((0,T(u_0 )),X_{\beta})$  such that $ u(t)\in  X_{\alpha}^{+} $ for all $ t\in [0,T(u_0 ) ) $.
\end{theorem}
\begin{proof} From Proposition \ref{Proposition:3.2}, it is clear that $ \mathcal{T}(t) X_{\alpha}^{+} \subset  X_{\alpha}^{+} $ for all $ t\geq 0 $.
Let $ \varphi \in  X_{\alpha}^{+} $. So from \cite[Corollary 4]{Martin-Smith} it suffices to show that  
\begin{equation}
 \lim_{h\rightarrow 0} h^{-1}d(\varphi + h \tilde{\mathcal{K}}(t,\varphi);X_{\beta-1}^{+})=0 \quad \text{ for each } t \geq 0, \label{subtangential condition}
\end{equation}
where $d(z;X_{\beta-1}^{+}):=\inf_{x \in X_{\beta-1}^{+}} \| z-x\|_{\beta -1}$ is the distance of a point $z\in X_{\beta-1}$ from the subset $X_{\beta-1}^{+}$. 
 First, we prove (pointwisely) that   
\begin{equation}
 \lim_{h\rightarrow 0} h^{-1}d(\sup_{\omega >0}\omega^{\beta}(R(\omega, \mathcal{A}_{-1}-\lambda)\left( \varphi(x) + h [\tilde{\mathcal{K}}(t,\varphi)](x)\right) ;[0,+\infty))=0 \quad \text{ for each } t \geq 0, \, a.e. \ x\in \Omega. \label{subtangential condition pointwise}
\end{equation}
Then, the formula \eqref{subtangential condition pointwise} holds since the transformation $\sup_{\omega >0} \omega^\beta R(\omega, \mathcal{A}_{-1}-\lambda) $ preserves the positivity, and due to \cite[Remark 1.2]{Martin-Smith} by the fact that 
$\mathcal{K}_{i}(t,0)\geq 0$  and $[(\omega-\mathcal{A}_{-1})  \mathcal{D} \mathcal{M}]_{i}0= 0$ for all $ t\geq 0 $ which gives that $\tilde{\mathcal{K}}_{i}(t,0)\geq 0$.
Note that the operator $ R(\omega, \mathcal{A}_{-1}-\lambda) $ is positive (see \cite{Batkai}) which yields the positivity of $\sup_{\omega >0} \omega^\beta R(\omega, \mathcal{A}_{-1}-\lambda) $. Hence, we aim to prove that \eqref{subtangential condition} holds. Let $\mid \cdot \mid_p$ be the $p$-norm in $\mathbb{R}^5$ defined as $ | (x_1,\cdots,x_5)^*|_p=(\sum_{k=1}^5 |x_i|^p)^{\frac{1}{p}}$.
Then, for $ \varphi \in X $, the norm $$ \|\varphi\|_{p}:= (\int_{x\in \Omega} \mid \varphi(x)\mid_{p} dx)^{\frac{1}{p}}$$ is equivalent to the norm on $X$. This is due the fact that all the norms in $\mathbb{R}^5$ are equivalent. Similarly, using this new norm $ \| \cdot \|_p$, we can define an associated equivalent norm for $X_{\alpha}$, and the new equivalent norm on $ X_{\beta-1} $ which is given by $$\| \cdot\|_{\beta-1,p}= \sup_{\omega >0}\| \omega^{\beta} R(\omega,\lambda-\mathcal{A}_{-1})\cdot\|_{p}.$$
 Let us define the Euclidean projection onto $[0,+\infty)$, $  \pi_{\Lambda}:\mathbb{R}^{5}\longrightarrow \Lambda_{0,+\infty}$, by 
$$ \mid x -  \pi_{\Lambda} x\mid=d(x,[0,+\infty)).  $$
Notice that the mapping $ \pi_{\Lambda} $ is well-defined and continuous on $ \mathbb{R}^n$ (eventually it is $1$-Lipschitz continuous).
Let $ \varepsilon >0$ and let $h>0$. Define 
$$ \varphi_{h}(x):= \pi_{\Lambda} (\varphi(x) + h [\tilde{\mathcal{K}}(t,\varphi)](x) ) \quad \text{for } t \geq 0, \, x\in \Omega.$$
So, $  \varphi_{h} \in X_{\beta-1}^{[0,+\infty)} $ and
\begin{align}
d(\varphi + h \tilde{\mathcal{K}}(t,\varphi);X_{\beta-1}^{[0,+\infty)})^{p} &\leq 
\| \varphi + h \tilde{\mathcal{K}}(t,\varphi)-\varphi_{h} \|^p_{\beta-1} \nonumber \\
&\leq \sup_{\omega >0}\omega^{\beta}  \int_{x\in \Omega}\mid   R(\omega,\lambda-\mathcal{A}_{-1})\left( \varphi(x) + h [\tilde{\mathcal{K}}(t,\varphi)](x)-\varphi_{h}(x)\right)  \mid_{p}^{p}dx \nonumber\\
& = \int_{x\in \Omega}  d(\sup_{\omega >0} \omega^{\beta} R(\omega,\lambda-\mathcal{A}_{-1})\left(\varphi(x) + h [\tilde{\mathcal{K}}(t,\varphi)](x)\right) ;\Lambda_{0,+\infty} )^p dx \label{Estim distance}
\end{align}
Moreover, for $ 0<h \leq \delta $, for some $\delta>0$, it follows in view of the convexity of the operator distance, by the continuity of $\tilde{\mathcal{K}}$ and using  \eqref{subtangential condition pointwise}, that 
 $$  \int_{\Omega} d(\sup_{\omega >0}\omega^{\beta}(R(\omega, \mathcal{A}_{-1}-\lambda)\left(  \varphi(x) + h [\tilde{\mathcal{K}}(t,\varphi)](x)\right) ;[0,+\infty) )^{p} dx\leq |\Omega | (h \varepsilon)^{p} . $$
Hence, by \eqref{Estim distance} we have $$ d(\varphi + h \tilde{\mathcal{K}}(t,\varphi);X_{\beta-1}^{+}) \leq |\Omega|^{\frac{1}{p}} h \varepsilon \quad \text{for all }  0<h \leq \delta ,$$ 
which proves the result.
\end{proof}
\subsection{$L^1$-boundedness}
In this Section, we show the $ L^1 $-boundedness of the total population density of our model \eqref{eq:main1System APC21}-\eqref{initial condition}. Let $ u_{0} \in X_{\alpha}^{+} $ and let $u(t,u_0)=(\rho_1(t,\cdot),\cdots,\rho_5(t,\cdot))^{*}$ for all $ t\in [0,T(u_0)) $ be the corresponding maximal solution. It is clear, from Theorem \ref{Theorem maximal existence}, that $u(\cdot,u_0) \in X_{\alpha}\hookrightarrow C^{1}(\overline{\Omega})^5\hookrightarrow L^1(\Omega)^5$. This means that each $\rho_i$ for $ i=1,\cdots,5 $, is bounded with respect to $\Omega$ and then it is $L^1(\Omega)$. Hence, the following mapping: 
$$ t\in [0,T(u_0)) \longmapsto U(t):= \int_{\Omega} \left[ \rho_1(t,x)+\cdots+\rho_5(t,x)\right] dx \in \mathbb{R}, $$ 
is well-defined.
Furthermore, we have
\begin{proposition}[$L^1$-boundedness]

$$ 0\leq U(t) \leq 1 \quad \text{ for all   } t\in  [0,T(u_0 ) ).$$
\end{proposition}
\begin{proof}
By assumption on the initial condition, we have $$ U(0)=\int_{\Omega} \theta (x) dx=1 .$$
In otherwise, the positivity result in Theorem \ref{Positivity} gives $$ U(t) \geq 0  \quad \text{ for all   } t\in [0,T(u_0 ) ) .$$ 
Moreover, the mapping $ U $ is well-defined and it is continuously differentiable on $[0,T(u_0)) $. Hence, we show that 
 $$ \dfrac{d}{dt}U(t) \leq 0, \quad t \in [0,T(u_0 ) ) . $$
Indeed, 
using the Green-Ostrogradski formula, we obtain that 
\begin{align*}
\dfrac{d}{dt}U(t)&=\int_{\Omega} \partial_t \left[\rho_1+\rho_2+\rho_3+\rho_4+\rho_5 \right]  dx \\ &=\sum_{i=1}^{5} d_i\int_{\Omega}   \Delta \rho_i  dx - \int_{\Omega}  \nabla(\rho_2 \vec{v}_2(\rho))dx-\int_{\Omega}  \nabla(\rho_3 \vec{v}_3(\rho)) dx - \sum_{i=1}^{3} \delta_i \int_{\Omega}   \rho_i  \\
 &=\sum_{i=1}^{5} d_i\int_{\partial \Omega}  \nabla \rho_i \cdot n d \sigma(x) - \int_{\partial\Omega}  \rho_2 \vec{v}_2(\rho)\cdot n d\sigma(x)-\int_{\partial \Omega}  \rho_3 \vec{v}_3(\rho)\cdot n  \sigma(x) - \sum_{i=1}^{3} \delta_i \int_{\Omega}   \rho_i  dx\\
 &=-\sum_{i=1}^{5} v_{i,out}\int_{\Gamma_{2}}   \rho_i d\sigma(x) - \sum_{i=1}^{3} \delta_i \int_{\Omega}   \rho_i  dx \leq 0, \quad t\in [0,T(u_0 ) ).
\end{align*}
The last estimate is a consequence of the positivity of the terms $\rho_i$, $i=1,\cdots,5$. That is, $ U $ is decreasing and then, we have
$$ U(t) \leq U(0)=1 \quad \text{ for all   } t\in  [0,T(u_0 ) ).$$
\end{proof}

\section{Numerical Simulations}\label{sec:4} 
In this section, we present several numerical simulations for different population evacuation scenarios during a catastrophic event.  In order to highlight the behavior of populations during such an event, we study the case where there is no return to everyday life, which corresponds to the case where $$\phi(t)=0 \quad \text{ for all } t\geq 0, $$ and, for simplicity, we take $\gamma (t)=1$ for all $t\geq 0$. In this case the system \eqref{eq:main1System APC211} is not time-dependent (it is autonomous) and all the results of Section \ref{sec:3} still hold. \\
Here, the population is assumed to have a low risk culture, which means that most people tend to panic. All the parameters of the spatio-temporal APC model \eqref{eq:main1System APC21}-\eqref{initial condition} are set as in Table \ref{tab1}. 

\begin{table}[h!]
\begin{center}
\begin{tabular}{|R{2.5cm}||C{3.5cm}|}
\hline & \textbf{Parameters}  \\
\cline{2-2}              & $d_1=0.001$  \\
\cline{2-2}\textbf{Diffusion} & $d_2=0.05 $ \\
\cline{2-2}              & $d_3=0.01$  \\
\cline{2-2}              & $d_4= 0.01$ \\
\hline 
\cline{2-2}                     & $V_{2,max}=0.3$  \\
\cline{2-2} \textbf{Advection}     & $V_{3,max}= 0.2$ \\
\hline 
\cline{2-2}                       & $v_{1,out}=0.2$  \\
\cline{2-2} \textbf{Speed at the}   & $v_{2,out}=  0.1 $ \\
\cline{2-2}   \textbf{boundary }  & $v_{3,out}=0.3$  \\
\cline{2-2}                & $v_{4,out}= 0.2$ \\
\hline                                  
\end{tabular}
\begin{tabular}{|R{2.5cm}||C{3.5cm}|}
\hline & \textbf{Parameters}  \\
\cline{2-2}                 & $\alpha_{13} =0.6$  \\
\cline{2-2}  \textbf{Imitation}  & $\alpha_{12} =0.7 $ \\
\cline{2-2}                 & $\alpha_{23}=0.6$  \\
\cline{2-2}                 & $\alpha_{32}= 0.7$ \\
\hline 
\cline{2-2}               & $c_1=0.1$  \\
\cline{2-2}               & $c_2= 0.4$ \\
\cline{2-2} \textbf{Intrinsic}  & $b_1=0.1$  \\
\cline{2-2}  \textbf{transitions}  & $b_2= 0.2$ \\
\cline{2-2}               & $b_3=0.001$  \\
\cline{2-2}               & $b_4= 0.001$ \\
\hline                                  
\end{tabular}
\caption{Table of parameter values. Here, we choose $d_1 <d_3=d_4 <d_2$, since the population in an alert state scarcely diffuse and the panic population is supposed to be the most diffusive one. Moreover, we are interested to consider a population with a low risk culture, so, for example, we take $c_2 >c_1$ and $b_2>b_1$, as in \cite{LanzaFede}.}
\label{tab1}
\end{center}
\end{table}

With regard to the diffusion process, we suppose that the crowd in a alert state hardly diffuses, since in a alert behavior, pedestrians are moving to look for information and to identify the hazard. Thus, the diffusion coefficient $d_1$ should be considered small compared to the ones of the other populations. Moreover, here it is assumed that the most diffusive population is the panic population, since in panic behavior pedestrians move randomly in different directions, and thus $d_2$ should be considered the largest diffusion coefficient.

Thus, for the diffusivity coefficients of the control, daily, and return-to-day behaviors, we assume $d_1<d_i <d_2$ for $i=3,4,5$.

We assume that the domain presents a target escape region (denoted by $\Gamma_{2}$), thus the desired direction vector is defined by \eqref{Desired direction} where $\vec{\nu}(x)_{|\Omega}=(\nu_{x_1} , \nu_{x_2} )'$ is given by:  
 \begin{equation}\label{Direction of the mouvment nu}
            \left\{
                 \begin{array}{l}
               \nu_{x_1} = -\dfrac{x_1-x_1^{p}}{\sqrt{(x_1-x_1^p)^{2}+(x_2-x_2^p)^{2}}},\\
                 \nu_{x_2} = -\dfrac{x_2-x_2^{p}}{\sqrt{(x_1-x_1^p)^{2}+(x_2-x_2^p)^{2}}},
                 \end{array}\right.
 \end{equation}
where $ (x_1^p,x_2^p) $ is a centered point in $\Gamma_{2}$ localized out from $\overline{\Omega}$, see Figure \ref{Direction}. So, the desired direction in both situations of control and panic are supposed to be the same which is given by $\vec{\nu}(x)$. For more details about the desired direction of pedestrians, we refer to the works of Hughes \cite{Hughes2002} and also to the references \cite{Coscia,Inria}.

\begin{center}
\begin{figure}[h!]
\includegraphics[width=7cm,height=4cm]{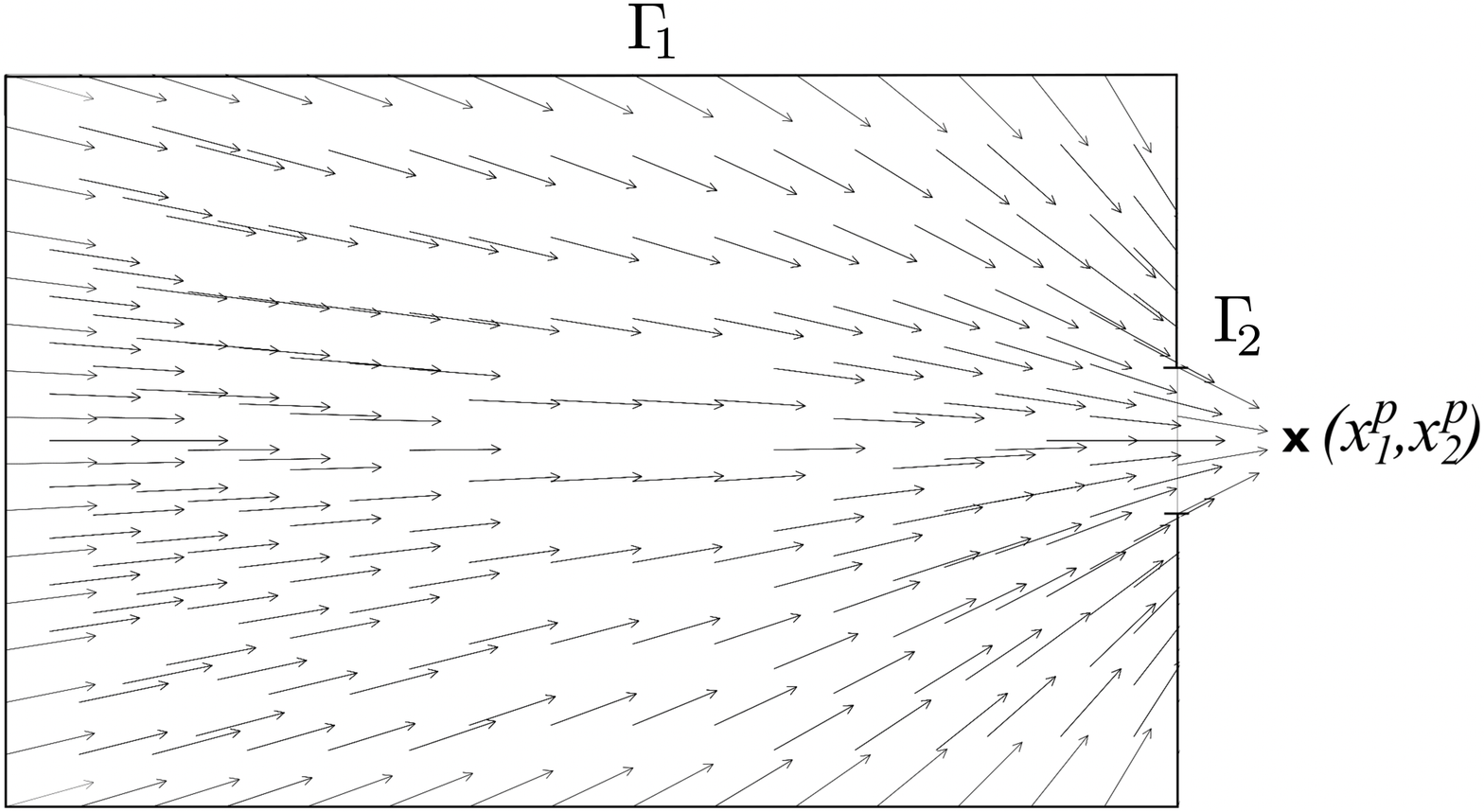}
\caption{The direction vector $\vec{\nu}(x_1,x_2)$ given in \eqref{Direction of the mouvment nu} describing the desired direction of pedestrians to reach the point $(x_1^p,x_2^p)$ which is located outside the domain $\overline{\Omega}$ since the population looks to escape from the exit $ \Gamma_{2} $ towards this point.}
\label{Direction}
\end{figure}
\end{center}

\begin{figure}[h!]
\begin{minipage}[c]{.3\linewidth}
\includegraphics[width=3.5cm,height=2cm]{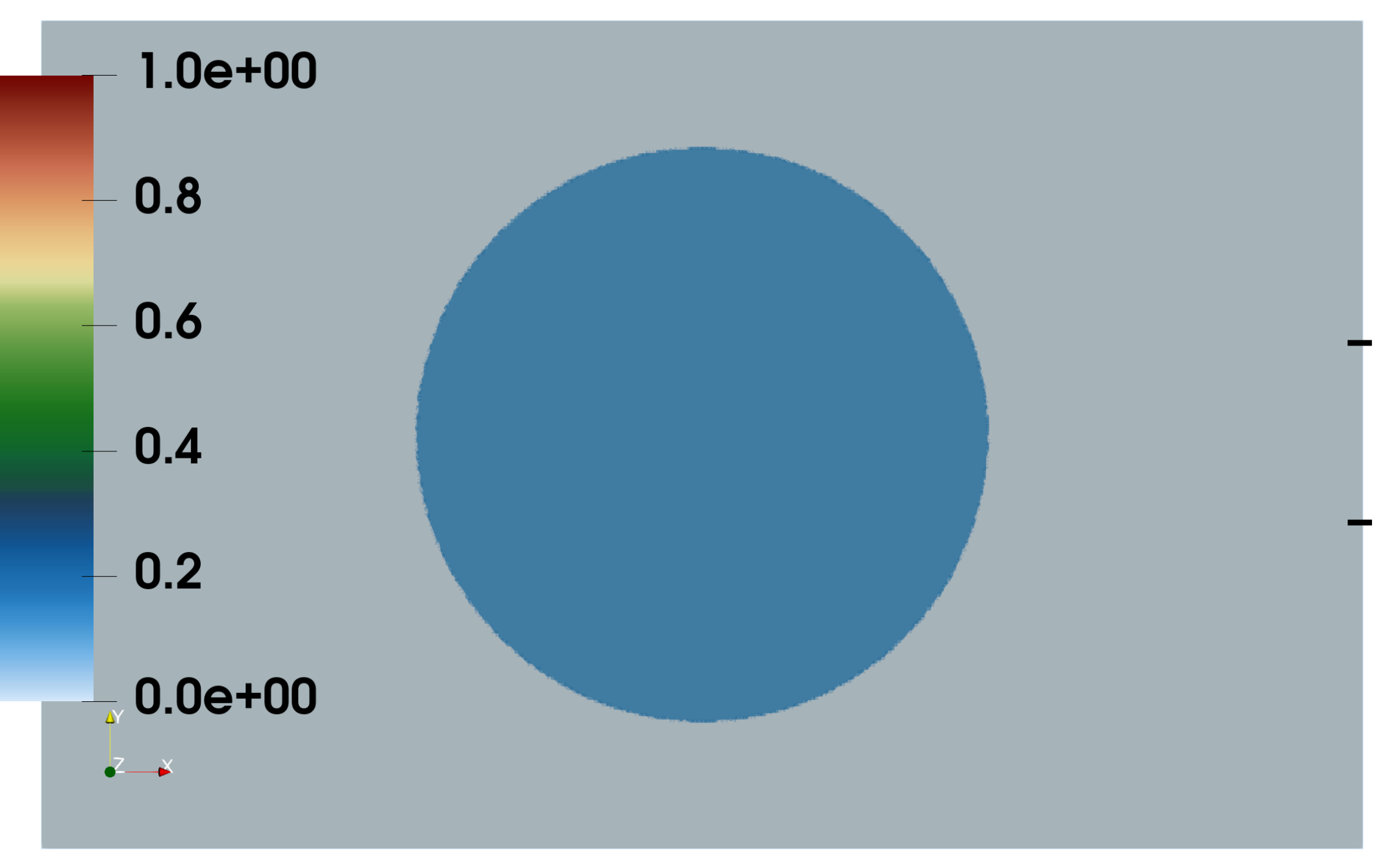}\\
\begin{center}
(a) Scenario 1
\end{center}
\end{minipage} 
\begin{minipage}[c]{.3\linewidth}
\includegraphics[width=3.5cm,height=2cm]{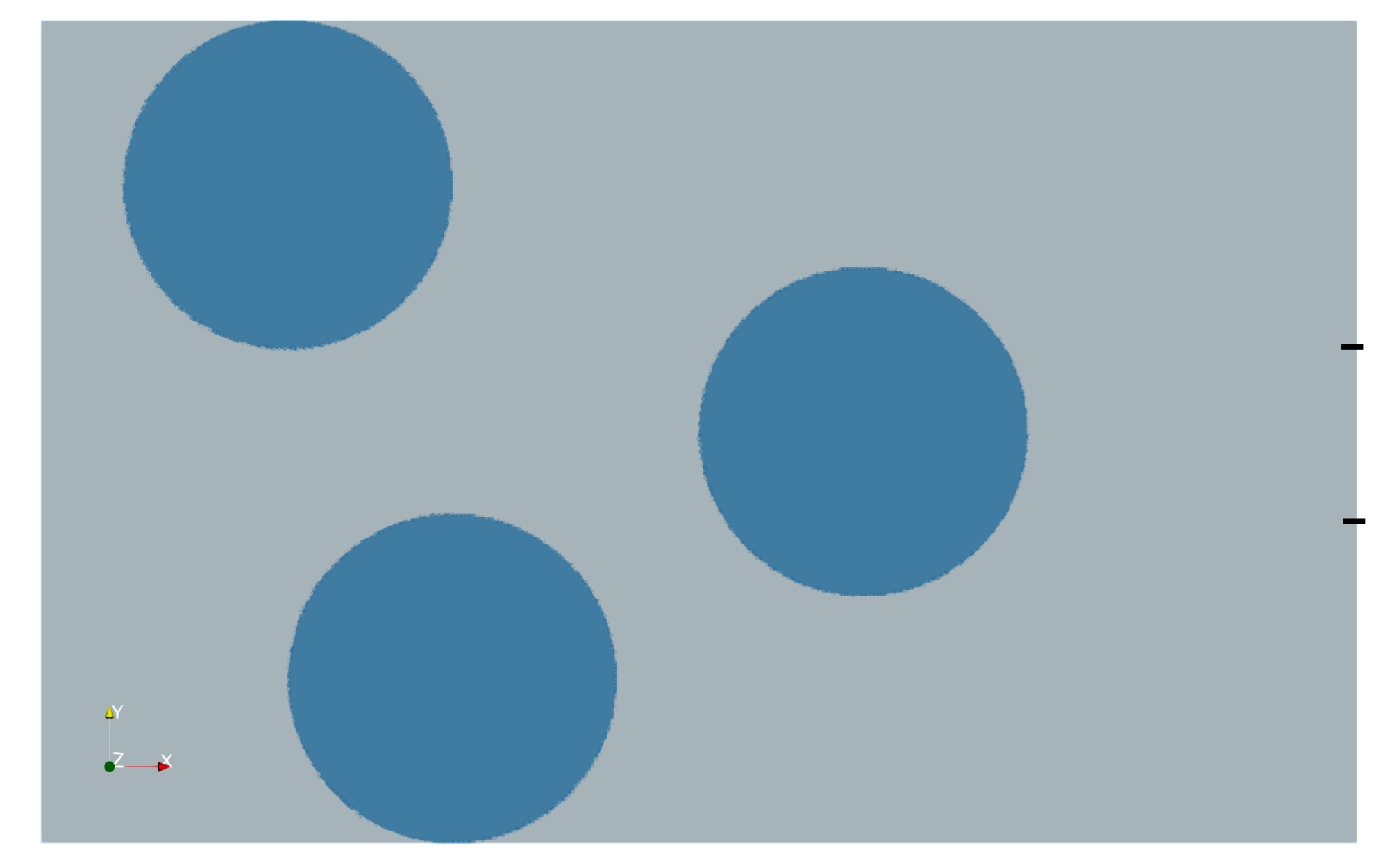}\\
\begin{center}
(b) Scenario 2
\end{center}
\end{minipage}
\begin{minipage}[c]{.3\linewidth}
\includegraphics[width=3.5cm,height=2cm]{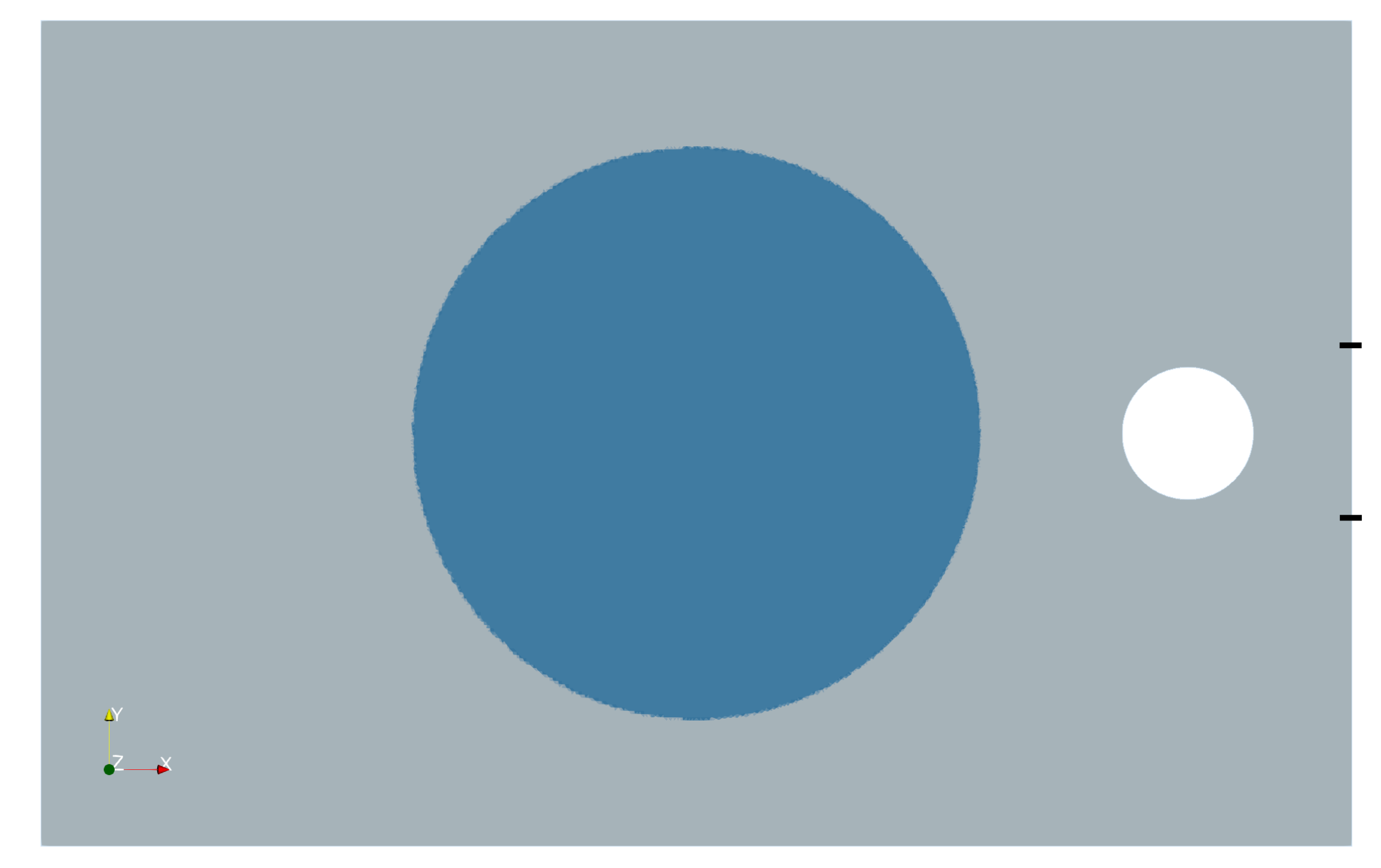}\\
\begin{center}
(c) Scenario 3
\end{center}
\end{minipage}
\caption{\small{\textbf{Initial conditions: }initial location of the population for each scenario: (a) the population is concentrated in a single group in the center of the domain; (b) the population is subdivided into three groups; (c) an obstacle is located between the exit and the population, which is concentrated in a single group within the domain. We recall that the exit is on the right of the domain, see Figure \ref{Direction}. }}
\label{The scenarios}
\end{figure}

In the following, we consider three different scenarios for the evacuation of a population whose aim is to escape by the unique exit $\Gamma_{2}$ (see Figure \ref{The scenarios}):\\ 
\textbf{Scenario 1: Evacuation of one centered cluster population.}
Here we consider a pedestrians population located in a single group within the domain.\\
\textbf{Scenario 2: Evacuation of a population subdivided into three groups.}
Here we consider a population subdivided into three separated groups of pedestrians in different spatial localizations.\\
\textbf{Scenario 3: Evacuation of a population with an obstacle in front of the exit.} Here we take into account the situation in which an obstacle is located between the exit and the population concentrated in a single group in the center of the domain.\\


 First of all, in order to highlight the time evolution of the different human behaviors, namely, alert, panic, control and daily behaviors, in  Figure \ref{Others behaviors} we present  the simulation results of Scenario 1, since the dynamics is analogous in the other scenarios.
The more the color goes from light blue to dark red, the higher the population
density is.

\begin{figure}[h!]
\begin{minipage}[c]{.19\linewidth}
\begin{center}
$$t=50$$ $$(a_1)$$ 
\end{center}
\includegraphics[width=3cm,height=2cm]{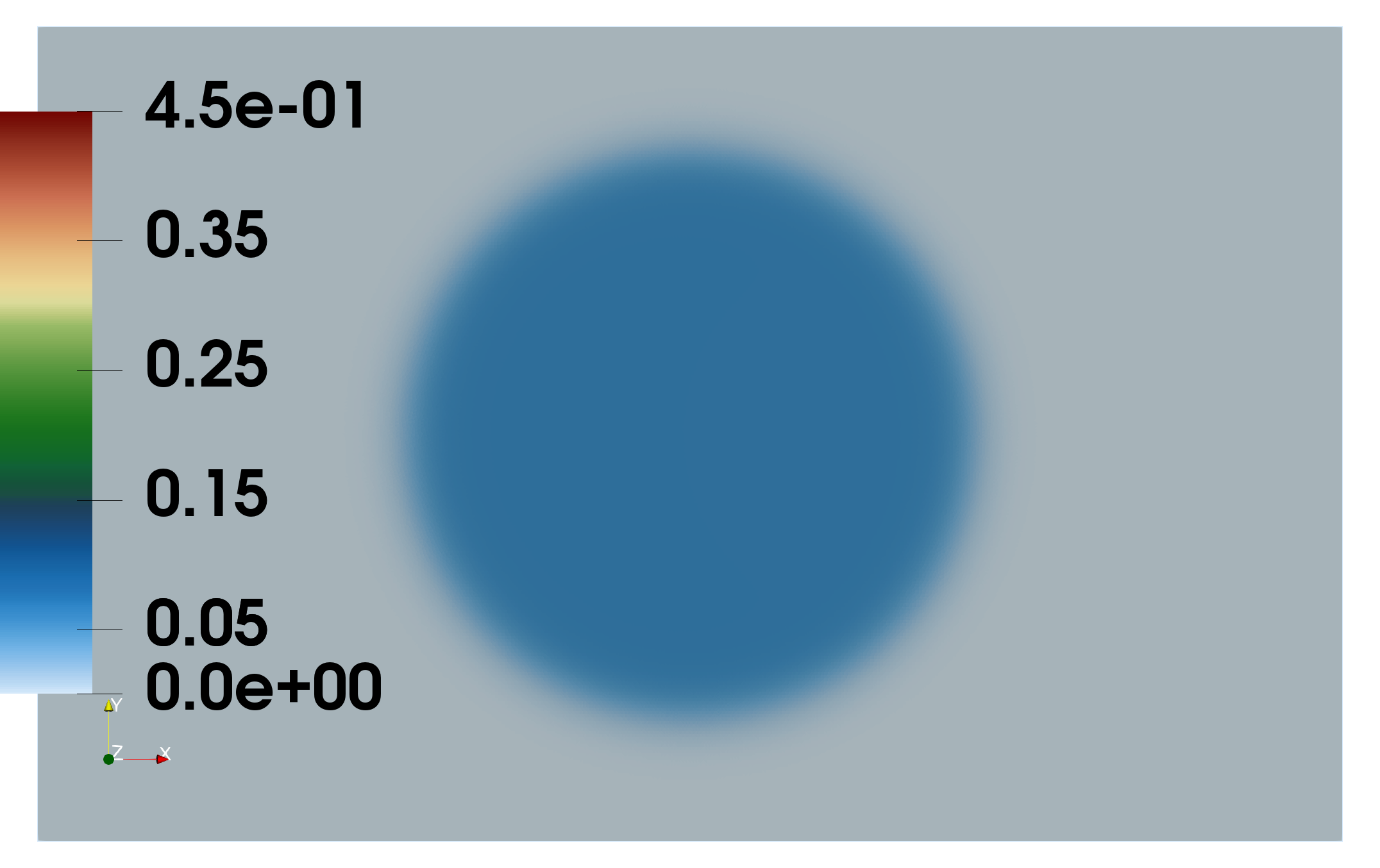}
\end{minipage}
\begin{minipage}[c]{.19\linewidth}
\begin{center}
$$t=100$$ $$(b_1)$$ 
\end{center}
\includegraphics[width=3cm,height=2cm]{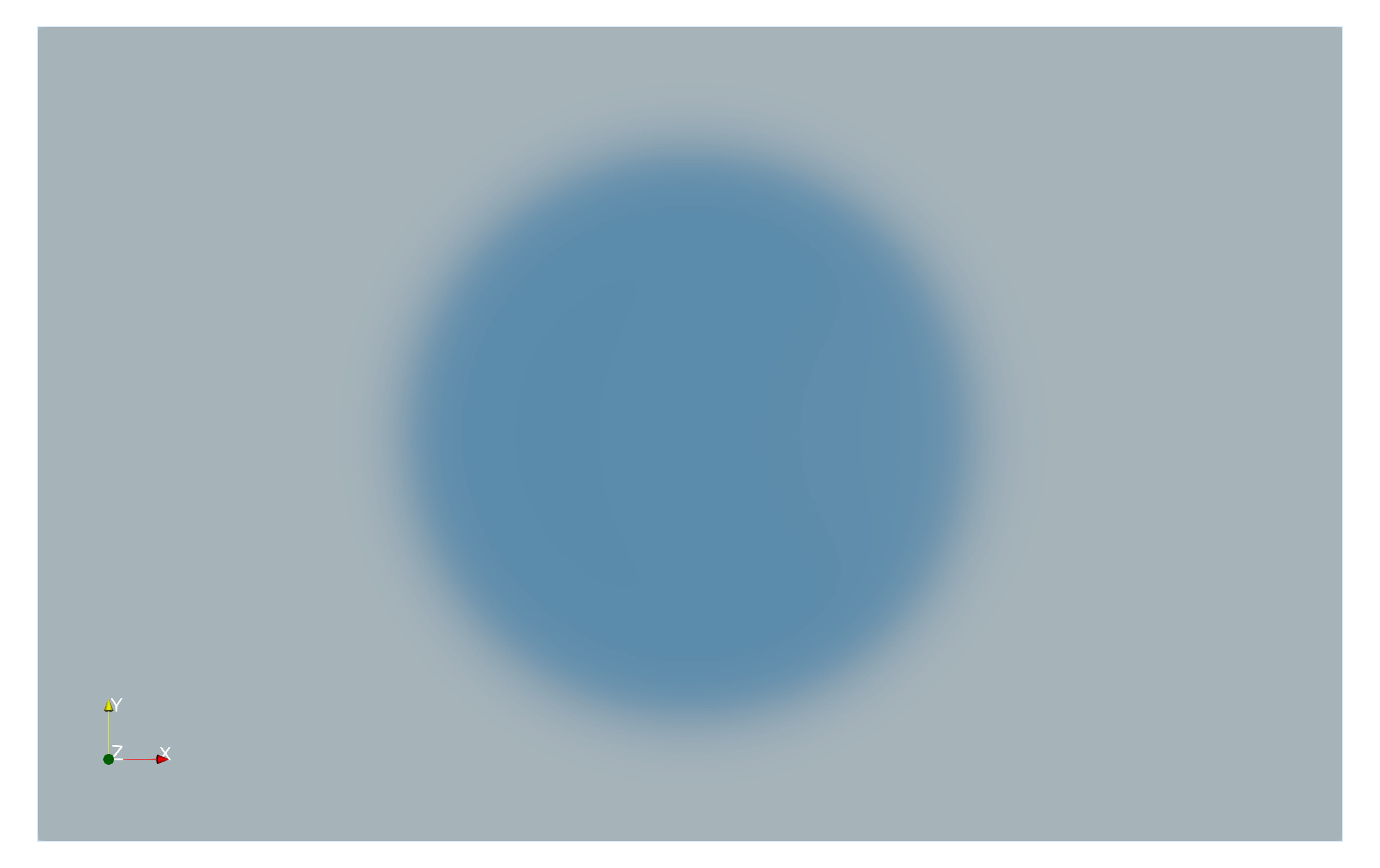}
\end{minipage}
\begin{minipage}[c]{.19\linewidth}
\begin{center}
$$t=150$$ $$(c_1)$$ 
\end{center}
\includegraphics[width=3cm,height=2cm]{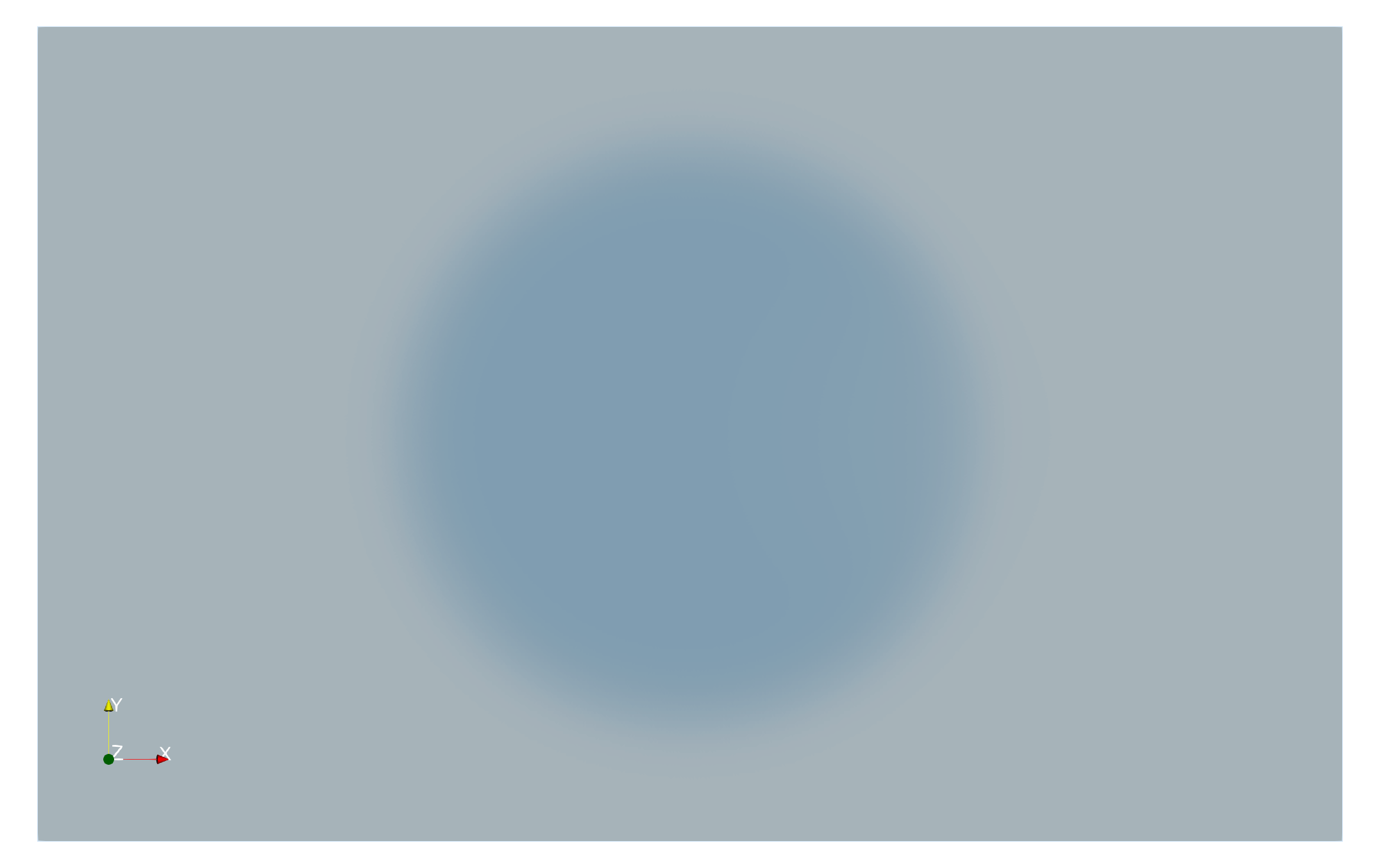}
\end{minipage}
\begin{minipage}[c]{.19\linewidth}
\begin{center}
$$t=200$$ $$(d_1)$$ 
\end{center}
\includegraphics[width=3cm,height=2cm]{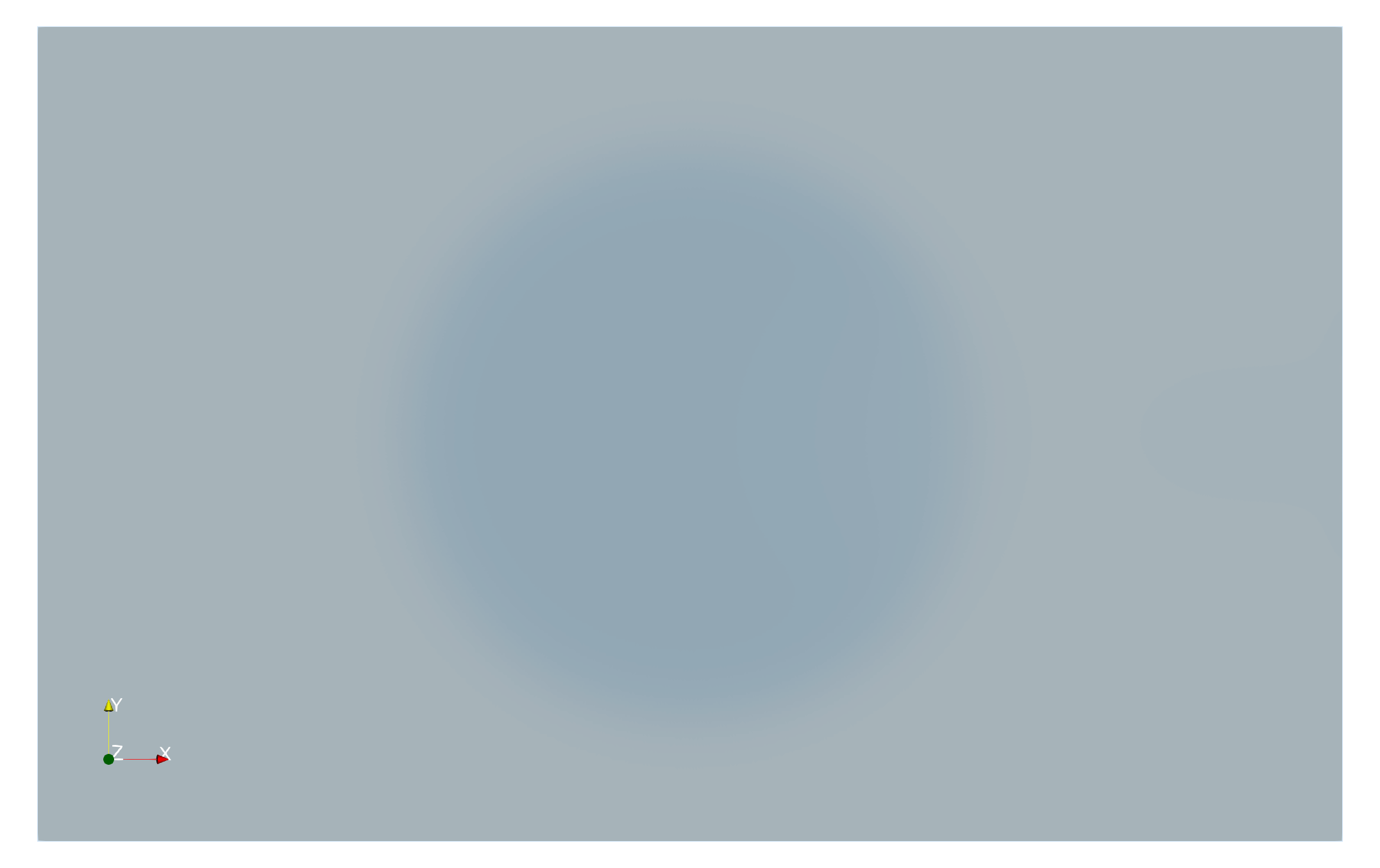}
\end{minipage}
\begin{minipage}[c]{.19\linewidth}
\begin{center}
$$t=250$$ $$(e_1)$$ 
\end{center}
\includegraphics[width=3cm,height=2cm]{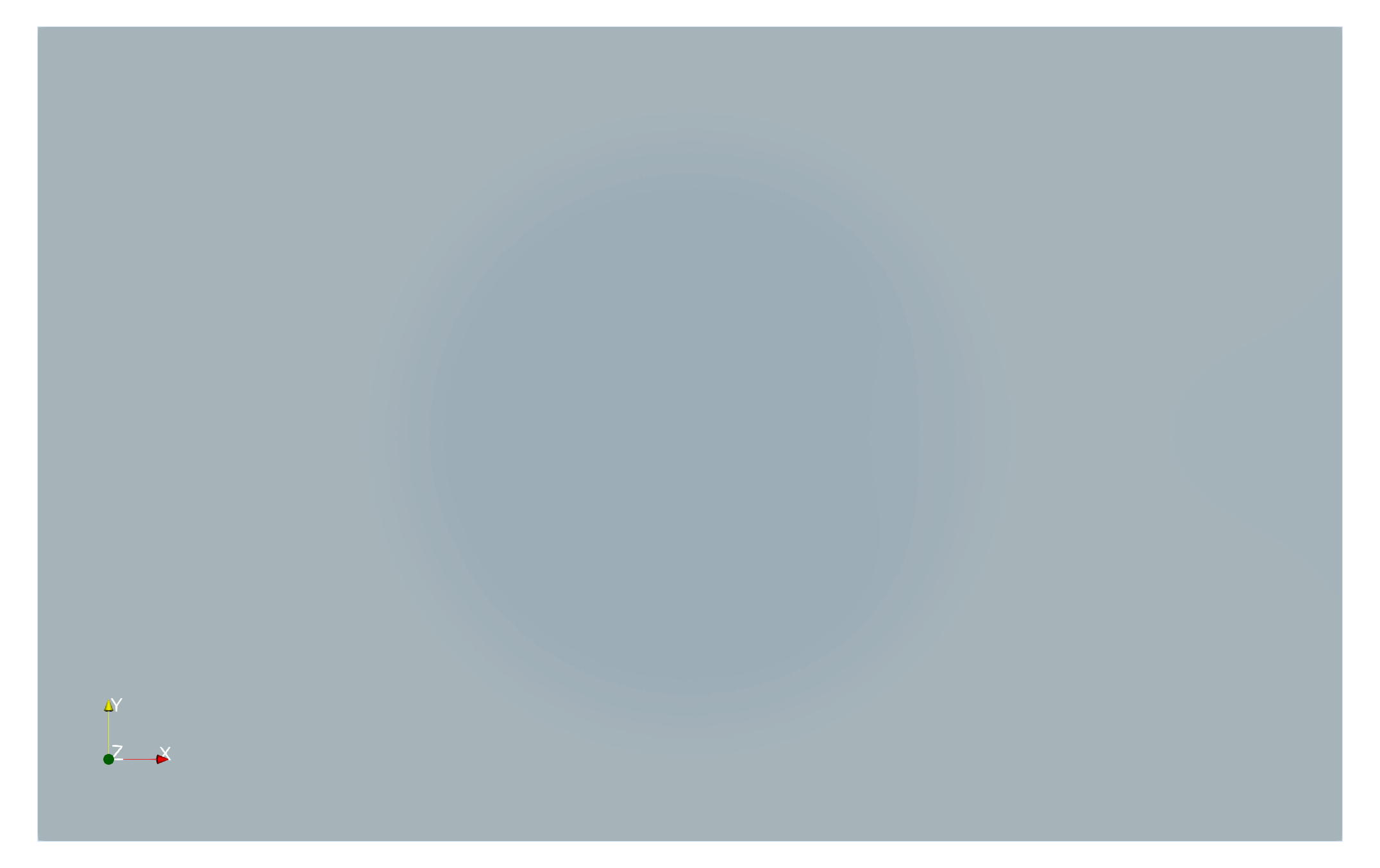}
\end{minipage}
Population in alert $\rho_1$
\begin{center}
\begin{minipage}[c]{.19\linewidth}
$$(a_2)$$ 
\includegraphics[width=3cm,height=2cm]{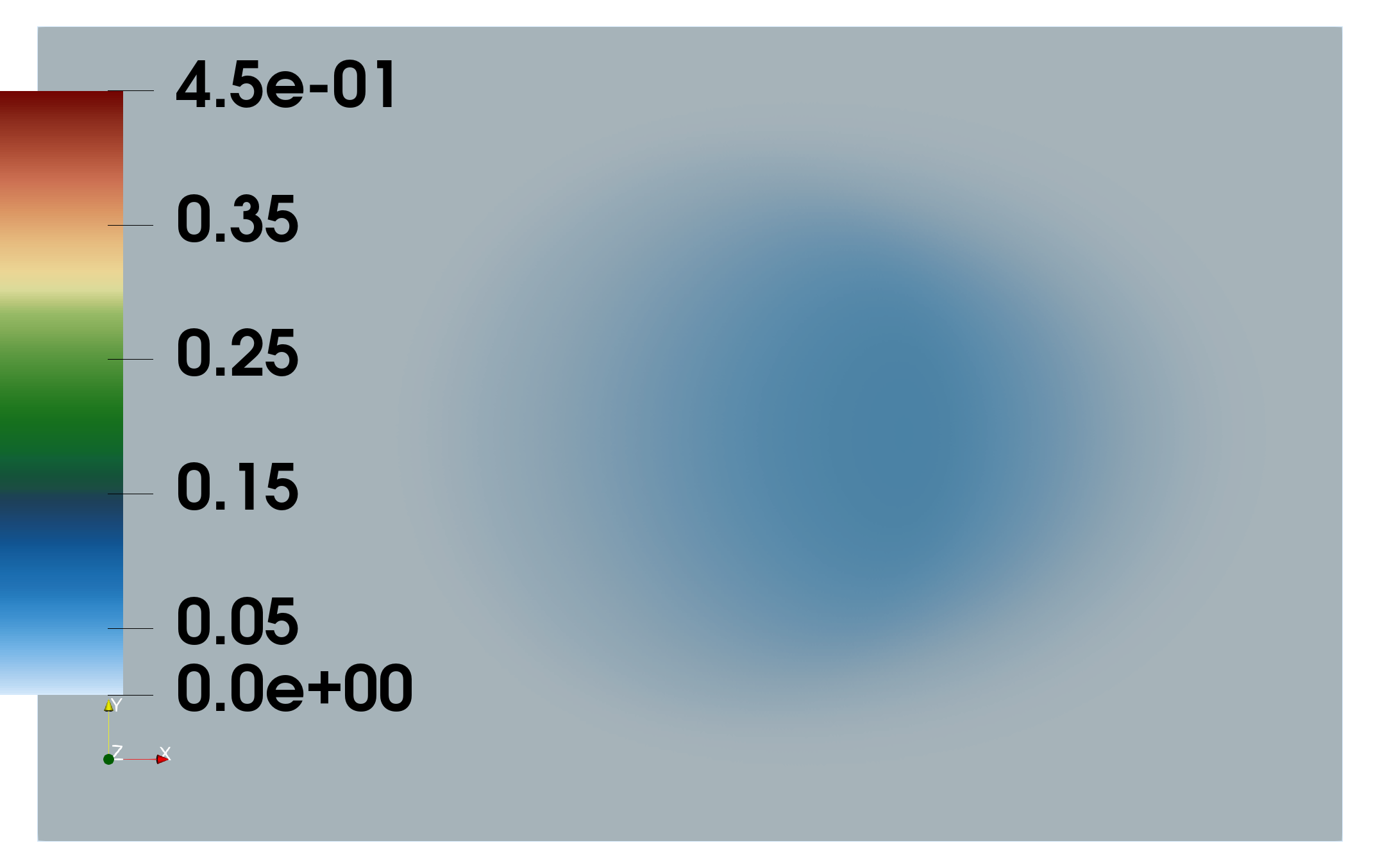}
\end{minipage}
\begin{minipage}[c]{.19\linewidth}
$$(b_2)$$ 
\includegraphics[width=3cm,height=2cm]{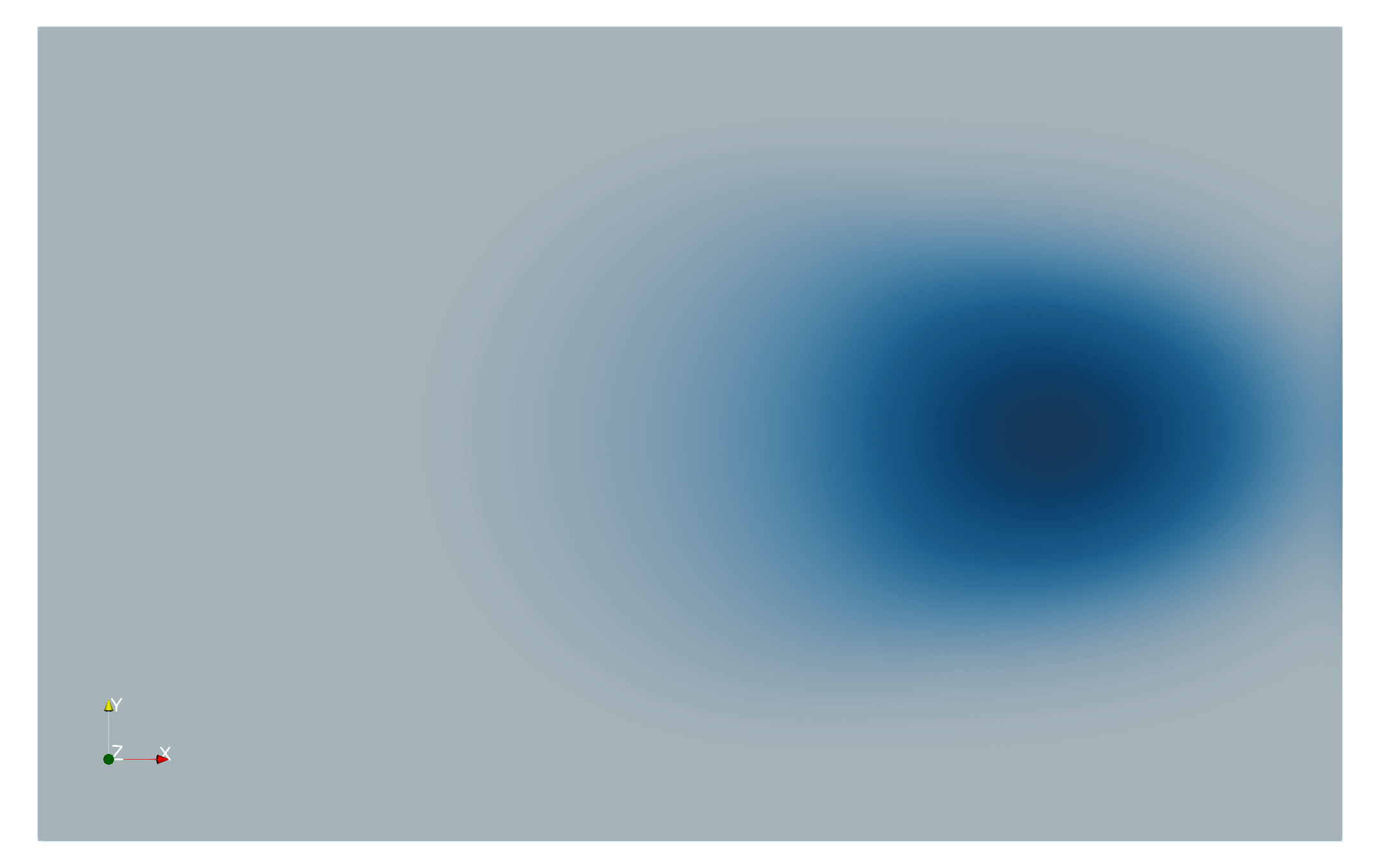}
\end{minipage}
\begin{minipage}[c]{.19\linewidth}
$$(c_2)$$ 
\includegraphics[width=3cm,height=2cm]{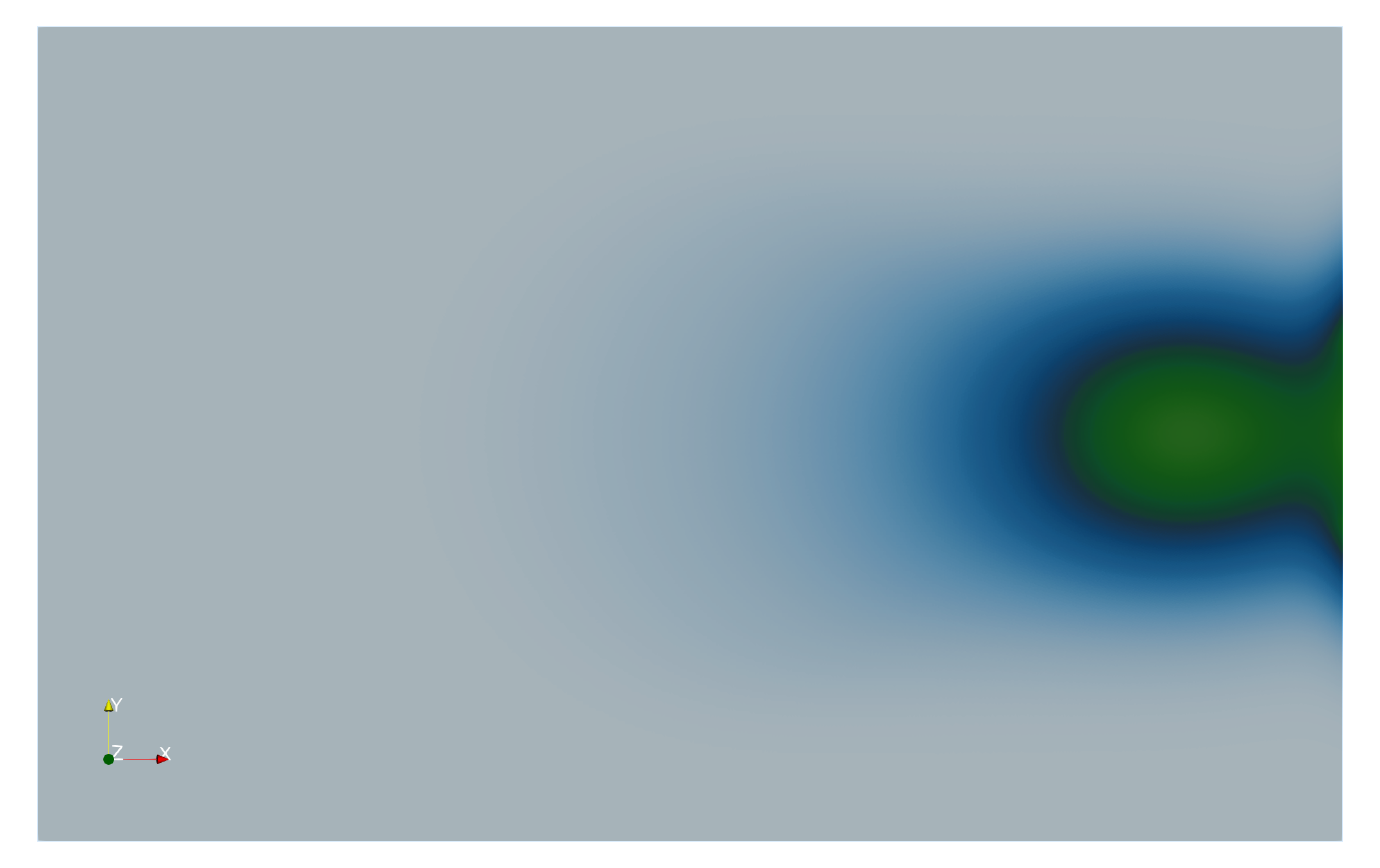}
\end{minipage}
\begin{minipage}[c]{.19\linewidth}
$$(d_2)$$ 
\includegraphics[width=3cm,height=2cm]{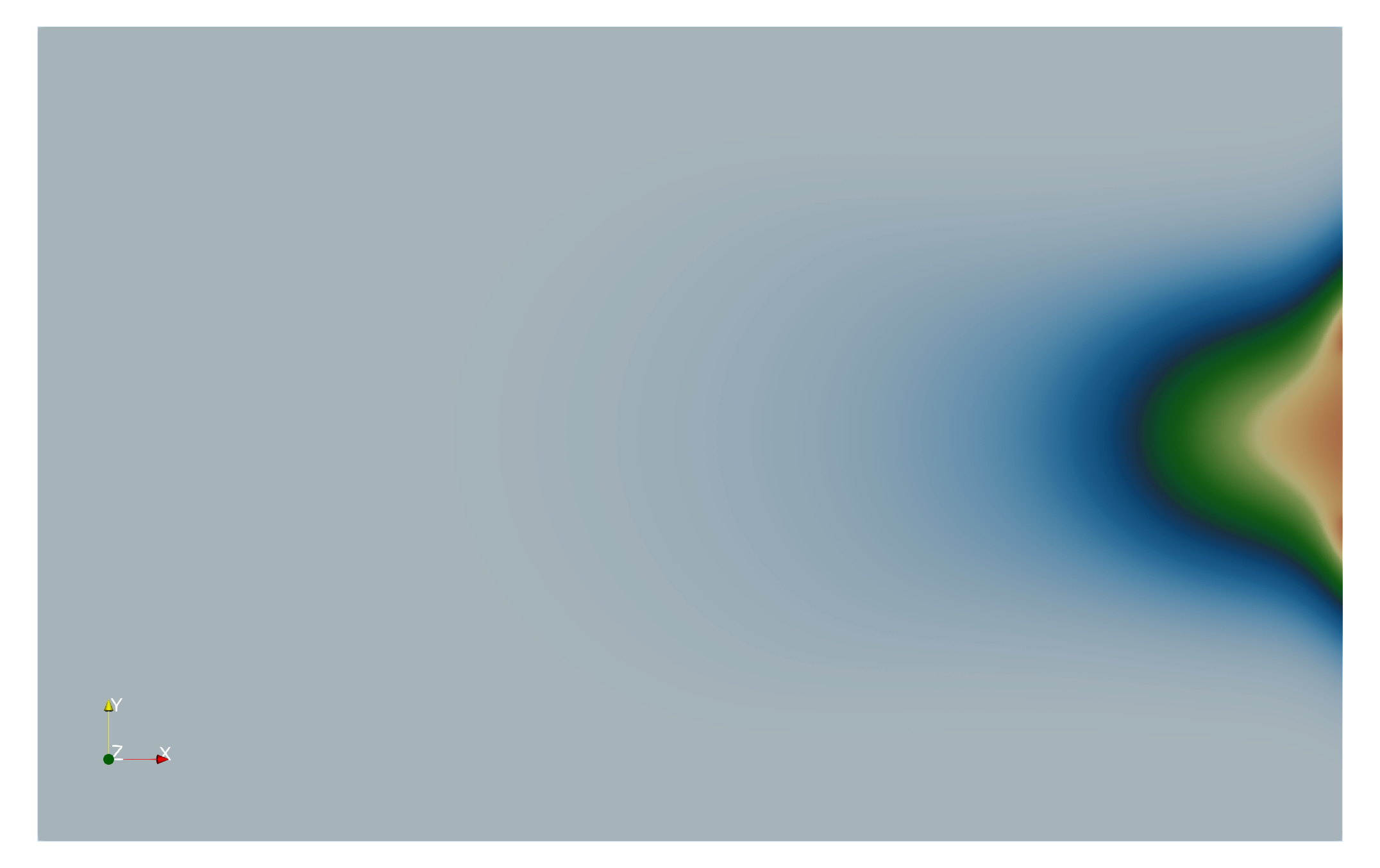}
\end{minipage}
\begin{minipage}[c]{.19\linewidth}
$$(e_2)$$ 
\includegraphics[width=3cm,height=2cm]{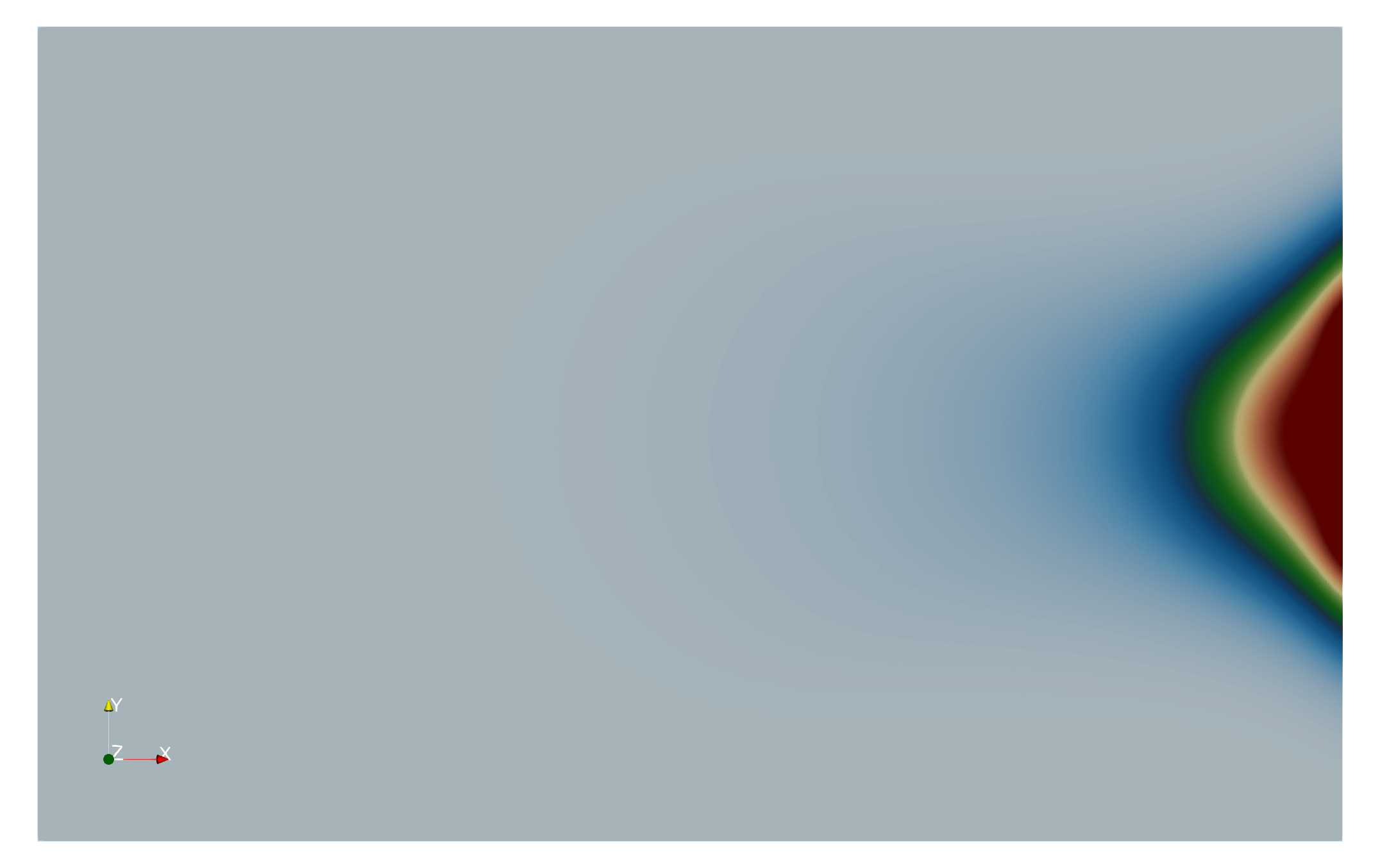}
\end{minipage}
Population in panic $\rho_2$
\end{center}
\begin{center}
\begin{minipage}[c]{.19\linewidth}
$$(a_3)$$ 
\includegraphics[width=3cm,height=2cm]{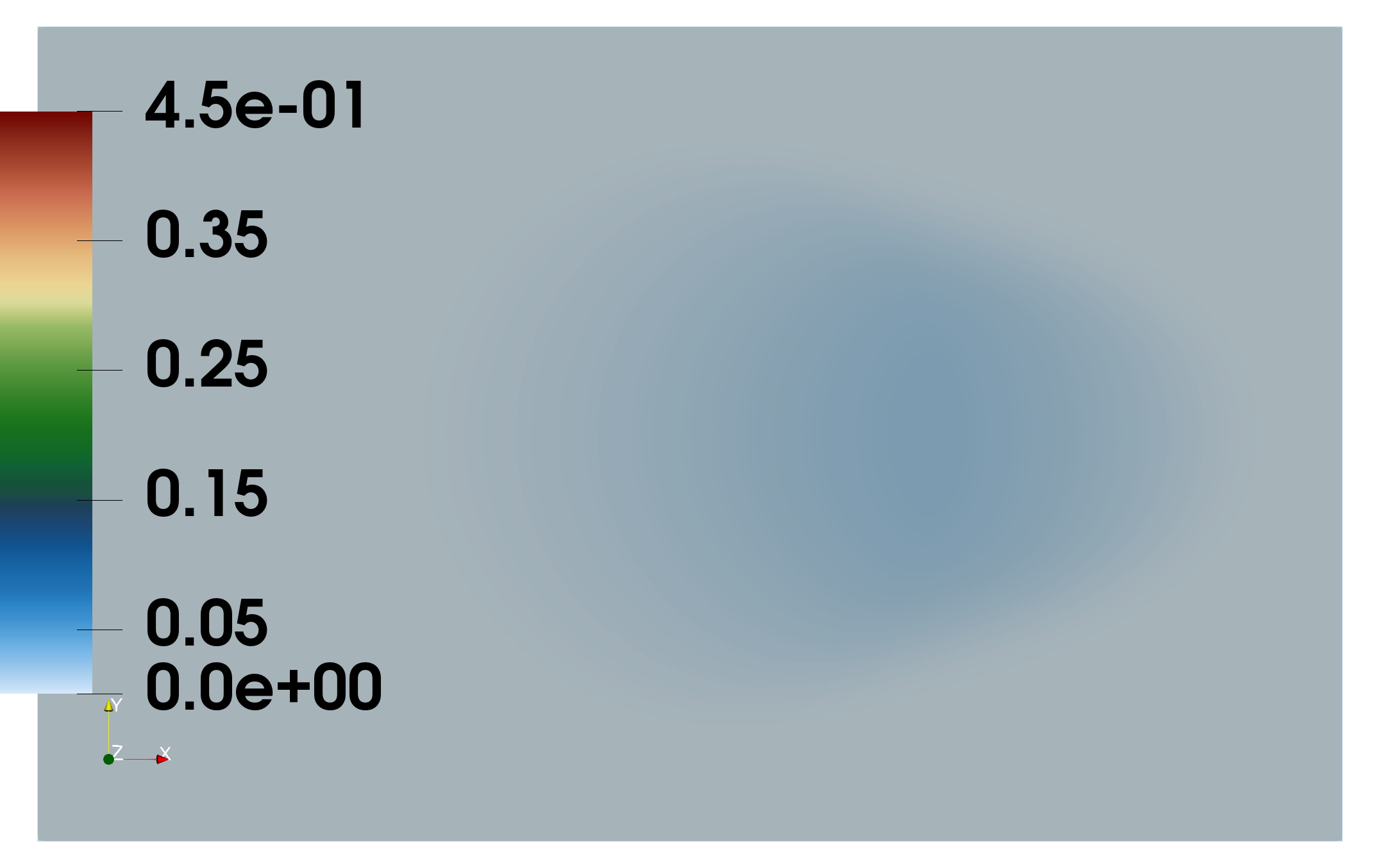}
\end{minipage}
\begin{minipage}[c]{.19\linewidth}
$$(b_3)$$ 
\includegraphics[width=3cm,height=2cm]{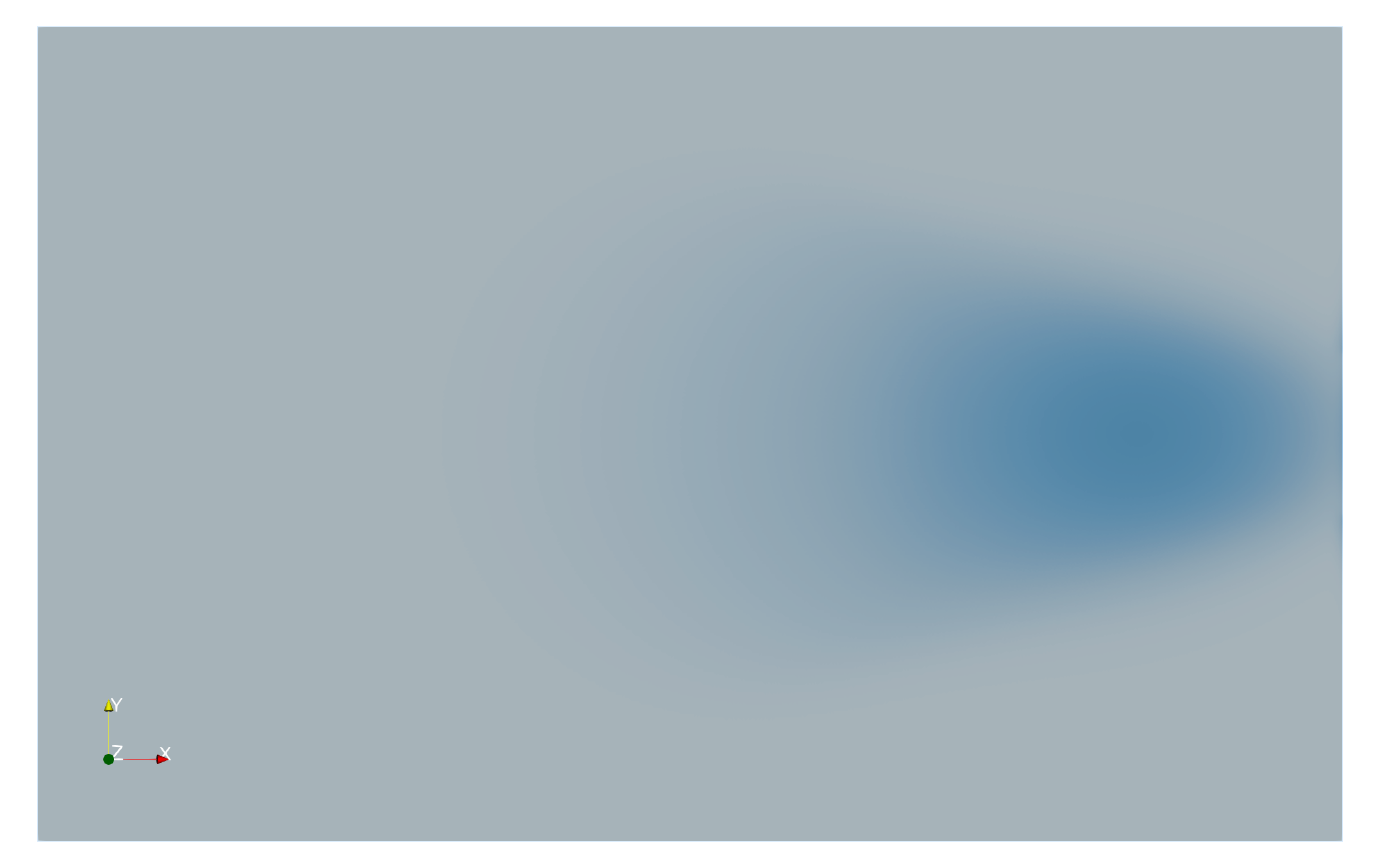}
\end{minipage}
\begin{minipage}[c]{.19\linewidth}
$$(c_3)$$ 
\includegraphics[width=3cm,height=2cm]{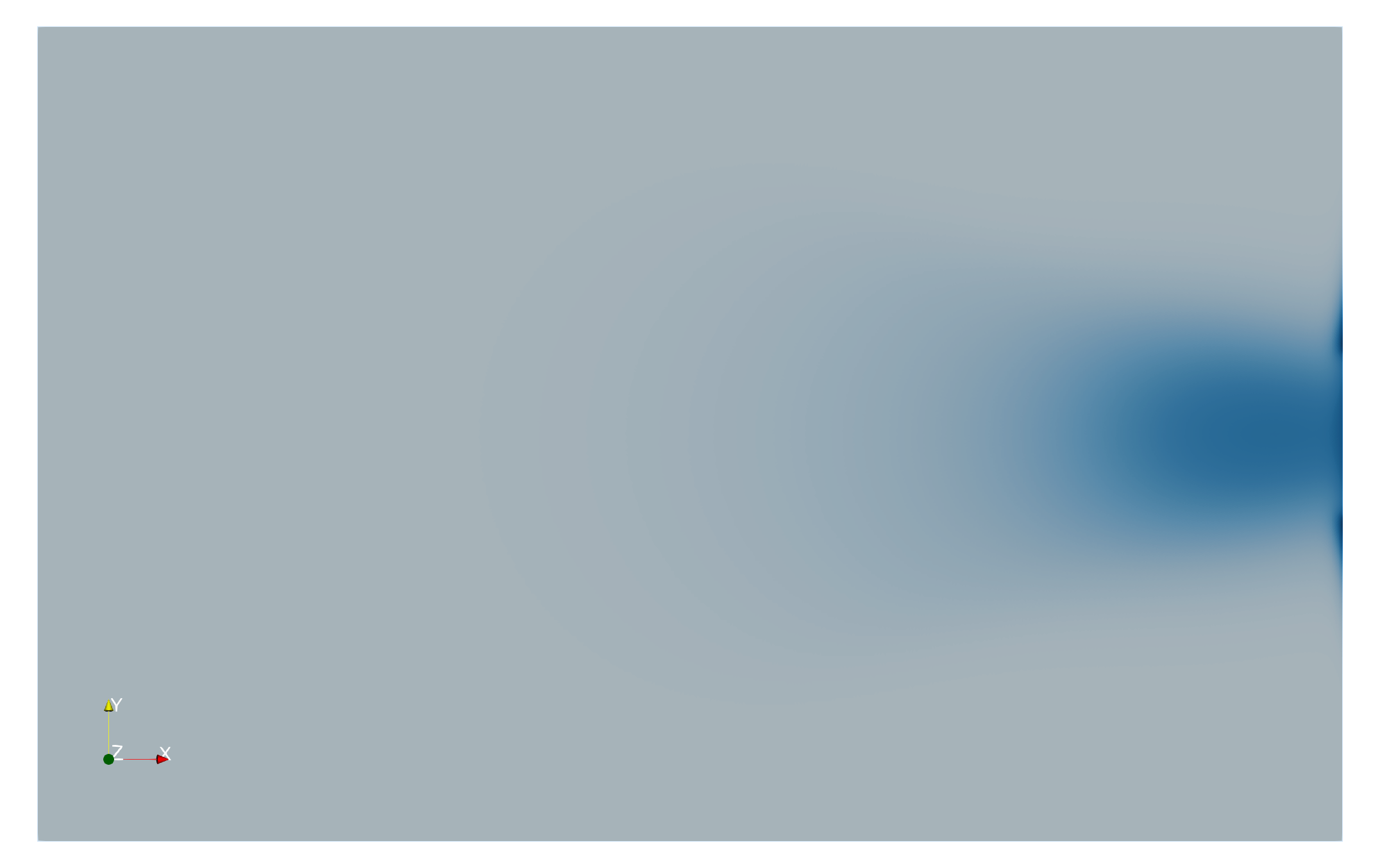}
\end{minipage}
\begin{minipage}[c]{.19\linewidth}
$$(d_3)$$ 
\includegraphics[width=3cm,height=2cm]{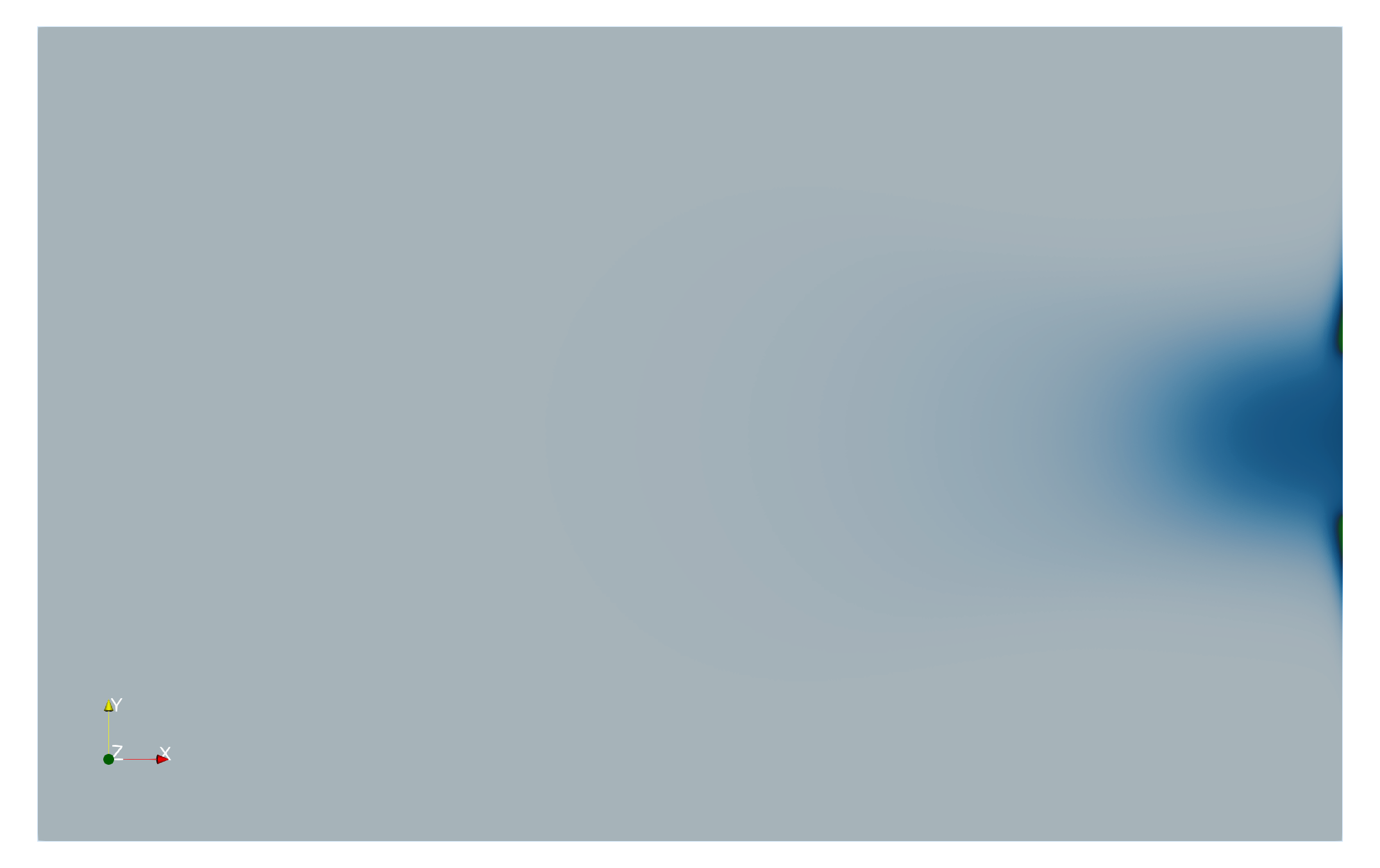}
\end{minipage}
\begin{minipage}[c]{.19\linewidth}
$$(e_3)$$ 
\includegraphics[width=3cm,height=2cm]{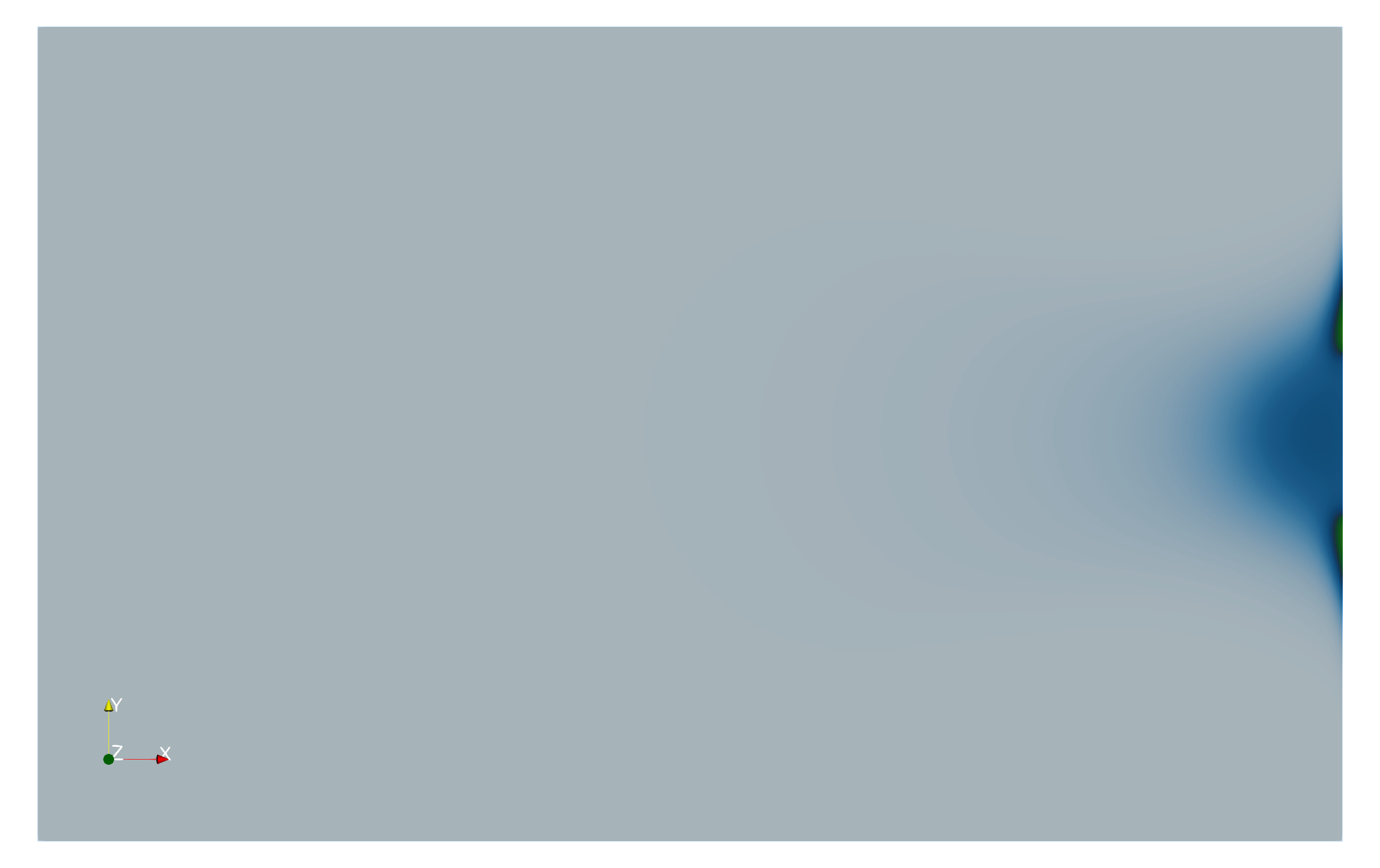}
\end{minipage}
Population in control $\rho_3$
\end{center}
\begin{center}
\begin{minipage}[c]{.19\linewidth}
$$(a_4)$$ 
\includegraphics[width=3cm,height=2cm]{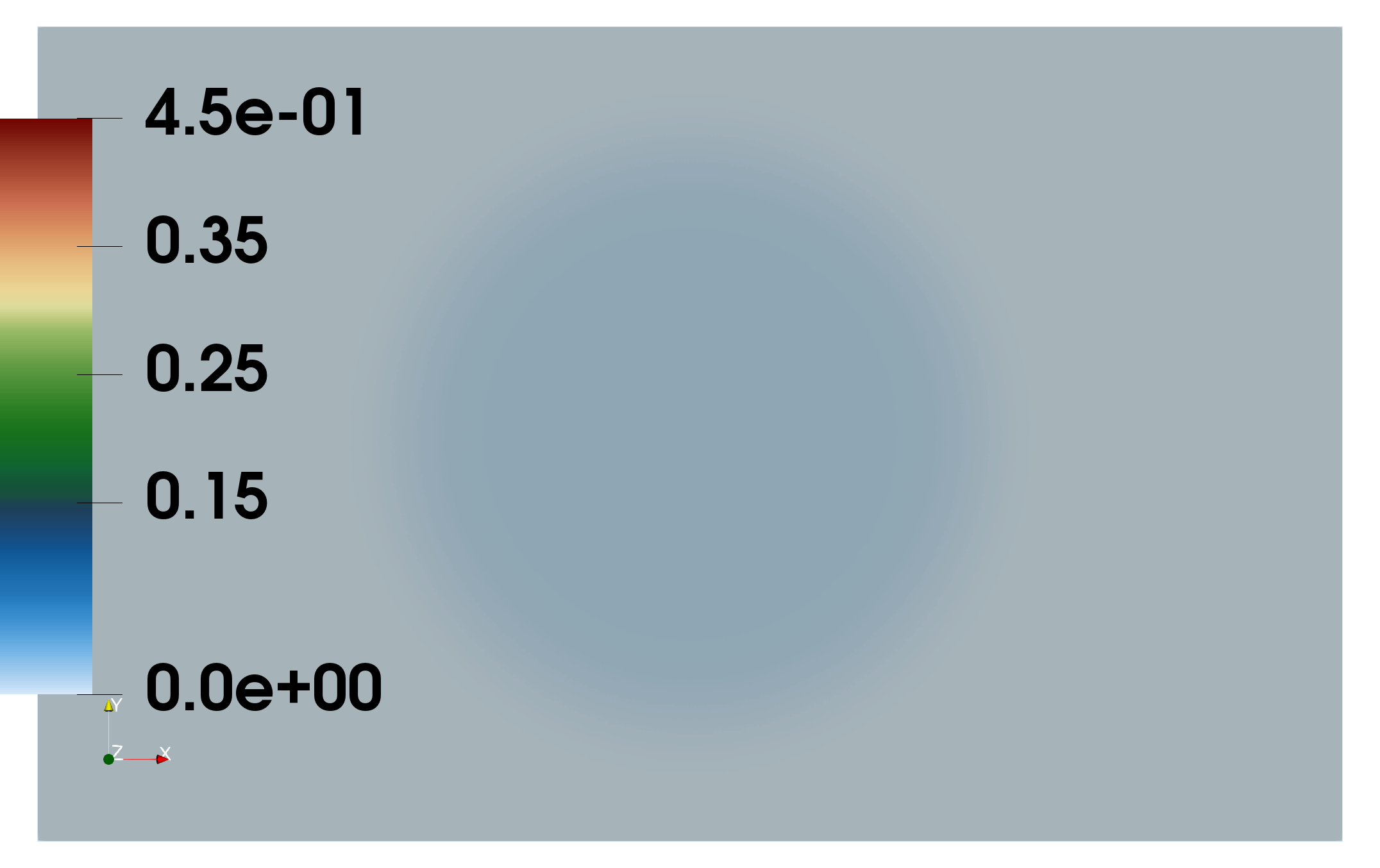}
\end{minipage}
\begin{minipage}[c]{.19\linewidth}
$$(b_4)$$ 
\includegraphics[width=3cm,height=2cm]{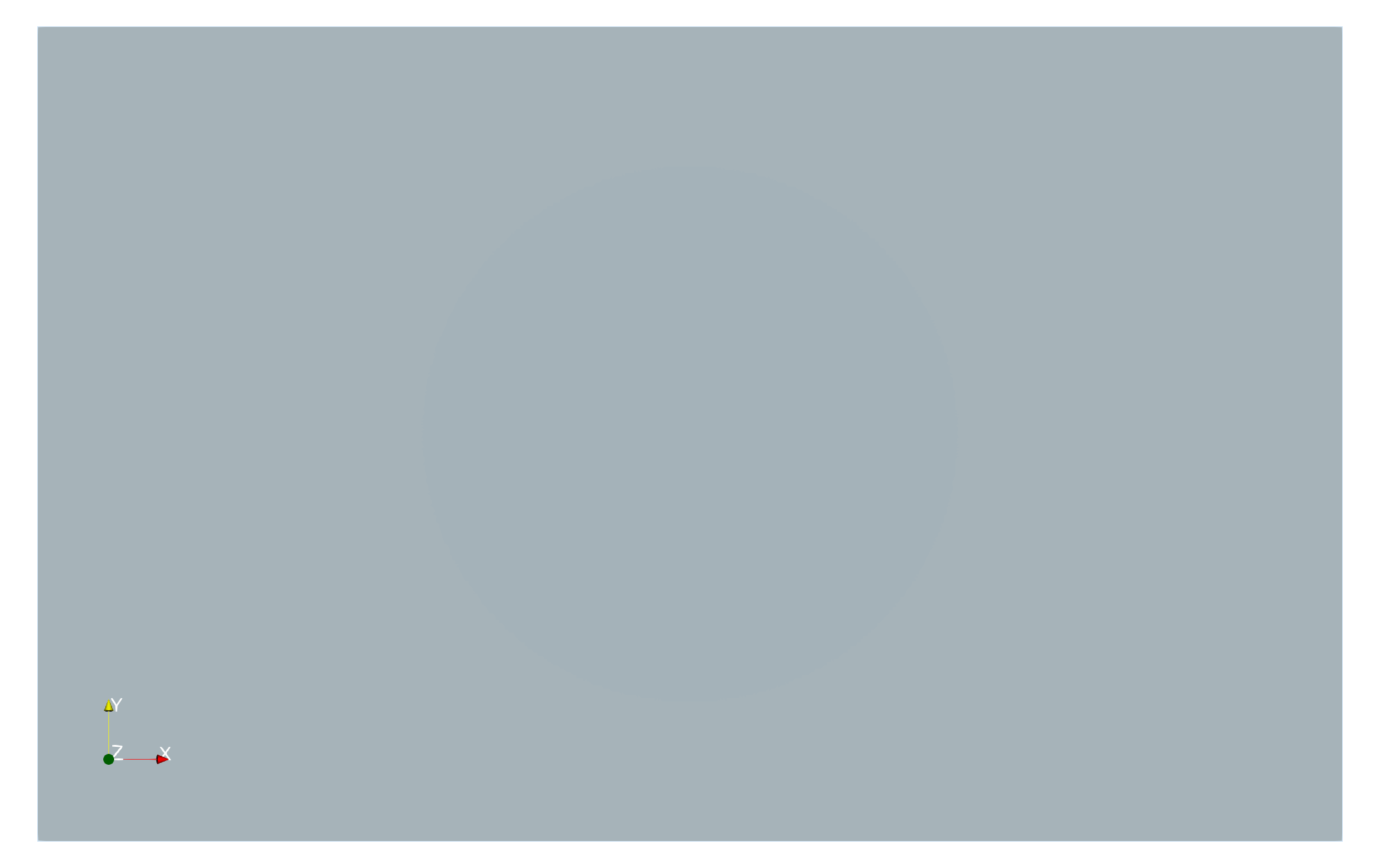}
\end{minipage}
\begin{minipage}[c]{.19\linewidth}
$$(c_4)$$ 
\includegraphics[width=3cm,height=2cm]{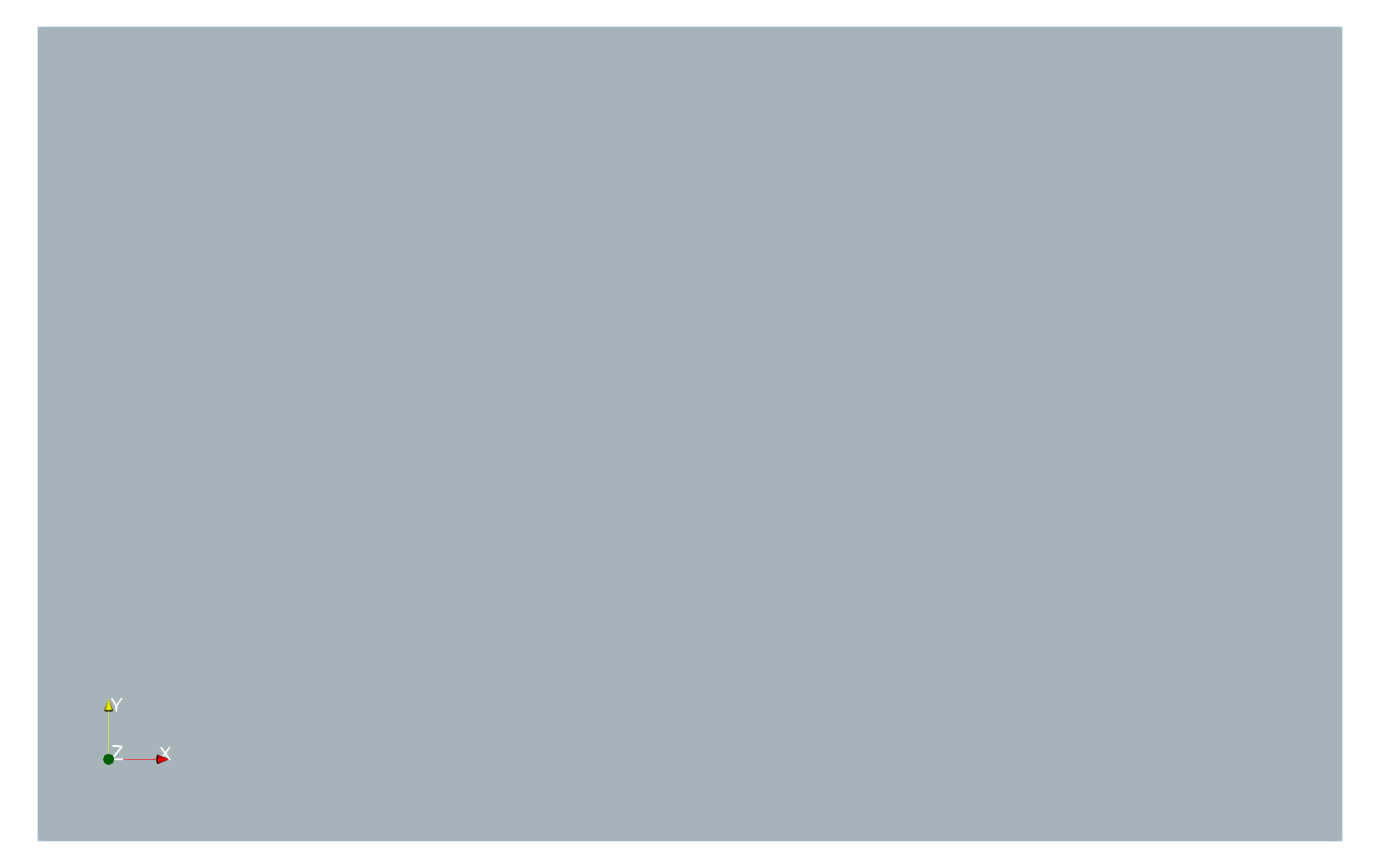}
\end{minipage}
\begin{minipage}[c]{.19\linewidth}
$$(d_4)$$ 
\includegraphics[width=3cm,height=2cm]{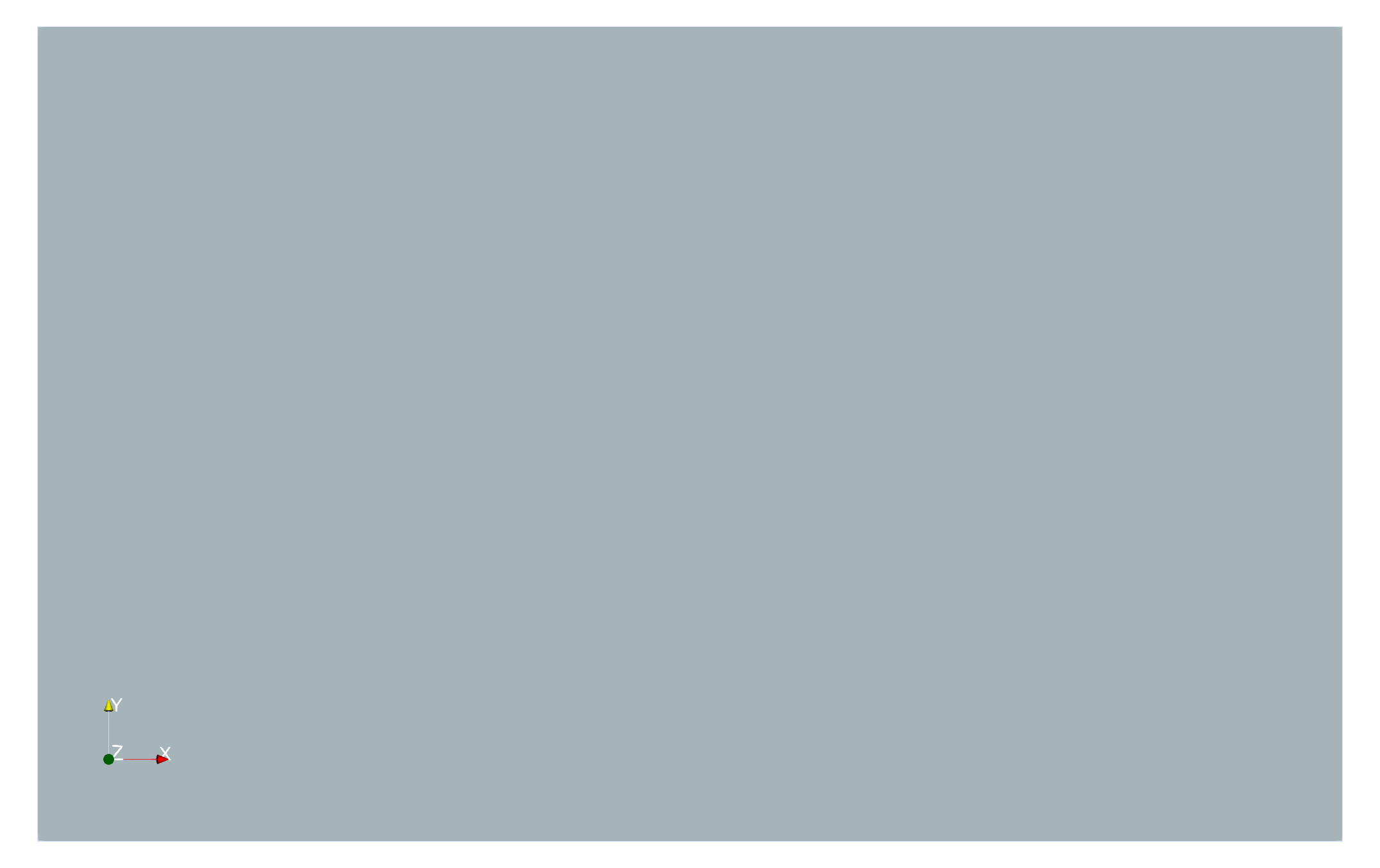}
\end{minipage}
\begin{minipage}[c]{.19\linewidth}
$$(e_4)$$ 
\includegraphics[width=3cm,height=2cm]{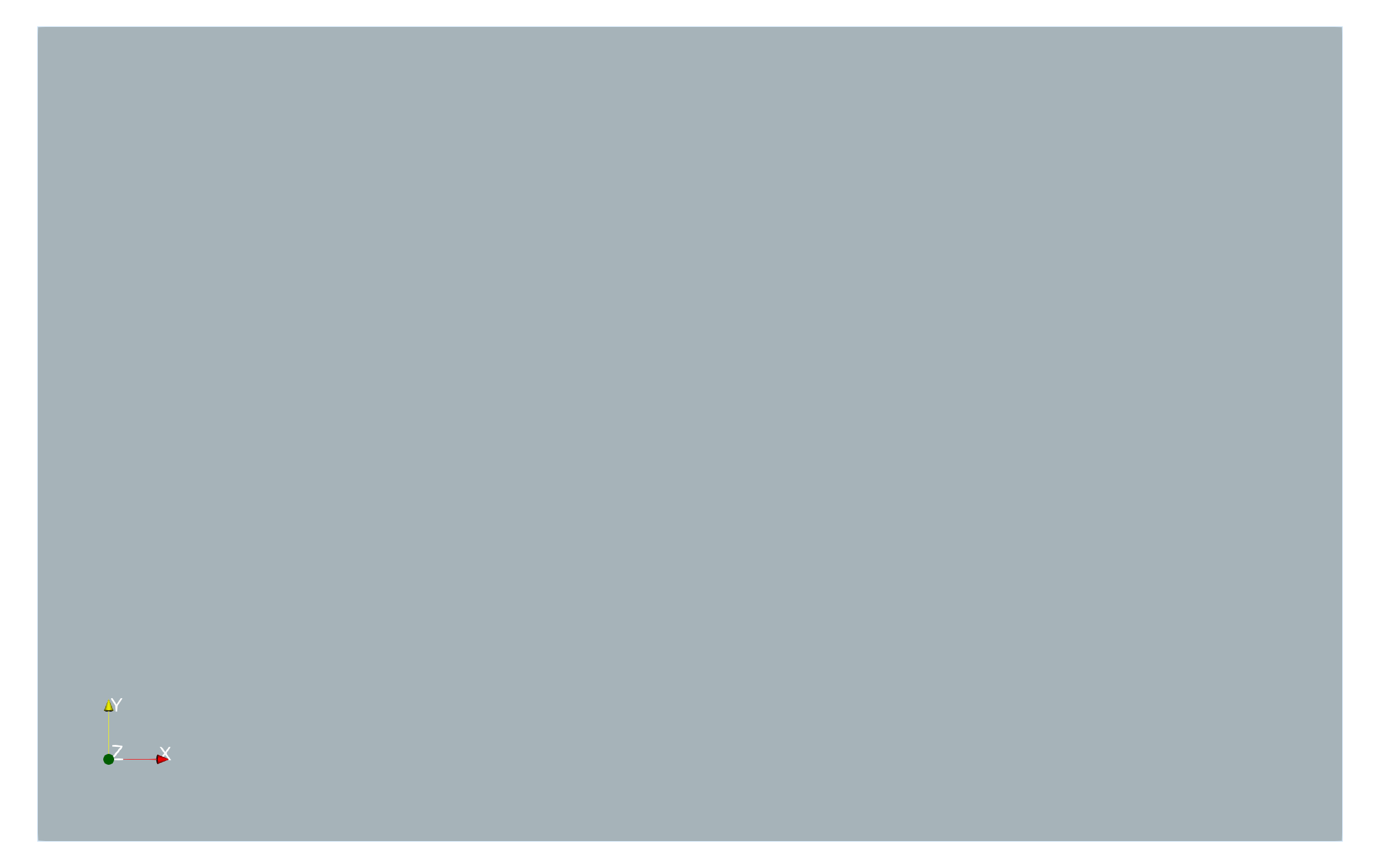}
\end{minipage}
Population in daily behavior $\rho_4$
\caption{{\small{In order to present the time evolution of the human behaviors (alert, panic, control and daily behaviors), we present the simulations results of this populations in Scenario 1 and with respect to captures times $t=50,100,150, 200$ and $250$ respectively. So, each row represent a human behavior. We notice that at the beginning of the simulation, there is a majority of daily and alert populations (row 1 and row 4 respectively) rather than population in a state of panic and control (row 2 and row 3 respectively). This dynamic depends on the structure of the APC model, described in Sections 2: at t = 0 everyone is in a daily behavior, then everyone goes through the state of alert before becoming panicked or controlled.  }}}
\label{Others behaviors}
\end{center}
\end{figure}


 We notice that at the beginning of the simulation, for $t=50$, there is a majority of daily and alert populations rather than population in a state of panic and control. This dynamic depends on the structure of the APC model, described in Section \ref{Section0}: at $t=0$ everyone is in a daily behavior, then everyone goes through the state of alert before becoming panicked or controlled. Moreover, since the diffusion coefficients for $\rho_1$ and $\rho_2$ are low, the position of the populations is still more or less the same as at the beginning, see $ (a_1)$-$(a_4) $ in Figure \ref{Others behaviors}.
For $t=250$, for all scenarios, the dynamics of the APC is fully developed. Moreover, diffusion and advection phenomena are now visible: the whole  population is in panic  and control state, and they are concentrated near the exit, while the populations in alert state and in daily behavior are negligible, see $(e_1)$-$(e_4) $ in Figure \ref{Others behaviors}.
Moreover, since (see Table \ref{tab1}) the population considered here has low risk culture, the dominated behavior is that of panic.

Comparing the evacuation of population in panic in the three different scenarios, see in Figures \ref{Panic0}, $ (e_1)$ for Scenario 1, $(e_2)$ for Scenario 2 and $ (e_3)$ for Scenario 3 , one notices a strong congestion at the level of the exit in the first scenario. In the second scenario, splitting the initial population into three clusters reduces this congestion. Finally, the presence of an obstacle as in the third scenario further reduces congestion.\\


\begin{figure}[h!]
\begin{center}
\begin{minipage}[c]{.19\linewidth}
\begin{center}
 $t=50$ \\
 $(a_1)$
\end{center}
\includegraphics[width=3cm,height=2cm]{PANICS1_50.png}
\end{minipage}
\begin{minipage}[c]{.19\linewidth}
\begin{center}
 $t=100$ \\
 $(b_1)$
\end{center}
\includegraphics[width=3cm,height=2cm]{PANICS1_100.png}
\end{minipage}
\begin{minipage}[c]{.19\linewidth}
\begin{center}
$t=150$\\
$(c_1)$ 
\end{center}
\includegraphics[width=3cm,height=2cm]{PANICS1_150.png}
\end{minipage}
\begin{minipage}[c]{.19\linewidth}
\begin{center}
 $t=200$\\$(d_1)$
\end{center}
\includegraphics[width=3cm,height=2cm]{PANICS1_200.png}
\end{minipage}
\begin{minipage}[c]{.19\linewidth}
\begin{center}
$t=250$\\
$(e_1)$ 
\end{center}
\includegraphics[width=3cm,height=2cm]{PANICS1_250.png}
\end{minipage}
Scenario 1
\end{center}
\begin{center}
\begin{minipage}[c]{.19\linewidth}
\begin{center}
$$(a_2)$$
\end{center}
\includegraphics[width=3cm,height=2cm]{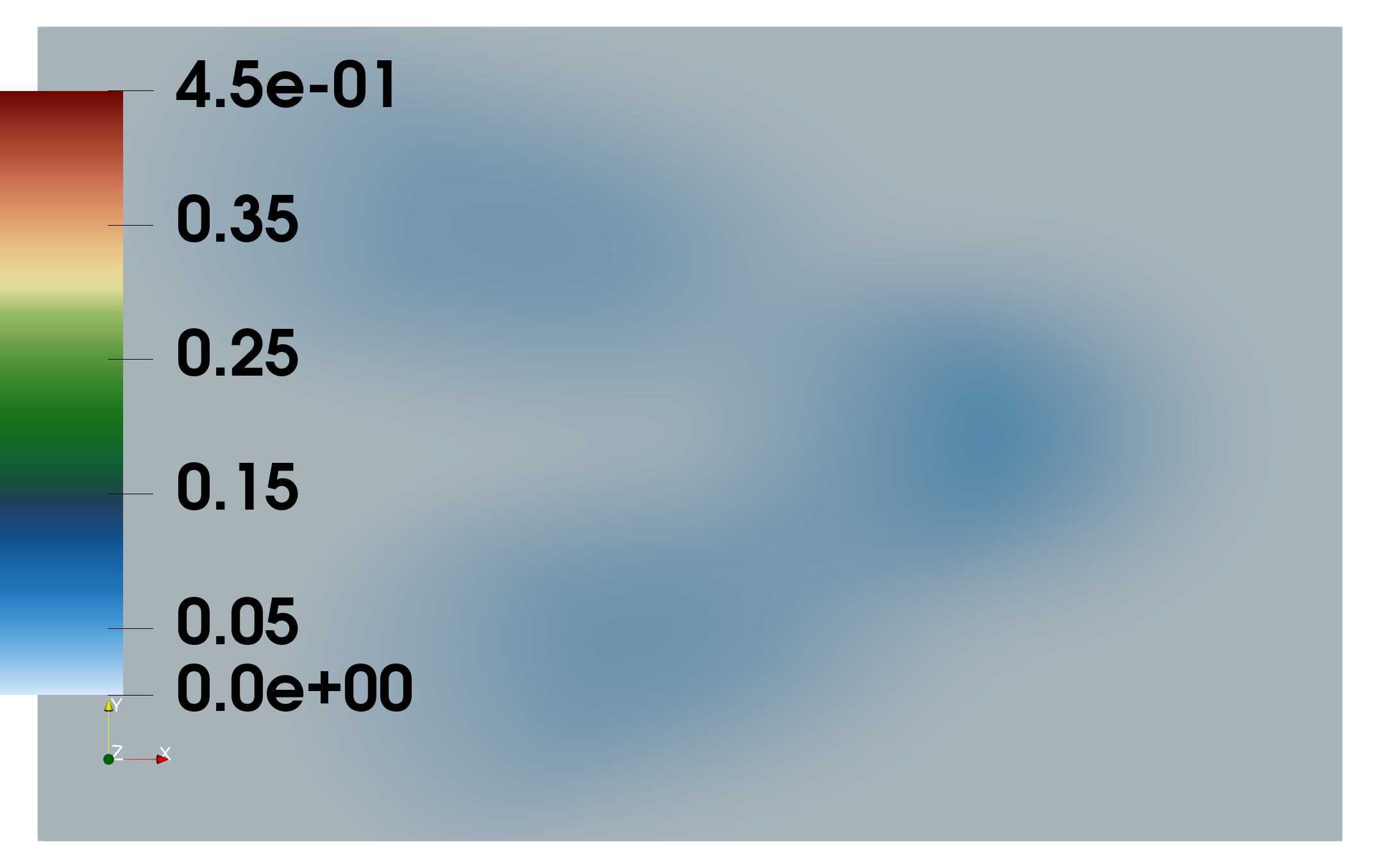}
\end{minipage}
\begin{minipage}[c]{.19\linewidth}
\begin{center}
$$(b_2)$$
\end{center}
\includegraphics[width=3cm,height=2cm]{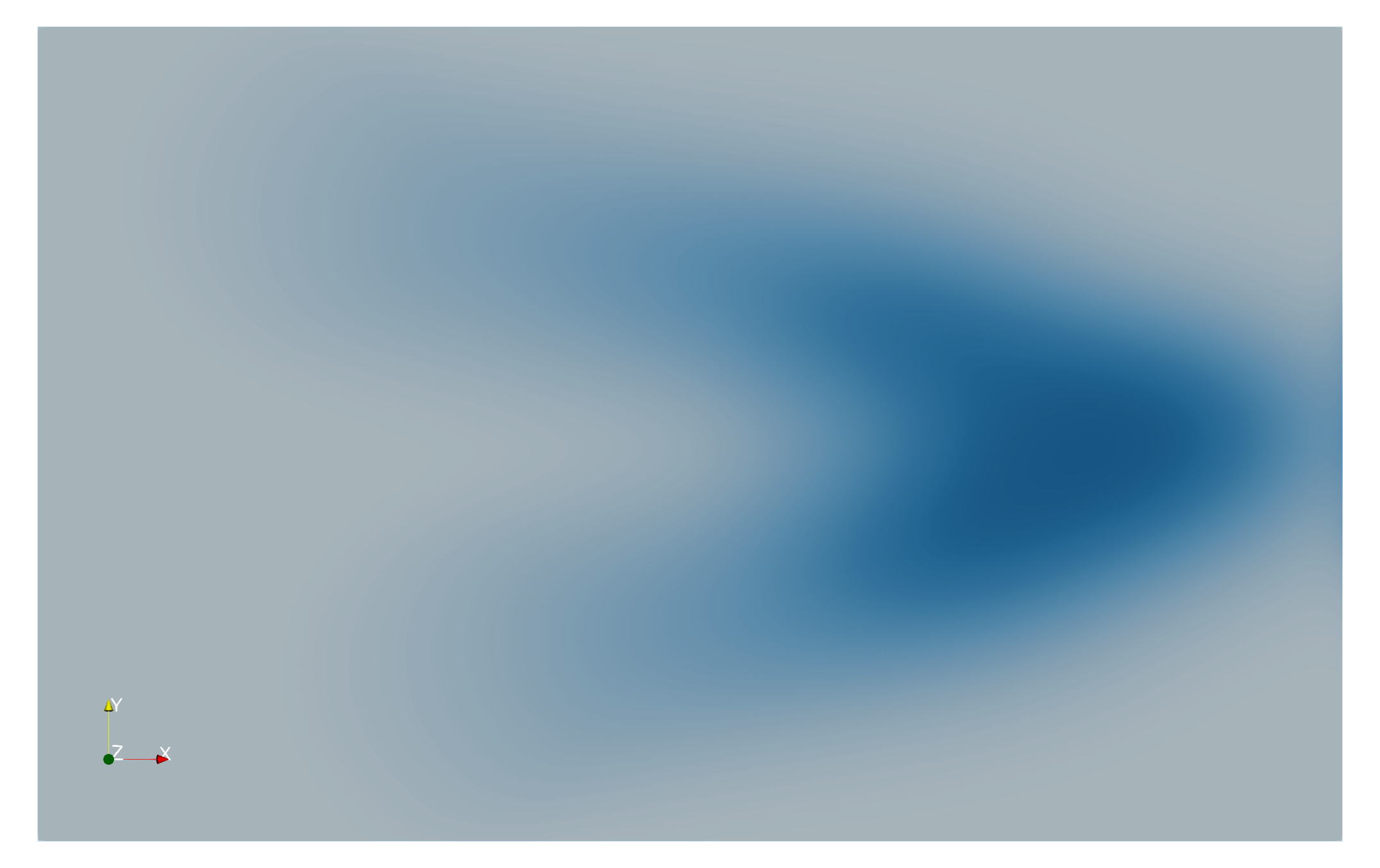}
\end{minipage}
\begin{minipage}[c]{.19\linewidth}
\begin{center}
$$(c_2)$$
\end{center}
\includegraphics[width=3cm,height=2cm]{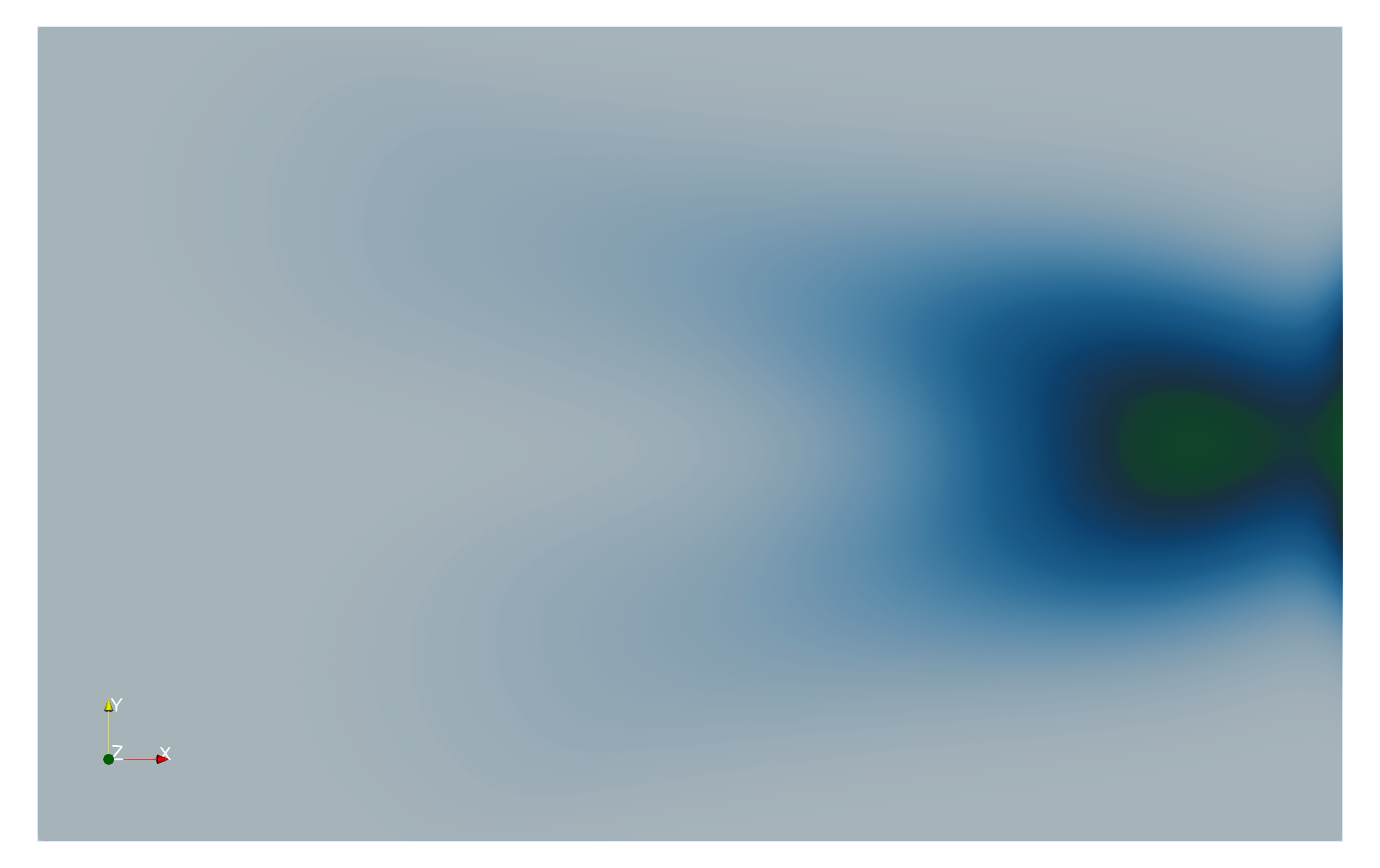}
\end{minipage}
\begin{minipage}[c]{.19\linewidth}
\begin{center}
$$(d_2)$$
\end{center}
\includegraphics[width=3cm,height=2cm]{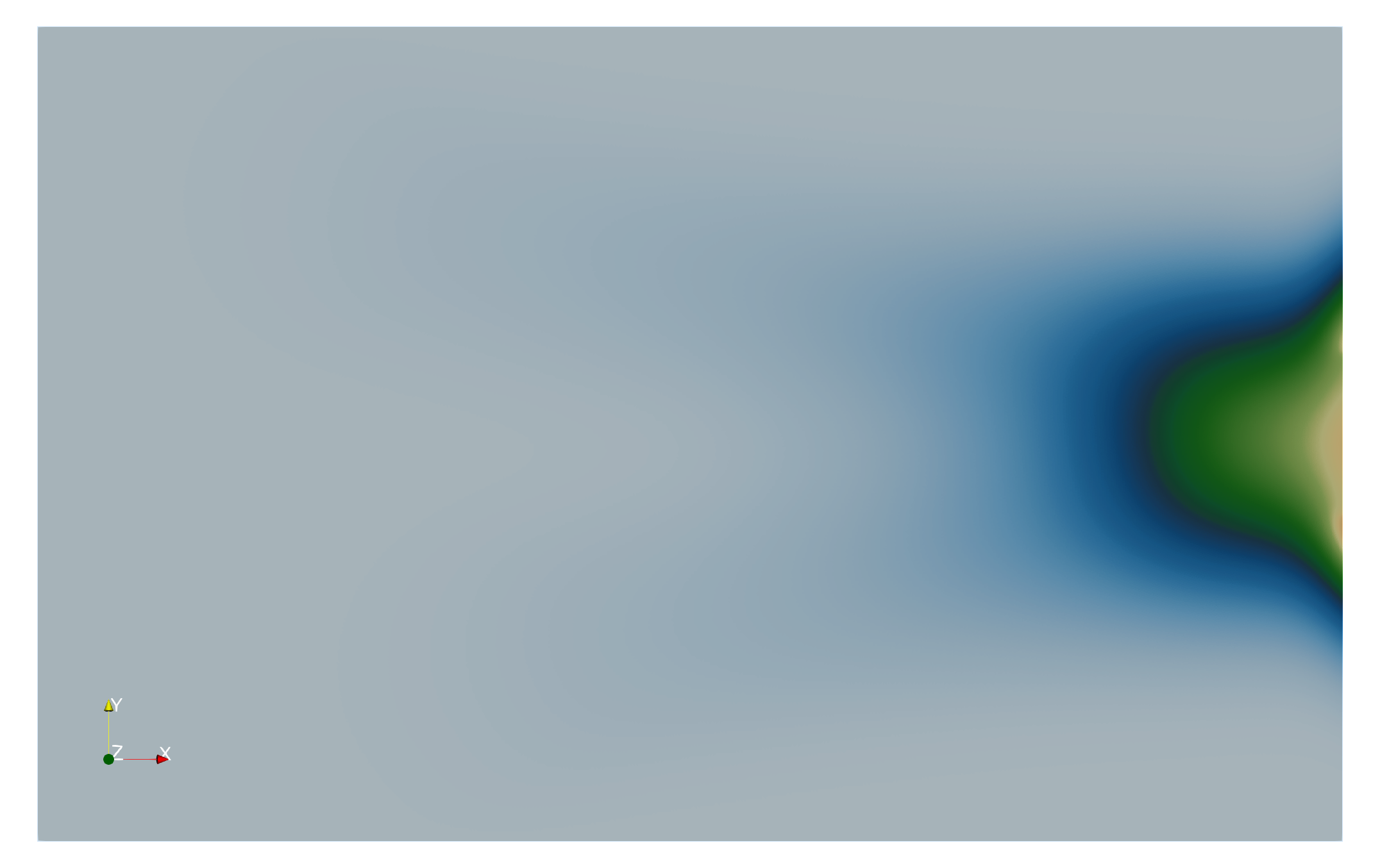}
\end{minipage}
\begin{minipage}[c]{.19\linewidth}
\begin{center}
$$(e_2)$$
\end{center}
\includegraphics[width=3cm,height=2cm]{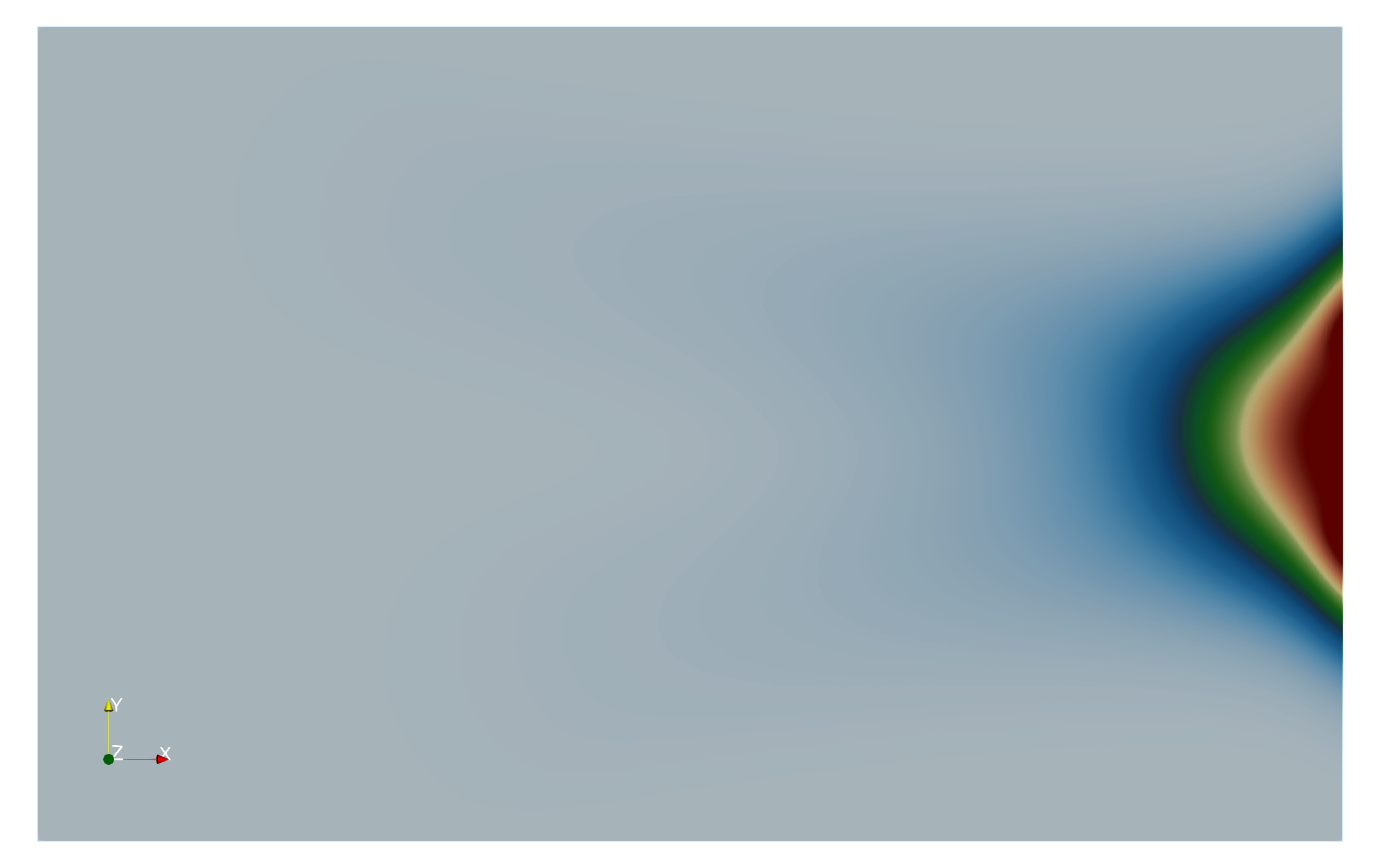}
\end{minipage}
Scenario 2
\end{center}
\begin{center}
\begin{minipage}[c]{.19\linewidth}
$$(a_3)$$
\includegraphics[width=3cm,height=2cm]{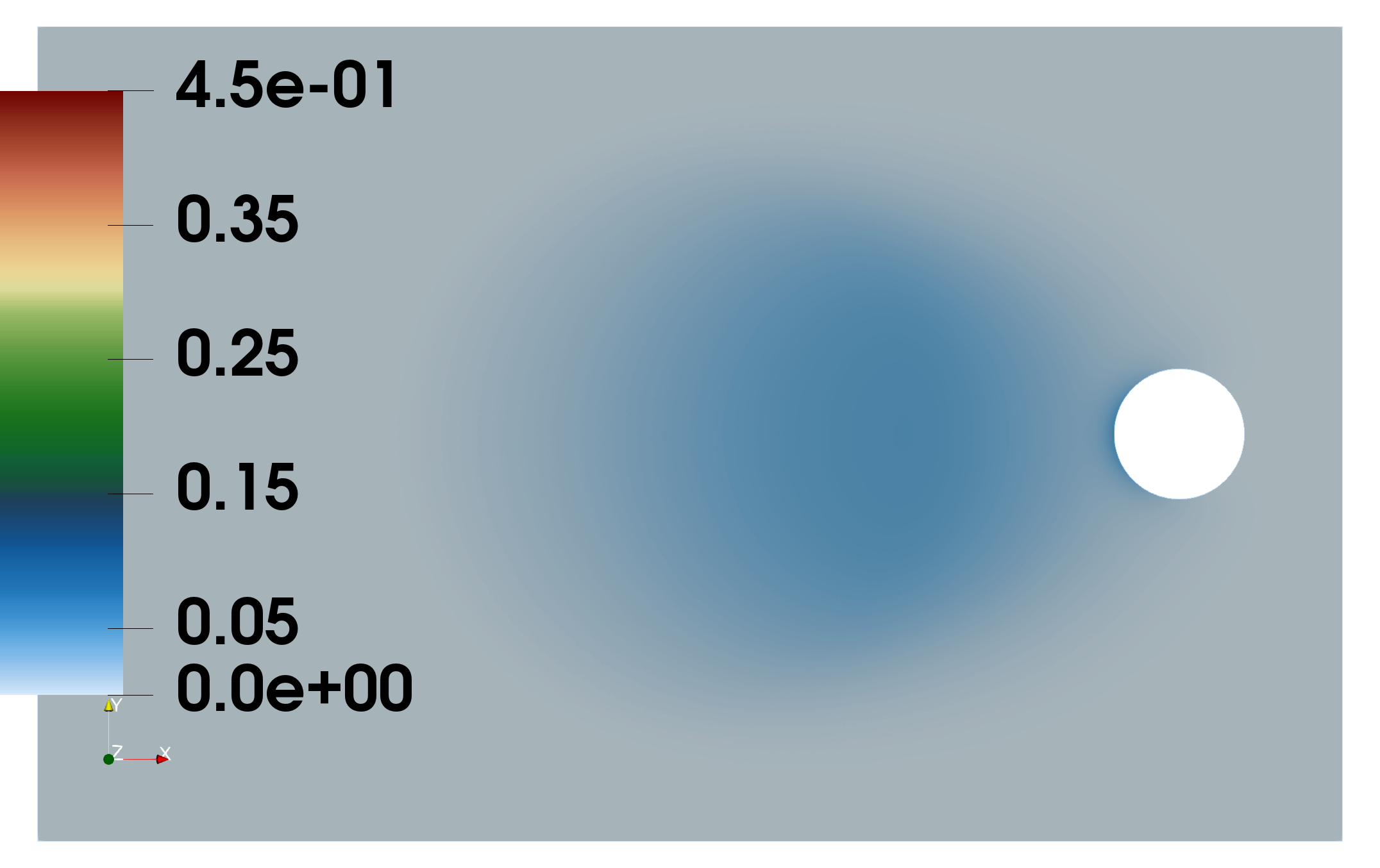}
\end{minipage}
\begin{minipage}[c]{.19\linewidth}
$$(b_3)$$
\includegraphics[width=3cm,height=2cm]{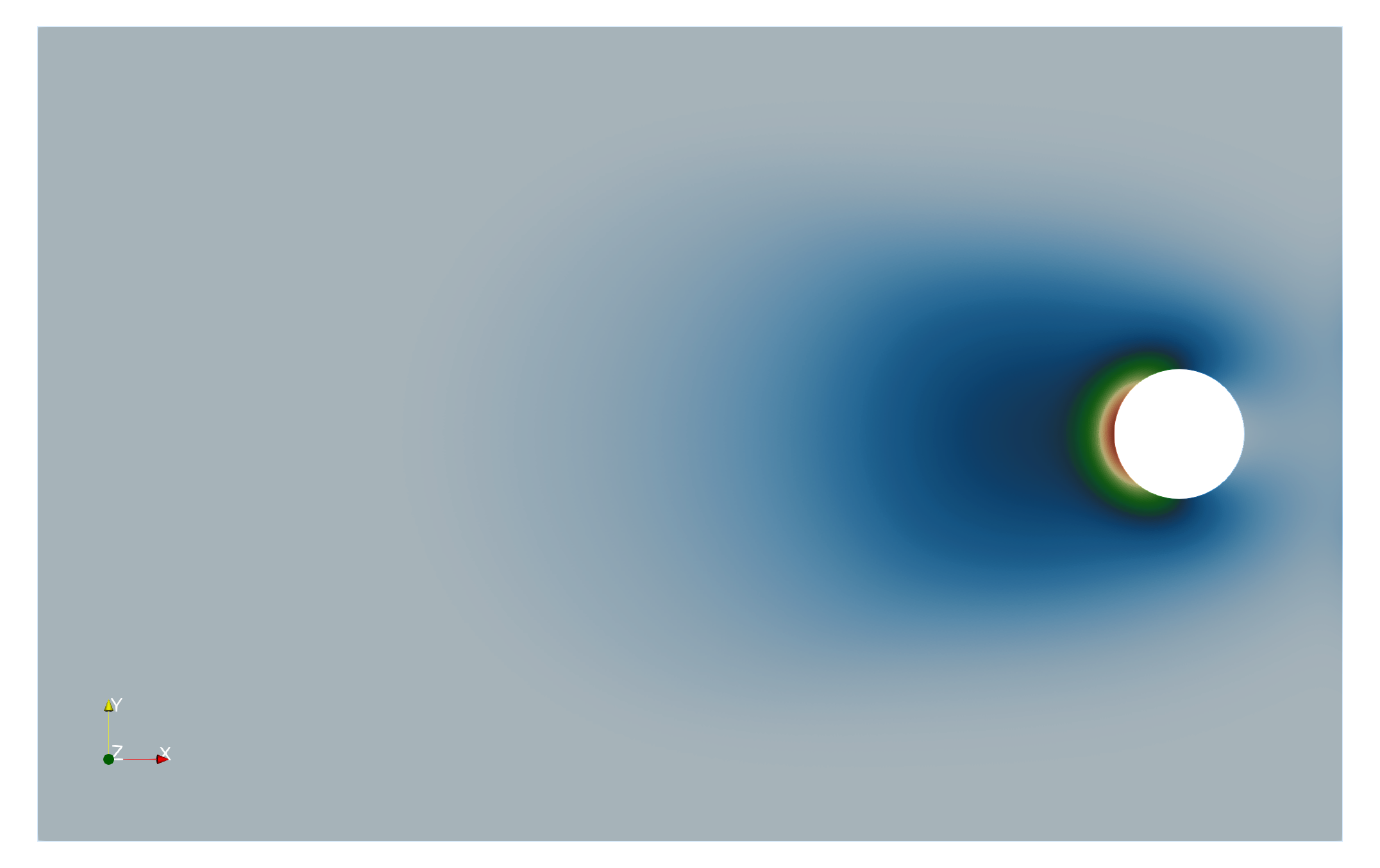}
\end{minipage}
\begin{minipage}[c]{.19\linewidth}
$$(c_3)$$
\includegraphics[width=3cm,height=2cm]{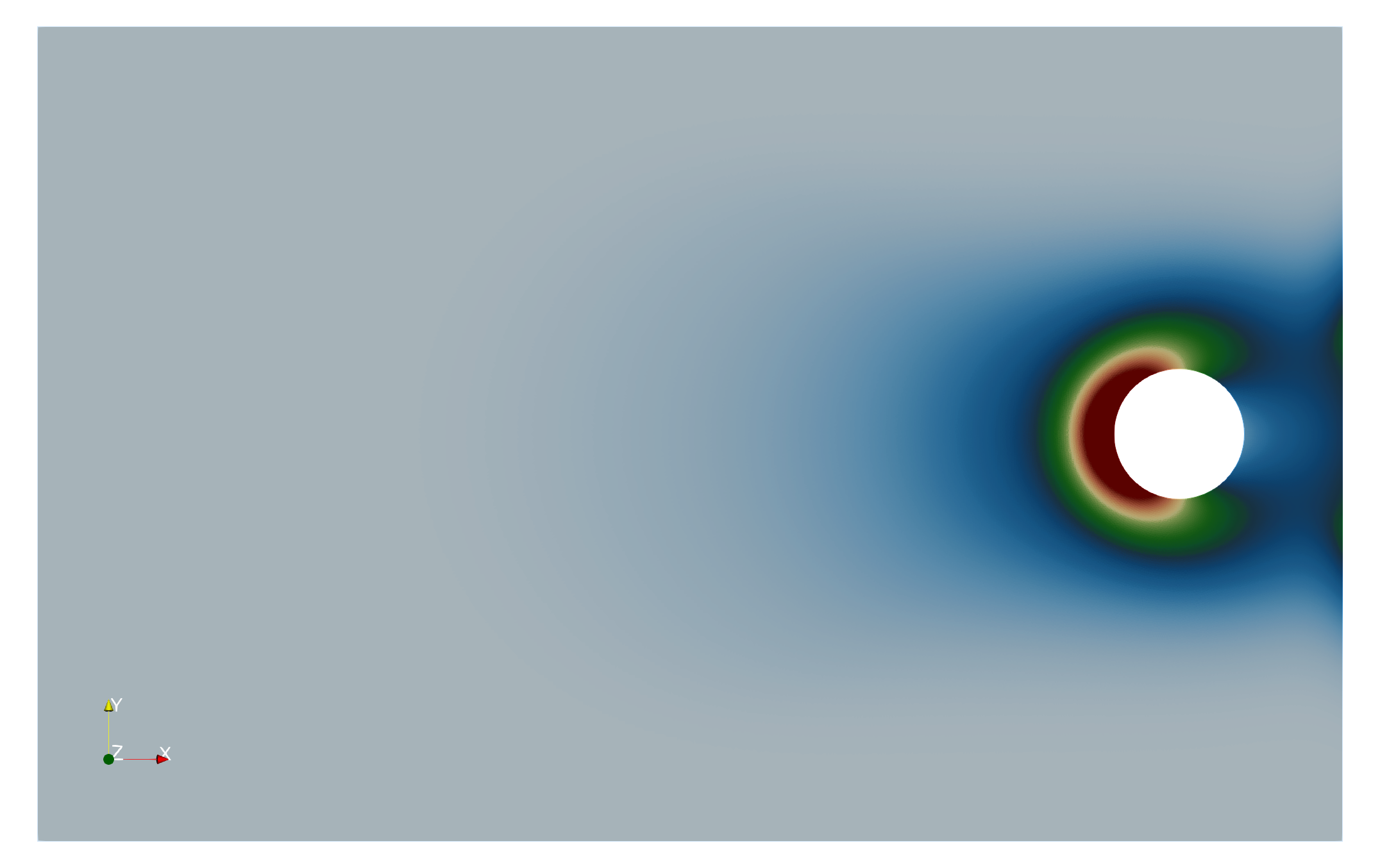}
\end{minipage}
\begin{minipage}[c]{.19\linewidth}
$$(d_3)$$
\includegraphics[width=3cm,height=2cm]{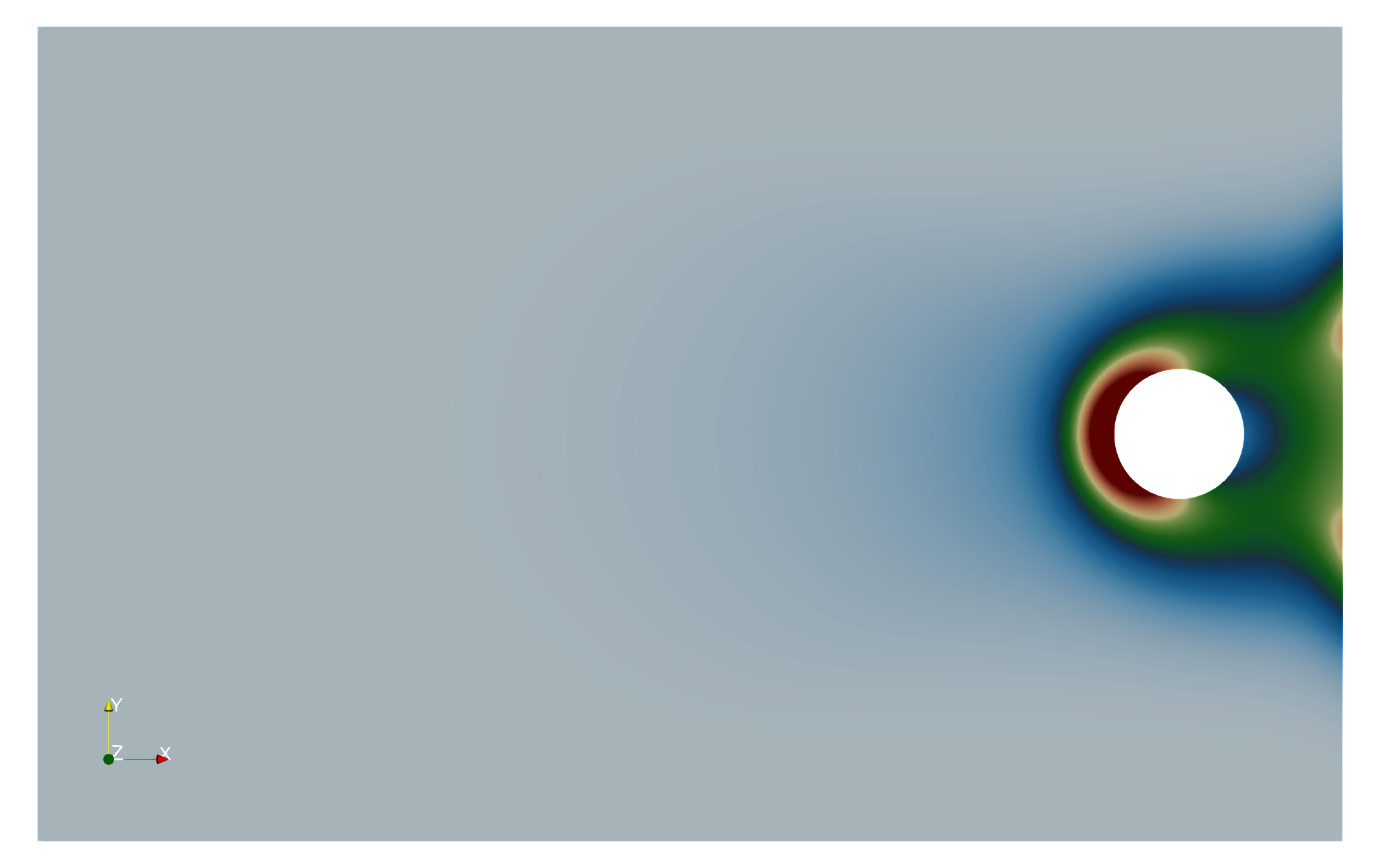}
\end{minipage}
\begin{minipage}[c]{.19\linewidth}
$$(e_3)$$
\includegraphics[width=3cm,height=2cm]{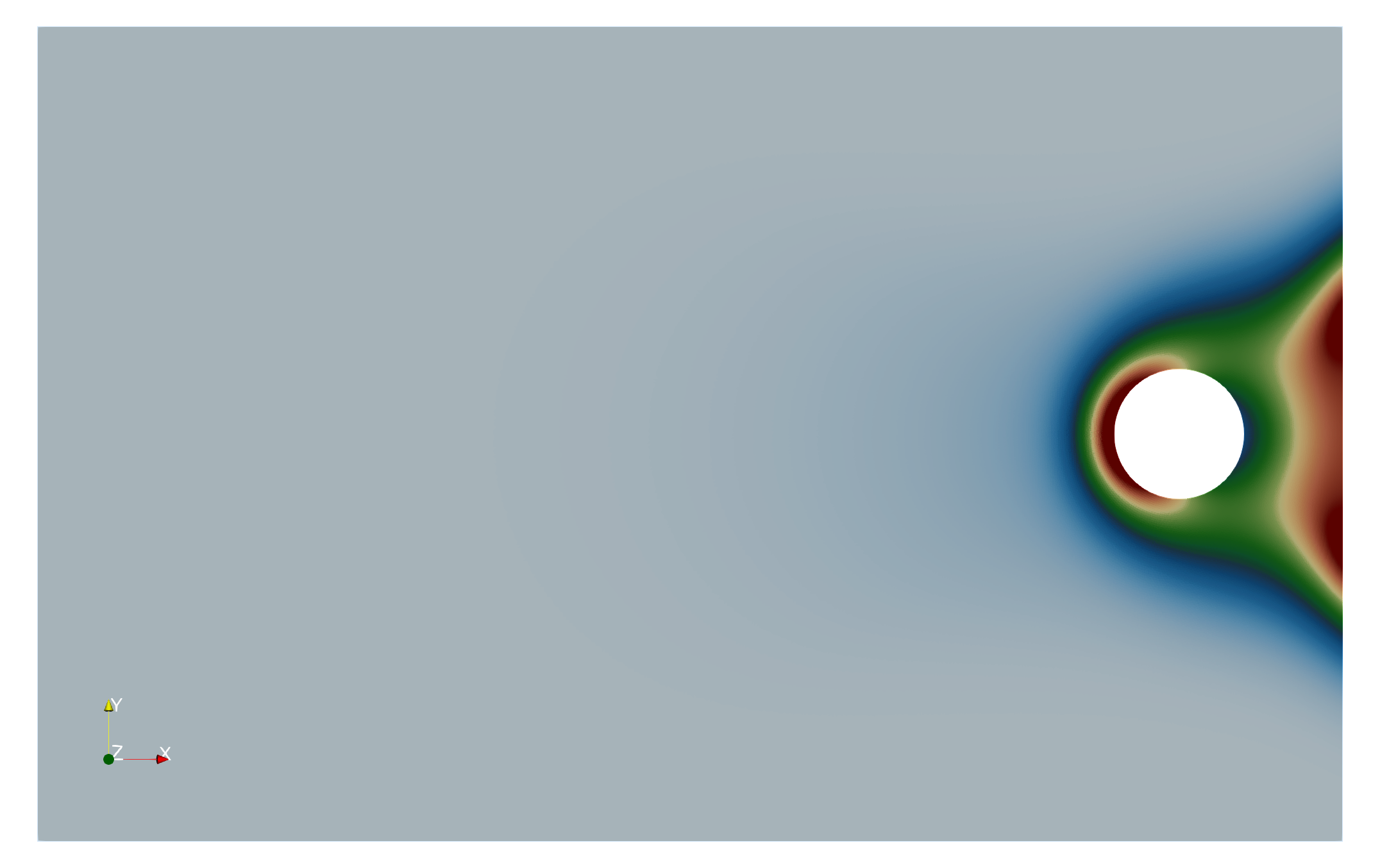}
\end{minipage}
Scenario 3
\caption{{\small{Population in panic $ \rho_2 $ over the three scenarios at the captures times $t=50,100,150, 200$ and $250$ respectively. Notice that, each row represent the (time) evolution of the population at each scenario: at time $t=50$ of each scenario (column $(a)$), population in panic is small with respect to the initial population, since at the this time the majority of population is concentrated in alert, but over time,  panicked ones become greater and greater, thanks to the APC dynamics, and the fact that the population has a low risk culture, in addition, the phenomena of diffusion and advection are now visible (the population in panic concentrated near the exit). Furthermore, by comparing captures at each column $(b)$, $(c)$, $(d)$ and $(e)$, we can observe that in Scenario 2 there is slightly less congestion near the exit than in Scenario 1, this is can be highlighted from time $t=100$ to $t=250$. Moreover, in Scenario 3, we observe that the congestion is less than in the previous cases. Indeed, the role of the obstacle is to facilitate the access to the exit. }}}
\label{Panic0}
\end{center}
\end{figure}

\section{Conclusion}
In this work, we introduce  a new spatio-temporal APC (alert, panic and control) model describing the evacuation of a population presented via different human behaviors during a catastrophic event. First, using the first-order macroscopic crowd theory, we derive a new spatio-temporal APC model. It is a system of advection-diffusion-reaction equations with nonlinear Robin boundary conditions. Then, using a semigroup approach and abstract evolution equations, we prove the local existence and a regularity result of the solutions of our model.  Moreover, we establish the positivity of the solution using the positively invariant sets approach. Finally, to illustrate our results, we present different numerical simulations of the population evacuation, using three different scenarios.
\clearpage
\begin{appendices}
\section{Proofs of the preliminary results of Section \ref{sec:3}}\label{Annexe:A}
 \noindent  \textbf{Proof of Proposition \ref{Proposition:3.2}.}
\textbf{(i)} It is well-known from \cite[Section 3.1.1]{Lunardi}, that for each $i=1,\cdots,5,$  the realization of the Laplacian operator $ \mathcal{A}_i=d_i\Delta $ on $ L^p(\Omega) $ with homogeneous Neumann boundary condition generates a contraction holomorphic $C_0$-semigroup $(\mathcal{T}_{i}(t))_{t\geq 0}$ of angle $\pi/2$, we refer also to \cite[Sections 1.4]{Davies} for a similar result. So that, $ \mathcal{A}_{0} $ generates a contraction holomorphic $C_0$-semigroup $(\mathcal{T}(t))_{t\geq 0} $ of angle $\pi/2$ on $X$ as a diagonal matrix-valued operator. Hence, the semigroup $(\mathcal{T}(t))_{t\geq 0} $ is given by:  
\begin{equation}\label{definition semigroup}
\mathcal{T}(t)=diag(\mathcal{T}_{1}(t),\cdots ,\mathcal{T}_{5}(t) )^*, \quad t\geq 0.
\end{equation}
\textbf{(ii)} The compactness of the semigroup $ (\mathcal{T}(t))_{t\geq 0} $ (when for each $t>0$, $\mathcal{T}(t)$ defines a compact operator) follows from the composition of the semigroup $ (\mathcal{T}(t))_{t\geq 0} $ \eqref{definition semigroup} using the results of \cite[Sections 1.6]{Davies}.  For the positivity of the semigroup  $(\mathcal{T}(t))_{t\geq 0} $ on the Banach lattice $X$ it suffices to prove that the resolvent $ R(\lambda , \mathcal{A}_0) $ is a positive operator for $\lambda \in \varrho (\mathcal{A}_{0})$ large enough. That is, to prove that, for $ \varphi (x) \geq 0 $  we have $ R(\lambda, \mathcal{A}_0) \varphi(x) \geq 0$ for all $x\in \Omega$, see \cite[ Chapter VI, Section 1.8]{Nagel},  we refer also to the result of invariance under closed sets \cite[Theorem 5.1]{Pazy}. This is equivalent to prove that, for $\psi \in D(\mathcal{A}_0):=\ker(\mathcal{L})$ which is the unique solution of the following resolvent equation: $$\varphi(x)= (\lambda-\mathcal{A}_0)\psi(x) \geq 0, \quad x\in \Omega,$$ we have $\psi$ is positive ($\psi$ is always exists, the question is about its positivity). This holds directly from the maximum principle, see \cite{Davies}.\\

\noindent  \textbf{Proof of Proposition \ref{Proposition:3.5}.}
Let $ 0 \leq \delta \leq 1  $, the fact that the extension semigroup $ (\mathcal{T}_{\delta-1}(t))_{t\geq 0} $ exists as a strongly continuous positive semigroup with generator $( \mathcal{A}_{\delta-1},D(\mathcal{A}_{\delta-1})=X_{\delta}) $ is due to \cite{Batkai}. The analycity and compactness of $ (\mathcal{T}_{\delta-1}(t))_{t\geq 0} $ follows from \cite{Greiner}.
\\

 \noindent  \textbf{Proof of Proposition \ref{Lemma Local Lipschitz}.}
Recall that $\tilde{\mathcal{K}}:=\mathcal{K}+ (\lambda-\mathcal{A}_{\beta-1})  \mathcal{D} \mathcal{M}$. Hence, by definition, we have $(\lambda-\mathcal{A}_{\beta-1}) \mathcal{D} \in L(\partial X , X_{\beta-1})$ for $\lambda \in \varrho(\mathcal{A}_0)=\varrho(\mathcal{A}_{\beta-1})$. So that, to prove the Lipschitz continuity of $ \tilde{\mathcal{K}} $ in bounded sets, it suffices to examine the operators $ \mathcal{K} $ and $\mathcal{M}$. Indeed, let $(\varphi_1,\cdots,\varphi_5)^*:=\varphi, (\upsilon_1,\cdots,\upsilon_5)^*:=\upsilon \in X_{\alpha} $ and $R >0$ be such that $  \| \varphi\|_{\alpha}, \ \| \upsilon\|_{\alpha} \leq R $. By construction, the functions $\mathcal{F}, \ \mathcal{G}$ and $\mathcal{H}$ are (pointwisely) Lipschitz continuous in bounded sets. That is, there exist  $ L^i_R \geq 0,$ $ i=1,2,3 $ (depending only on $R$) such that
\begin{align*}
| \mathcal{F}(\rho_1,\rho_3)(x)-\mathcal{F}(\upsilon_1,\upsilon_3)(x)|\leq L^1_R \left( | \rho_1(x)-\upsilon_1(x)|+| \rho_3(x)-\upsilon_3(x)|\right), \quad x\in \Omega 
 \end{align*} 
  \begin{align*}
| \mathcal{G}(\rho_1,\rho_2)(x)-\mathcal{G}(\upsilon_1,\upsilon_2)(x)|\leq L^2_R \left( | \rho_1(x)-\upsilon_1(x)|+| \rho_2(x)-\upsilon_2(x)|\right),\quad x\in \Omega 
 \end{align*} 
 and 
 \begin{align*}
| \mathcal{H}(\rho_2,\rho_3)(x)-\mathcal{H}(\upsilon_2,\upsilon_3)(x)| \leq L^3_R \left( | \rho_2(x)-\upsilon_2(x) |+| \rho_3(x)-\upsilon_3(x) |\right), \quad x\in \Omega 
 \end{align*}
for some $ L^i_R \geq 0,$ $ i=1,2,3 $. So, by passing to the $L^p$-norm, and using the continuous embedding $ W^{2\alpha,p}\hookrightarrow C^1 (\overline{\Omega}) $ we obtain that 
\begin{align}
\| \mathcal{F}(\rho_1,\rho_3)-\mathcal{F}(\upsilon_1,\upsilon_3)\|\leq |\Omega |  L^1_R \left( \| \rho_1-\upsilon_1 \|_{0,\alpha}+\| \rho_3-\upsilon_3 \|_{0,\alpha}\right), \label{Lemma1 Formu1}
\end{align} 
\begin{align}
\| \mathcal{G}(\rho_1,\rho_2)-\mathcal{G}(\upsilon_1,\upsilon_2)\|\leq | \Omega |  L^2_R \left( \| \rho_1-\upsilon_1 \|_{0,\alpha}+\| \rho_2-\upsilon_2 \|_{0,\alpha} \right), \label{Lemma1 Formu2}
 \end{align} 
 and 
 \begin{align}
\| \mathcal{H}(\rho_2,\rho_3)-\mathcal{H}(\upsilon_2,\upsilon_3) \| \leq | \Omega | L^3_R \left( \| \rho_2-\upsilon_2 \|_{0,\alpha}+\| \rho_3-\upsilon_3 \|_{0,\alpha} \right), \label{Lemma1 Formu3}
 \end{align}
 where $|\Omega |$ is the Lebesgue measure of $\Omega$.
Now, we show that the terms $ \nabla \cdot  (\varphi_2 V_{2}(\varphi)\nu)) $, $ \nabla \cdot (\varphi_3 V_{3}(\varphi)\nu) $ are Lipschitz continuous in bounded sets in $ X_{\alpha} $. So, a straightforward calculus yields that
$$ \varphi_2 V_{2}(\varphi)(x)-   \upsilon_2 V_{2}(\upsilon)(x)=V_{2,max}\left( \left( 1-\tilde{\varphi}(x) \right) \left[ \varphi_2 (x) -\upsilon_2 (x)\right] +\upsilon_2 (x) \sum_{i=1}^{5} \left[ \varphi_i (x) -\upsilon_i (x)\right]\right) , \quad x\in \Omega. $$

Furthermore, using the regularity of $\varphi$, $\upsilon$ and $\nu:=\nu_{|\Omega} \in W^{1,\infty}(\Omega)$ since the gradient operator $\nabla $ is linear, we obtain that
\begin{eqnarray*}
&&\nabla \cdot ( \varphi_2 V_{2}(\varphi)\nu(x))-   \nabla \cdot (\upsilon_2 v_{2}(\upsilon)\nu (x))\\
&=& \nabla \cdot ( \varphi_2 V_{2}(\varphi)\nu(x))- \nabla \cdot ( \varphi_2 v_{2}(\upsilon)\nu(x))+\nabla \cdot ( \varphi_2 v_{2}(\upsilon)\nu(x))-  \nabla \cdot (\upsilon_2 v_{2}(\upsilon)\nu (x))\\
&=&\underbrace{ \nabla  \varphi_2 (x)\cdot(v_{2}(\varphi)\nu(x))- \nabla   \varphi_2(x) \cdot (v_{2}(\upsilon)\nu(x)) + \varphi_2 (x)\nabla \cdot(v_{2}(\varphi)\nu(x))   -\varphi_2(x) \nabla \cdot (v_{2}(\upsilon)\nu(x))}_{I_1 (x)} \\
&&+\underbrace{\nabla  \varphi_2 (x)\cdot ( v_{2}(\upsilon)\nu(x)) -  \nabla \upsilon_2 (x)\cdot ( v_{2}(\upsilon)\nu (x))+ \varphi_2 (x)\nabla \cdot ( v_{2}(\upsilon)\nu(x)) -   \upsilon_2 (x)\nabla\cdot (v_{2}(\upsilon)\nu (x))}_{I_2 (x)}, \quad x\in \Omega.
\end{eqnarray*}
Then, we have
\begin{align*}
\mid I_1 (x) \mid &\leq  M_{V_{2}} \left( \mid \nabla  \varphi_2 (x)\mid \sum_{i=1}^{5} \mid \varphi_i(x) -\upsilon_i(x) \mid + \mid \varphi_2 (x)\mid\sum_{i=1}^{5} \mid \nabla (\varphi_i(x) -\upsilon_i(x)) \mid \right) , \quad x\in \Omega,
\end{align*}
\begin{align*}
\mid I_2 (x) \mid &\leq  M_{V_{2}} \left( \mid \nabla  (\varphi_2 (x)-\upsilon_2(x))\mid \sum_{i=1}^{5} (1+\mid \varphi_i(x)\mid) + \mid \varphi_2 (x)-\upsilon_2 (x) \mid\sum_{i=1}^{5} \mid \nabla (\upsilon_i(x) \mid \right)  , \quad x\in \Omega,
\end{align*}
where $ M_{V_{2}}=V_{2,max} |\nu|_{\infty} $. Hence, we have 
\begin{eqnarray*}
&& \mid \nabla \cdot ( \varphi_2 V_{2}(\varphi)\nu(x))-   \nabla \cdot (\upsilon_2 v_{2}(\upsilon)\nu (x)) \mid \leq \\
&&  M_{V_{2}} \left( \mid \nabla  \varphi_2 (x)\mid \sum_{i=1}^{5} \mid \varphi_i(x) -\upsilon_i(x) \mid + \mid \varphi_2 (x)\mid\sum_{i=1}^{5} \mid \nabla (\varphi_i(x) -\upsilon_i(x)) \mid \right)  \\
&&+ M_{V_{2}} \left( \mid \nabla  (\varphi_2 (x)-\upsilon_2(x))\mid \sum_{i=1}^{5} (1+\mid \varphi_i(x)\mid) + \mid \varphi_2 (x)-\upsilon_2 (x) \mid\sum_{i=1}^{5} \mid \nabla (\upsilon_i(x) \mid \right), \quad x\in \Omega.
\end{eqnarray*}
Therefore, by passing $L^p$-norm, and using the continuous embedding $ W^{2\alpha,p}\hookrightarrow C^1 (\overline{\Omega}) $ again, we have
\begin{eqnarray*}
\| \nabla \cdot ( \varphi_2 V_{2}(\varphi)\nu)-   \nabla \cdot (\upsilon_2 V_{2}(\upsilon) \nu) 
\| & \leq & 
  |\Omega | M_{V_2}  \left( \left( 1+ \| \varphi \| \right) \| \nabla (\varphi_2  -   \upsilon_2 ) \|_{\infty} + \| \nabla \varphi \nu \| \| \varphi_2 -\upsilon_2 \|_{\infty}  \right) \nonumber \\
&&+|\Omega | M_{V_2} \left(\| \varphi -\upsilon\| \| \nabla \upsilon_2  \nu \|_{\infty} + 
\| \upsilon_2 \|_{\infty} \| \nabla (\varphi -\upsilon ) \nu \|\right) .
\end{eqnarray*}
This leads to 
\begin{eqnarray}
\| \nabla \cdot ( \varphi_2 V_{2}(\varphi)\nu)-   \nabla \cdot (\upsilon_2 V_{2}(\upsilon) \nu  \| & \leq &  |\Omega | L^4_R \|\varphi - \upsilon \|_{\alpha}
  \label{Lemma1 Formu5}
\end{eqnarray}
Similarly, we obtain that
and
\begin{eqnarray*}
\| \nabla \cdot  \varphi_3 V_{3}(\varphi)\nu-   \nabla \cdot \upsilon_3 V_{3}(\upsilon) \nu 
\| & \leq & 
  |\Omega |M_{V_3}  \left( \left( 1+ \| \varphi \| \right) \| \nabla (\varphi_3  -  \upsilon_3) \|_{\infty} + \| \nabla \varphi \nu \| \| \varphi_3 -\upsilon_3 \|_{\infty}  \right) \nonumber \\
&&+ |\Omega | M_{V_3} \left(\| \varphi -\upsilon\| \| \nabla \upsilon_3  \nu \|_{\infty} + 
\| \upsilon_3 \|_{\infty} \| \nabla (\varphi -\upsilon ) \nu \| \right),
\end{eqnarray*}
where $M_{V_3}= V_{3,max} |\nu |_{\infty}  $, and
\begin{eqnarray}
\| \nabla \cdot  \varphi_3 V_{3}(\varphi)\nu-   \nabla \cdot \upsilon_3 V_{3}(\upsilon) \nu 
\| & \leq & |\Omega | L^5_R \| \varphi -\upsilon\|_{\alpha} . \label{Lemma1 Formu7}
\end{eqnarray}
Now, we show that, there exists $L^9_{R}\geq 0$ such that 
$$\| \mathcal{M}\varphi -\mathcal{M}\upsilon \|_{\partial X} \leq L^{9}_{R} \|\varphi - \upsilon\|_{\alpha} .$$
To obtain that, we use the continuous embedding $ W^{1,p}( \Omega)^{5} \hookrightarrow  X_{\alpha} $ and the trace theorem, see \cite{Brezis}.
Indeed, first, we have,
$$ \mid \varphi_2 V_{2}(\varphi)(x)-   \upsilon_2 V_{2}(\upsilon)(x)\mid \leq V_{2,max}\left( ( 1+\mid \tilde{\varphi}(x)\mid) \mid \varphi_2 (x) -\upsilon_2 (x)\mid +\mid\upsilon_2 (x)\mid \sum_{i=1}^{5}\mid \varphi_i (x) -\upsilon_i (x)\mid\right), \;  x\in \Omega $$
Then, using the corresponding norms, we have
$$ | \varphi_2 V_{2}(\varphi)-   \upsilon_2  V_{2}(\upsilon)|_{p} \leq | \Omega | V_{2,max}\left( ( 1+\| \varphi\|) \| \varphi_2  -\upsilon_2 \|_{\infty} +\|\upsilon_2 \|_{\infty}\| \varphi -\upsilon\|\right) , $$
That is using the embedding $ X_{\alpha} \hookrightarrow C^{1}(\overline{\Omega})^{5} $, we obtain that

\begin{equation}
 | \varphi_2 V_{2}(\varphi)-   \upsilon_2  V_{2}(\upsilon)|_{p} \leq L^{6}_{R} \| \varphi -\upsilon\|_{\alpha}. \label{A.7}
\end{equation}
Arguing similarly, we obtain that
\begin{equation}
 | \varphi_3 V_{3}(\varphi)-   \upsilon_3 V_{3}(\upsilon)|_{p} \leq L^{7}_{R} \| \varphi -\upsilon\|_{\alpha}). \label{A.8}
\end{equation}
In otherwise, by estimating, this time the gradient of the corresponding terms, we obtain that 
\begin{equation}
 |  \nabla\varphi_2 V_{2}(\varphi)-    \nabla \upsilon_2  V_{2}(\upsilon)|_{p} \leq L^{8}_{R} \| \varphi -\upsilon\|_{\alpha}, \label{A.9}
\end{equation}
and
\begin{equation} | \nabla  \varphi_3 V_{3}(\varphi)-    \nabla \upsilon_3 V_{3}(\upsilon)|_{p} \leq L^{9}_{R} \| \varphi -\upsilon\|_{\alpha}. \label{A.10}
\end{equation}
Hence, using the estimates  \eqref{A.7}-\eqref{A.10}, we obtain $ L^{9}_{R} \geq 0$ depending only on $R$ and the norm of the trace operator, such that
$$\| \mathcal{M}\varphi -\mathcal{M}\upsilon \|_{\partial X} \leq L^{9}_{R} \|\varphi - \upsilon\|_{\alpha} .$$
Therefore, since $(\omega-\mathcal{A}_{\beta-1}) \mathcal{D}\in L(\partial X , X_{\beta-1})$, we obtain that,
\begin{eqnarray}
\nonumber \| (\omega-\mathcal{A}_{\beta-1}) \mathcal{D} (\mathcal{M}\varphi- \mathcal{M}\upsilon) \|_{ X_{\beta-1}}&\leq & c \| \mathcal{M}\varphi -\mathcal{M}\upsilon  \|_{\partial
 X} \\ &\leq & c L_{R}^{9} \| \varphi -\upsilon  \|_{\alpha}, \label{M estilmates}
\end{eqnarray} 
for some positive $c \geq \| (\omega-\mathcal{A}_{\beta-1}) \mathcal{D}\|_{\partial X \rightarrow X_{\beta -1} }$.
Furthermore, the functions $\gamma$ and $\phi$ (by construction, see \eqref{phi} and \eqref{gamma}) are respectively uniformly $ l_\gamma $-Lipschitz continuous and $ l_\phi $-Lipschitz continuous with respect to $t$ so that $0 \leq \gamma(t) , \phi(t) \leq 1$. Then, we have
\begin{eqnarray}
\| \gamma(t) \varphi_4 -\gamma(s) \upsilon_4 \| \leq |\Omega | l_\gamma (\|\varphi_4 \|_{\infty} \mid t-s \mid + \|\varphi_4 -\upsilon_4  \|_{\alpha}), \quad t,s \geq 0, \label{Lemma1 Formu8}
\end{eqnarray}
and
\begin{eqnarray}
\| \phi(t) \varphi_3 -\phi(s) \upsilon_3 \|  \leq |\Omega | l_\phi (\|\varphi_3 \|_{\alpha} \mid t-s \mid + \|\varphi_3 -\upsilon_3  \|_{\alpha}), \quad t,s \geq 0. \label{Lemma1 Formu9}
\end{eqnarray}
Thus,  from \eqref{Lemma1 Formu1}-\eqref{Lemma1 Formu7}, \eqref{Lemma1 Formu8},\eqref{Lemma1 Formu9} and using the embedding $ X\hookrightarrow  X_{\beta-1} $, we can find $L^{10}_R \geq 0 $ such that
\begin{align}
\| \mathcal{K}(t,\rho)-\mathcal{K}(s,\upsilon) \|_{\beta-1}
\leq L^{10}_R (\mid t-s \mid+ \|\rho-\upsilon \|_{\alpha}) \quad \text{ for all } t,s \geq 0. \label{K estimate}
\end{align}
Consequently, from \eqref{K estimate} and \eqref{M estilmates}, we have
\begin{align*}
\| \tilde{\mathcal{K}}(t,\rho)-\tilde{\mathcal{K}}(s,\upsilon) \|_{\beta-1} 
\leq L_R (\mid t-s \mid+ \|\rho-\upsilon \|_{\alpha}) \quad \text{ for all } t,s \geq 0 \text{ and } \varphi, \upsilon \in X_{\alpha}.
\end{align*}
This proves the result.\\

The following proof follows the lines of the proofs of \cite[Section 4.3]{Pazy}.\\
\noindent\textbf{Proof of Lemma \ref{Lemma regularity}.}
 Let us define
\begin{align*}
v(t)= \int_{0}^{t} \mathcal{T}_{\beta-1}(t-s) \mathcal{B}(t)  ds + \int_{0}^{t} \mathcal{T}_{\beta-1}(t-s) \left[ \mathcal{B}(s)- \mathcal{B}(t)\right]  ds :=  v_1 (t) +v_2 (t), \quad 0 \leq t \leq  T.
\end{align*}
It is clear that $v_1  \in C^1 ((0,T],X_{\beta}) $, this holds due to the analyticity of the semigroup $(\mathcal{T}_{\beta-1}(t))_{t\geq 0}$ in $X_{\beta-1}$, see Proposition \ref{Proposition:3.5}.
So it suffices to prove that $v_2(t) \in X_{\beta}:=D(\mathcal{A}_{\beta -1}) $ for all $t\in (0,T] $ and $\mathcal{A}_{\beta -1} v_2 (\cdot) $ is continuous on $[0,T]$. To this end, let $\varepsilon >0$ and consider
\begin{equation*}
            v^{\varepsilon}_2(t)= \left\{
                 \begin{aligned}
   &  \int_{0}^{t-\varepsilon} \mathcal{T}_{\beta-1}(t-s) \left[ \mathcal{B}(s)- \mathcal{B}(t)\right]  ds  &\text{ for }& t\geq \varepsilon \\
   & 0  &\text{ for }& t< \varepsilon.
                 \end{aligned}
                 \right.
\end{equation*}
So, the analyticity of the extension semigroup $ (\mathcal{T}_{\beta-1}(t))_{t\geq 0} $ on $X_{\beta-1}$ yields that 
$$ \mathcal{T}_{\beta-1}(t-s) \left[ \mathcal{B}(s)- \mathcal{B}(t)\right] \in D(\mathcal{A}_{\beta-1})=X_{\beta} \quad \text{ for } 0\leq s \leq t-\varepsilon.$$ 
Then,  $v^{\varepsilon}_2(t) \in D(\mathcal{A}_{\beta-1})$. Moreover, $ v^{\varepsilon}_2(t)  $ converges to $ v_2 (t)$ as $\varepsilon \rightarrow 0$. Since the operator $ \mathcal{A}_{\beta-1} $ is closed, to conclude, we only need to show that $ \mathcal{A}_{\beta-1} v^{\varepsilon}_{2} (t) =\displaystyle \int_{0}^{t-\varepsilon}  \mathcal{A}_{\beta-1}\mathcal{T}_{\beta-1}(t-s) \left[ \mathcal{B}(s)- \mathcal{B}(t)\right]  ds $ converges in $X_{\beta-1}$ as $\varepsilon \rightarrow 0$. That is, by the closedness of $\mathcal{A}_{\beta-1}$ we have
\begin{align*}
 \mathcal{A}_{\beta-1} v^{\varepsilon}_{2} (t)-\int_{0}^{t}  \mathcal{A}_{\beta-1}\mathcal{T}_{\beta-1}(t-s) \left[ \mathcal{B}(s)- \mathcal{B}(t)\right]  ds & = \int_{t-\varepsilon}^{t}  \mathcal{A}_{\beta-1}\mathcal{T}_{\beta-1}(t-s) \left[ \mathcal{B}(s)- \mathcal{B}(t)\right]  ds.
\end{align*}
Furthermore, the analytitcity of $ (\mathcal{T}_{\beta-1}(t))_{t\geq 0} $ yields that 
\begin{equation}
 \| \mathcal{A}_{\beta-1}\mathcal{T}_{\beta-1}(t)\|_{L(X_{\beta-1})} \leq l_0 \ t^{-1}, \quad t>0 \label{Lemma Holder conti Semig}
\end{equation}
for some $ l_0 \geq 0 $. Hence, using \eqref{Lemma Holder conti B}-\eqref{Lemma Holder conti Semig}, we conclude that 
$$ \| \mathcal{A}_{\beta-1} v^{\varepsilon}_{2} (t)-\int_{0}^{t}  \mathcal{A}_{\beta-1}\mathcal{T}_{\beta-1}(t-s) \left[ \mathcal{B}(s)- \mathcal{B}(t)\right]  ds \|_{\beta-1} \leq l l_0 \int_{0}^{\varepsilon} \sigma^{\eta-1} d \sigma \rightarrow 0 \text{ as } \varepsilon \rightarrow 0.$$
This proves that $ v_2(t) \in X_{\beta}$ and that $\mathcal{A}_{\beta-1} v_{2} (t)=\displaystyle\int_{0}^{t}  \mathcal{A}_{\beta-1}\mathcal{T}_{\beta-1}(t-s) \left[ \mathcal{B}(s)- \mathcal{B}(t)\right]  ds $ for $0<t \leq T$.
Moreover, $\mathcal{A}_{\beta-1}v(\cdot)$ is continuous in $(0,T]$ by construction. The continuity at $t=0$ can be obtained easily using the estimates \eqref{Lemma Holder conti B} and \eqref{Lemma Holder conti Semig}. This proves the result.
\section{Tables of the functions and the parameters of the APC model}\label{Annexe:B}
In the sequel, we briefly recall the functions and the parameters of the APC model \eqref{eq:main1System APC2 ODEs}.
\begin{table}[h!]

\caption{Functions in the APC model}

\label{tab:functions-systeme-ACP}

\centering

\begin{tabular}{l l}

\textbf{Functions}            &    \textbf{Notation}    \\

\hline

Beginning of the

catastrophe                    &        $\gamma(t)$        \\

Return to a daily behavior    &        $\phi(t)$    \\

Imitation functions        &        $\mathcal{F},\,\mathcal{G},\,\mathcal{H}$        \\

\end{tabular}

\end{table}

\begin{table}[h!]

\caption{Parameters of the APC model}

\label{tab:parameters-systeme-ACP}

\centering

\begin{tabular}{l l}

\textbf{Parameters}                            & \textbf{Notation}        \\

\hline

Intrinsic evolution from alert to control        &        $b_1$                     \\

Intrinsic evolution from alert to panic       &        $b_2$                     \\

Intrinsic evolution from control to alert       &        $b_3$                     \\

Intrinsic evolution from panic to alert         &        $b_4$                     \\

Intrinsic evolution from  panic to control    & $c_1$                    \\

Intrinsic evolution from control to panic        & $c_2$                    \\

Mortality rates for alert, panic and control populations respectively & $\delta_1$, $\delta_2$, $\delta_3$\\

Imitation from alert to control &$\alpha_{13}$\\


Imitation from alert to panic &$\alpha_{12}$\\


Imitation from panic to control &$\alpha_{23}$\\

Imitation from control to panic &$\alpha_{32}$\\





\hline

\end{tabular}

\end{table}
\end{appendices}
\vspace*{1cm}
\section*{ACKNOWLEDGMENT}
This work has been supported by the French government, through the National Research Agency (ANR) under the
Societal Challenge 9 “Freedom and security of Europe, its citizens and residents” with the reference number ANR-
17-CE39-0008, co-financed by French Defence Procurement Agency (DGA) and The General Secretariat for Defence and
National Security (SGDSN).

\bibliographystyle{plain}

\begin{thebibliography}{10}

\bibitem{Adams} R. A. Adams, \& J. J. Fournier, \emph{Sobolev spaces}. Elsevier, 2003.

\bibitem{Amann} H. Amann, \emph{Quasilinear Evolution Equations and Parabolic Systems},  Transactions of the American Mathematical Society, \textbf{293}(1), 191-227, (1986).

\bibitem{Amann2} H. Amann, \emph{ Parabolic evolution equations and nonlinear boundary conditions}, Journal of Differential Equations, \textbf{72}(2), 201-269, (1988).

\bibitem{Masmoudi}  J. Barreiro-Gomez, \& N. Masmoudi, {\em Differential games for crowd dynamics and applications}, Mathematical Models and Methods in Applied Sciences, (2023), 1--40.
\bibitem{Batc} C. K. Batchelor and G. K. Batchelor, \emph{An introduction to fluid dynamics}, Cambridge university press, 2000.

\bibitem{Batkai} A. Bátkai, B. Jacob, J. Voigt and J. Wintermayr, \emph{Perturbations of positive semigroups on $AM$-spaces}, Semigroup Forum \textbf{96}(2) (2018), 333--347.



\bibitem{Bellomo}
N. Bellomo, L. Gibelli, and N. Outada, \emph{On the interplay between behavioral dynamics and social interactions in human crowds}, Kinetic and Related Models, \textbf{12}, (2019), 397--409.


\bibitem{Bellomo2}
N. Bellomo, L. Gibelli, A. Quaini, \& A. Reali, \emph{Towards a mathematical theory of behavioral human crowds}, Mathematical Models and Methods in Applied Sciences, \textbf{32}(02) (2022), 321--358 .

\bibitem{Bellomo3} N. Bellomo, J. Liao, A. Quaini, L. Russo,  \& C. Siettos,  \emph{Human behavioral crowds: Review, critical analysis, and research perspectives}, Mathematical Models and Methods in Applied Sciences, \textbf{33}(8) (2023), 1611--1659.

\bibitem{BoKeLa}S. B\"ogli, J. B. Kennedy,  \& R. Lang, \emph{On the eigenvalues of the Robin Laplacian with a complex parameter}, Analysis and Mathematical Physics, \textbf{12}(1), (2022), 1--60.


\bibitem{Brezis} H. Brezis,  \emph{Functional analysis, Sobolev spaces and partial differential equations}, New York: Springer, \textbf{2}(3) 2011.


\bibitem{Buchmueller} S. Buchmueller and U. Weidmann, {\em Parameters of Pedestrians, Pedestrian Traffic and Walking Facilities}, Schriftenreihe des IVT. ETH Zurich, 2006.
%

\bibitem{Burger}  M. Burger \& J. F. Pietschmann, \emph{Flow characteristics in a crowded transport model}, Nonlinearity, \textbf{29}(11), (2016) 3528.

\bibitem{Burger2} M. Burger, I. Humpert, \& J. F. Pietschmann, \emph{On Fokker-Planck Equations with In-and Outflow of Mass}, Kinetic \& Related Models, \textbf{13}(2), 249 (2020).

\bibitem{ijbc}  G. Cantin, N. Verdiére, V. Lanza, M. Aziz-Alaoui, R. Charrier, C. Bertelle, D. Provitolo, and E. Dubos-Paillard, \emph{Mathematical modeling of human behaviors during catastrophic events: stability and bifurcations}, International Journal of Bifurcation and Chaos,  \textbf{26}(10) (2016), p. 1630025, 2016.

%


\bibitem{Cantin_Aziz}  G. Cantin, M. A. Aziz-Alaoui, \& N. Verdière, \emph{Large-time dynamics in complex networks of reaction–diffusion systems applied to a panic model}, IMA Journal of Applied Mathematics, \textbf{84}(5), (2019) 974--1000.


\bibitem{T.Dlotko} J. W. Cholewa, T. Dlotko, \& N. Chafee, \emph{Global attractors in abstract parabolic problems}, \textbf{(278)}, Cambridge University Press, 2000.

\bibitem{JMCoron} K. M. Coron, A. Keimer,  \& L. Pflug,  \emph{Nonlocal Transport Equations---Existence and Uniqueness of Solutions and Relation to the Corresponding Conservation Laws}, SIAM Journal on Mathematical Analysis, \textbf{52}(6), (2020) 5500--5532.

\bibitem{Coscia} V. Coscia, C. Canavesio, \emph{First-order macroscopic modeling of human crowd dynamics}, Math. Mod. Meth.
Appl. Sci. \textbf{18} (2008) 1217--1247.

\bibitem{CristianiPiccoliTosin}
E. Cristiani, B. Piccoli, \& A. Tosin, \emph{Multiscale modeling of pedestrian dynamics}, Springer.
, \textbf{12} 2014.

\bibitem{Crocq} L. Crocq,\emph{ Paniques Collectives} (Les). Odile Jacob, 2013. 

\bibitem{Davies}  E. B. Davies, \emph{Heat Kernels and Spectral Theory}, Cambridge Univ. Press, Cambridge 1989.




\bibitem{Desh} W. Desch, J. Milota and W. Schappacher, \emph{Least square control problems in nonreflexivec spaces}, Semigroup Forum, \textbf{62}, (2001) 337--357.

\bibitem{Pietschmann} M. Di Francesco, P. A. Markowich, J. F. Pietschmann, \& M. T. Wolfram, \emph{On the Hughes' model for pedestrian flow: The one-dimensional case}, Journal of Differential Equations, \textbf{250}(3), (2011) 1334--1362.



\bibitem{DiFrancesco} M. Di Francesco, S. Fagioli, \& E. Radici, \emph{Deterministic particle approximation for nonlocal transport equations with nonlinear mobility}, Journal of Differential Equations, \textbf{266}(5), (2019) 2830--2868.


%

\bibitem{Nagel}
K-J. Engel and R. Nagel, \emph{One-parameter semigroups for linear evolution equations}, Semigroup forum. Springer-Verlag, 2001.




\bibitem{Gomes}  S. N. Gomes, A. M. Stuart, \& M. T. Wolfram, \emph{Parameter estimation for macroscopic pedestrian dynamics models from microscopic data}, SIAM Journal on Applied Mathematics, \textbf{79}(4), (2019) 1475--1500.


%
\bibitem{Greiner} G. Greiner, \emph{Perturbing the boundary conditions of a generator}, Houston J. Math., 13, 213–229, (1987).
%






\bibitem{Hecht} F. Hecht, O. Pironneau, A. Le Hyaric, \& K. Ohtsuka,(2005). FreeFem++ manual.

\bibitem{Hermant} L. F. L. Hermant, \emph{Video data collection method for pedestrian movement variables \& development of a pedestrian
spatial parameters simulation model for railway station environments}, Ph.D. thesis, Stellenbosch University, 2012

\bibitem{Hirsch-Smith} M. W. Hirsch, \& H. Smith, \emph{Monotone dynamical systems}, In Handbook of differential equations: ordinary differential equations, North-Holland \textbf{(2)}, 239--357 (2006).












\bibitem{Hughes2002}
R. L. Hughes, \emph{A continuum theory for the flow of pedestrians}, Transportation Research Part B: Methodological, \textbf{36}(6):(2002) 507--535, 2002.
%
%



%

\bibitem{Magal-Ruan} P. Magal \& S. Ruan, \emph{Theory and applications of abstract semilinear Cauchy problems}, volume 201 of Applied Mathematical Sciences. Springer, Cham, 2018.

\bibitem{Marino} G. Marino, J-F. Pietschmann, \& A. Pichler,  \emph{Uncertainty Analysis for Drift-Diffusion Equations}, arXiv preprint arXiv:2105.06334 (2021).


\bibitem{Martin-Smith} R. H. Martin \& H. L. Smith,  \emph{Abstract functional-differential equations and reaction-diffusion systems}, Transactions of the American Mathematical Society, \textbf{321}(1), (1990) 1--44.

\bibitem{Martin} R. H. Martin, \emph{Nonlinear operators and differential equations in Banach spaces}, Krieger Publishing Co., Inc (1986).

\bibitem{Lanza_et_al.} V. Lanza,  E. Dubos-Paillard, R. Charrier, N. Verdière, D. Provitolo, O. Navarro, M.A.  Aziz-Alaoui, \emph{Spatio-Temporal Dynamics of Human Behaviors During Disasters: A Mathematical and Geographical Approach}, In Complex Systems, Smart Territories and Mobility Springer, Cham,  (2021) 201--218. 


\bibitem{LanzaFede} V. Lanza, D. Provitolo, N. Verdière, C. Bertelle, E. Dubos-Paillard, O. Navarro, R. Charrier, I. Mikiela, M. Aziz-Alaoui, A. Boudoukha, A. Tricot, A. Schleyer-Lindenmann, A. Berred, S. Haule, E. Tric,  \emph{Modeling and analysis of the impact of risk culture on human behavior during a catastrophic event}, Sustainability, \textbf{15}, 11063, (2023). 


\bibitem{Liu-Zh-Hu} Q. Liu, L. Lu, Y. Zhang, \& M. Hu, \emph{Modeling the dynamics of pedestrian evacuation in a complex environment}, Physica A: Statistical Mechanics and its Applications, \textbf{585}, 126426 (2022).

%
\bibitem{Lunardi} A. Lunardi, \emph{Analytic Semigroups and Optimal Regularity in Parabolic Problems}, Birkhauser, Basel, Boston, Berlin, 1995.

\bibitem{MauryBook} B. Maury \& S. Faure, \emph{Crowds in Equations: An Introduction to the Microscopic Modeling of Crowds}, World Scientific 2018.

\bibitem{Meyer} P. Meyer-Nieberg, \emph{Banach lattices}, Universitext. Springer-Verlag, Berlin, 1991.



%
%

\bibitem{Pazy}
A. Pazy, \emph{Semigroups of linear operators and applications to partial differential equations}, Springer Science \& Business Media, 2012.

\bibitem{Pi_Re} P. Pigeon \& J. Rebotier, Disaster Prevention Policies: A Challenging and Critical Outlook. Elsevier  2016.


\bibitem{Re_Ru} M. Reghezza-Zitt \& S. Rufat, \emph{Resilience imperative: Uncertainty, risks and disasters}, Elsevier 2015.

\bibitem{Rosini} M. D. Rosini, \emph{Nonclassical interactions portrait in a macroscopic pedestrian flow model}, Journal of Differential Equations, \textbf{246}(1), (2009) 408--427.

\bibitem{Schaefer} H. H. Schaefer, \emph{Banach lattices and positive operators}, Springer-Verlag, New York-Heidelberg, 1974. Die Grundlehren der mathematischen Wissenschaften, Band 215.

\bibitem{Sim_Lan_Hug} M. J. Simpson, K. A. Landman \&  B. D. Hughes, \emph{Multi-species simple exclusion processes}, Physica A: Statistical  Mechanics and its Applications, \textbf{388}(4), (2009) 399--406.

\bibitem{Seyfried2005}  A. Seyfried , B. Steffen , W. Klingsch , M. Boltes, \emph{The fundamental diagram of pedestrian movement revisited}, J. Stat. Mech. \textbf{10} (2005), p.10002.

\bibitem{Triebel} H. Triebel, \emph{Interpolation Theory, Function Spaces, Differential Operators}, North-Holland, Amsterdam, 1978.

\bibitem{Inria} M. Twarogowska, P. Goatin, \& R. Duvigneau, \emph{Numerical study of macroscopic pedestrian flow models} (Doctoral dissertation, INRIA) 2013.




\bibitem{Xia2009}Y. Xia, S. C. Wong, and Ch. W. Shu,  \emph{Dynamic continuum pedestrian flow model with
memory effect}, Phys. Rev. E,\textbf{79}:066113,  (2009).






\end{thebibliography}

\end{document}